\documentclass{amsart}
\usepackage{medstar}
\usepackage{tikz-cd}
\usepackage[unicode]{hyperref}
\usepackage{graphicx,color}
\usepackage{amssymb}
\usepackage{mathrsfs}
\usepackage{enumitem}
\usetikzlibrary{positioning}
\usepackage{pinlabel}
\usepackage{microtype}

\newtheorem{thmx}{Theorem}

\usepackage{macros_graphical}
\usepackage{sync_papers}

\title{A graphical category for higher modular operads}

\author[P. Hackney]{Philip Hackney}
\address{Department of Mathematics\\ University of Louisiana at Lafayette\\ Lafayette, LA 70504-3568 USA}
\email{philip@phck.net} 
\urladdr{http://phck.net}
\thanks{The first author acknowledges the support of Australian Research Council Discovery Project grant DP160101519.}

\author[M. Robertson]{Marcy Robertson}
\address{School of Mathematics and Statistics \\ The University of Melbourne \\ Melbourne, Victoria, Australia}\email{marcy.robertson@unimelb.edu.au}

\author[D. Yau]{Donald Yau}
\address{Department of Mathematics\\
The Ohio State University at Newark\\
Newark, OH \\ USA}
\email{dyau@math.ohio-state.edu}

\date{December 31, 2019}

\begin{document}

\begin{abstract}
We present a homotopy theory for a weak version of modular operads whose compositions and contractions are only defined up to homotopy.
This homotopy theory takes the form of a Quillen model structure on the collection of simplicial presheaves for a certain category of undirected graphs.
This new category of undirected graphs, denoted $\graphicalcat$, plays a similar role for modular operads that the dendroidal category $\Omega$ plays for operads.
We carefully study properties of $\graphicalcat$, including the existence of certain factorization systems.
Related structures, such as cyclic operads and stable modular operads, can be similarly treated using categories derived from $\graphicalcat$.
\end{abstract}

\maketitle

A modular operad, as introduced by Getzler and Kapranov \cite{MR1601666}, is a kind of cyclic operad equipped with self-compositions of operations.
That is, it is an algebraic structure consisting of a sequence of objects $P(n)$ (in a symmetric monoidal category), indexed by nonnegative integers $n$, together with families of `composition operations' $P(n) \otimes P(m) \to P(n+m-2)$ and `contraction operations' $P(n) \to P(n-2)$.
The canonical example is when $P(n)$ is the moduli space of Riemann surfaces with $n$ marked points (see \cite[II.5.6]{mss}). 
This paper, along with its companion \cite{modular_paper_two}, centers around a new category of graphs that permits a Segalic approach to the study of modular operads. 

The main goal of the present paper is to propose a precise definition for up-to-homotopy modular operads and provide a homotopy theory for such objects. 
Why might one pursue such a program?
One motivation comes from the work of Mann and Robalo who show, by passing to correspondences in derived stacks, that the operad of stable curves of genus zero acts on any smooth projective complex variety (see \cite[Theorem 1.1.2]{MannRobalo:BACGWT}).
This allows them to construct (in genus zero) Gromov--Witten invariants on the derived category of the variety in question. 
We anticipate that this work will be used in the study of higher genus version of that result; see Remark 1.2.1 of \cite{MannRobalo:BACGWT}, which is prefigured in the ordinary case in \cite[11.2]{Barannikov:MOBVG}. 

An additional motivation comes from work of the second author with Horel and Boavida de Brito on subgroups of the profinite Grothendieck--Teichm\"uller group. 
They conjecture that the homotopy automorphisms of a profinite completion of Getzler--Kapronov's modular operad $\overline{\mathcal{M}}_{*,*}$ \cite{MR1601666} will be isomorphic to the group $\Lambda$ which is a subgroup of the profinite Grothendieck--Teichm\"uller group which still contains the absolute Galois group \cite[Theorem A]{MR1752293}. 
The profinite completion of $\overline{\mathcal{M}}_{*,*}$ cannot be a strict modular operad as the profinite completion is often not product preserving \cite[Proposition 3.9]{BoavidaHorelRobertson:OGZCGTG}. 
The material in this paper and its companion \cite{modular_paper_two} provides the homotopy-theoretic framework required for this project. 

Around half of this paper is devoted to the introduction of a modular graphical category $\graphicalcat$ and a study of its properties.
The objects of this category are undirected, connected graphs with loose ends, while morphisms are given by `blowing up' vertices of the source into subgraphs of the target in a way that reflects iterated operations in a modular operad.
The category $\graphicalcat$ is actually a proper subcategory of the category of Feynman graphs studied by Joyal and Kock in \cite{JOYAL2011105}.
The restriction we make is partly to disallow `duplication of variables' from appearing in morphisms.\footnote{This follows the theory of algebras over operads, which generally can model types of algebras where each variable term appears exactly once in the defining equations.}
As in our earlier work on graph categories \cite{hrybook,hry_factorizations,hry_cyclic}, which made similar restrictions in other contexts, this bears fruit.
Namely, the weak factorization system that exists on the category of Joyal and Kock becomes an orthogonal factorization system on $\graphicalcat$, and, moreover, there is a generalized Reedy structure on $\graphicalcat$.
These two facts constitute the main theorems of the first half.

The heart of this paper is Section~\ref{section simplicial presheaves on U}, where we investigate the homotopy theory of simplicial $\graphicalcat$-presheaves.
Roughly, such a presheaf $X$ will be said to satisfy the \emph{Segal condition} if the value of $X$ at a graph $G$ is determined up to homotopy by the value of $X$ at each of the vertices of $G$.
If, additionally, the value of $X$ at an edge is contractible, then we say that $X$ is a \emph{Segal modular operad}.\footnote{This terminology is chosen to be consonant with `Segal operad' from \cite{bh}.}
\begin{thmx}\label{model structure intro}
The category of simplicial $\graphicalcat$-presheaves admits a model structure whose fibrant objects are the Segal modular operads.
\end{thmx}
If $X$ is a Segal modular operad, then after passing to the homotopy category of spaces one has an honest (unital, symmetric) modular operad.
Indeed, the two types of operations from the first paragraph can be found by working with those connected graphs that have precisely one internal edge.
On the other hand, there is a strict analogue of the Segal condition, and in the companion paper \cite{modular_paper_two} we prove the following nerve theorem.
It shows that the \emph{strict} Segal condition gives a characterization of (colored) modular operads.
\begin{thmx}\label{nerve theorem paper one intro}
Each graph freely generates a (colored) modular operad.
This process gives a functor $\graphicalcat \to \csm$ which induces a fully-faithful functor $N : \csm \to \Set^{\graphicalcat^{\oprm}}$.
The essential image of $N$ consists precisely of those presheaves which satisfy a strict Segal condition.
\end{thmx}

An attentive reader may have noticed that there is no notion of \emph{genus} for operations in the modular operads discussed above.
In the original definition \cite{MR1601666} of modular operad, the graded objects $P = \{P(n)\}_{n\geq 0}$ had an additional genus grading as $P(n) = \{P(n,g)\}_{g > 1 - \frac{n}{2}}$.
The composition operations are to be interpreted as additive on genus, while the contraction operations increase genus by one.
Of course we also have applications in mind where it is beneficial to keep track of this geometric information, so we provide a variant of $\graphicalcat$ whose objects are \emph{stable graphs}.
There are analogues of Theorem~\ref{model structure intro} and Theorem~\ref{nerve theorem paper one intro} for the category of stable graphs.

If $F : \Rr \to \Ss$ is a functor between small categories and $\modelcat$ is a bicomplete category, then there is a restriction functor $F^* : \modelcat^{\Ss} \to \modelcat^{\Rr}$ that has both adjoints (given by left and right Kan extension).
Suppose further that these diagram categories have model category structures.
As $F^*$ is both a left adjoint and a right adjoint, one would like to know whether or not $F^*$ is left Quillen, right Quillen, or neither.
For example, if $\modelcat$ is itself a model category and the two diagram categories both have the injective model structure (with cofibrations and weak equivalences defined to be levelwise), then it is immediate that $F^*$ is left Quillen.
In \cite{Barwick:OLRMCLRBL,hv15}, this question was considered when $F$ is a Reedy functor between (strict) Reedy categories.
Barwick classified those Reedy functors $F$ so that $F^*$ is left Quillen (resp.~right Quillen) for \emph{every} model category $\modelcat$.
A natural question is whether this classification can be adapted to the setting of generalized Reedy categories \cite{bm_reedy}.

In Section~\ref{section simply connected} we show that Barwick's characterization does not extend in the obvious way to the setting of generalized Reedy categories, by means of an explicit counterexample.
Our observation arose out of a careful comparison of the present paper to \cite{hry_cyclic}.
In that paper, we introduced a category $\Xi$ of undirected trees for the purposes of studying higher cyclic operads.
On the other hand, one could consider the subcategory $\graphicalcat_{\cycrm}$ of $\graphicalcat$ on the simply connected graphs.
There is a functor $\graphicalcat_{\cycrm} \to \Xi$, but it is not an equivalence and constitutes our counterexample.
The key difference between the two categories is the type of colored cyclic operads they can be used to model. 
The color sets of cyclic operads in \cite{hry_cyclic} do not have any additional structure, whereas in other settings \cite{MR3189430,DrummondColeHackney:DKHCO,Shulman:2CDCPMA} the color sets will come with an involution. 
Indeed, in \cite{modular_paper_two}, our modular operads have involutive color sets, so one should expect $\graphicalcat_{\cycrm}$ to model cyclic operads with involution.

\subsection*{Further directions and related work}
In this paper and its companion, we focus on the Segal condition for understanding (higher) modular operads.
There is another natural possibility, which is to consider the inner Kan condition for $\graphicalcat$-presheaves. 
This is currently being investigated by Michelle Strumilla as part of her PhD thesis.

Recently, it was shown by Ward that the operad governing modular operads is Koszul \cite[Theorem 3.10]{Ward:MPGH} (in the setting of groupoid-colored operads).
This important result opens the door to effective treatments of strongly homotopy modular operads, in the style of the strongly homotopy operads of van der Laan \cite{VanderLaan:CKDSHO}.
It is natural to ask about the relation of strongly homotopy modular operads to higher modular operads based on the category $\graphicalcat$.
In particular, it would be interesting to know if there is an analogue of Le Grignou's result \cite{MR3626563} (generalizing earlier work of Faonte \cite{MR3607208} in the context of $A_\infty$-categories) which says that every strictly unital, strongly homotopy (colored) operad yields an inner Kan dendroidal set \cite[\S7]{mw}.

\subsection*{Acknowledgments}
We thank Joachim Kock for influential conversations, and thank Sophie Raynor for carefully explaining her work to us.
We also thank Pedro Boavida de Brito, Daniel Davis, Gabriel C. Drummond-Cole, David Gepner, Geoffroy Horel, Marco Robalo, and members of the Centre of Australian Category Theory for helpful discussions, interest, and encouragement.

\section{Graphs and the category \texorpdfstring{$\graphicalcat$}{U}} \label{section:graphical category}

All graphs in this paper are undirected and are allowed to have `loose ends,' that is, it is not necessary for both ends (or either end) of an edge to touch a vertex. 
One possible concise definition for such a graph (compare \cite[\S 2]{MR1113284}) is a pair $(X,V)$ where $X$ is a space (more precisely, a locally finite, one-dimensional CW complex), $V$ is a finite set of points of $X$, and $X\setminus V$ is a one-manifold (without boundary) having only a finite set of connected components. 
Components of $X\setminus V$ are the edges of the graph, and elements of $V$ are the vertices.
Thus we may have loops divorced from any vertex (those components of $X\setminus V$ homeomorphic to $S^1$), edges loose at one end (those with one missing limit point in $X$), and free floating edges (components of $X$ homeomorphic to $(0,1)$ which contain no vertices).

\[
\includegraphics[scale=.5]{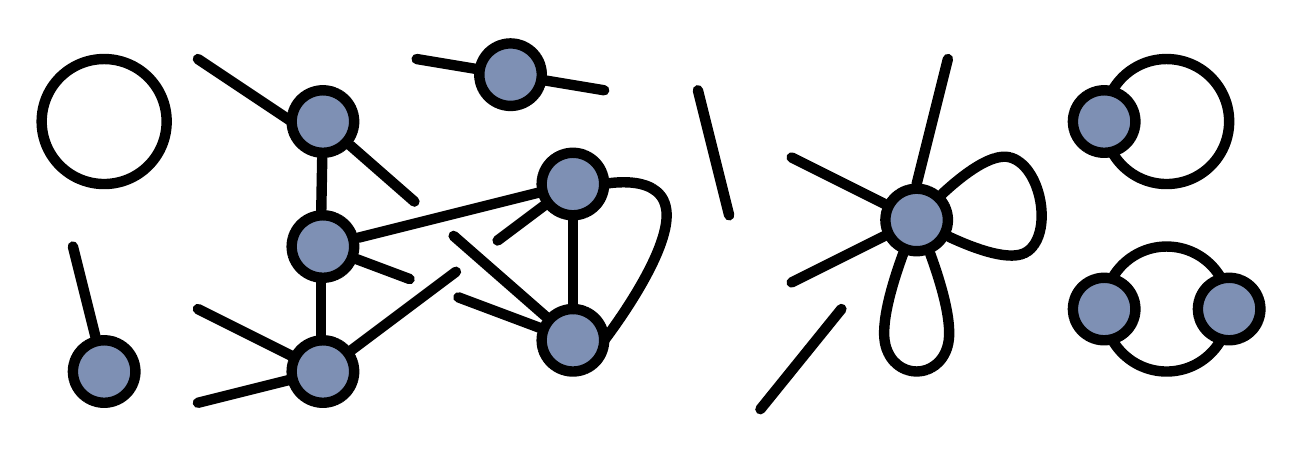}
\]

We now give some basic definitions. 
An \emph{arc} of a graph is an edge together with a chosen orientation.
Thus for any graph there a set $A$ of arcs, which comes equipped with a free involution $i$ given by reversing orientation. 
There is a partially-defined function $t: A \nrightarrow V$ which takes an arc to the vertex it points towards (that is, if $f : (0,1) \to X$ parametrizes the oriented edge $a \not\approx S^1$, then $t(a) = \lim_{x \to 1} f(x)$); we write $D \subseteq A$ for the domain of $t$. 
For each vertex $v\in V$, there is a corresponding neighborhood $\nbhd(v) = t^{-1}(v) \subseteq D \subseteq A$\index{$\nbhd(v)$} consisting of arcs which point towards $v$.

\begin{remark}\label{remark nondetermination boundary}
It is important to note that knowledge of $A$, $V$, $i$, and $t$ does not allow us to reconstruct the original graph. 
The only point of ambiguity is when $a$ and $ia$ are both in the complement of the domain of $t$; we cannot tell if the associated edge is $S^1$ or $(0,1)$.
One way to account for this difference is by considering the \emph{boundary}
$\eth(G)\subseteq A \setminus D$ of the graph.
Concretely, $\eth(G)$ may be identified with the set of ends (as in Definition 1 of \cite{HughsRanicki:EC}) of $X$; in the one-dimensional setting one can consider a free compactification $\bar X$ of $X$ and then we have $\eth(G)$ is in bijection with the discrete space $\bar X \setminus X$.
Abstractly, the boundary $\eth(G)$ is a subset of the complement of $D$ so that $i\eth(G) \subseteq D \amalg \eth(G)$ and $iD \setminus D \subseteq \eth(G)$.
\end{remark}

\begin{example}\label{examples geometric}
Let us give geometric descriptions of several important graphs.
\begin{itemize}
\item The \emph{exceptional edge}, denoted $\exceptionaledge$, is the graph where $X$ is the open interval $(0,1)$ and $V = \varnothing$.
\item The \emph{nodeless loop} is the graph where $X$ is the circle $S^1$ and $V = \varnothing$.
\item 
Let $n \geq 0$ be an integer.
The $n$-star $\medstar_n$ is the graph $(X,V)$ where $V = \{0 \} \subseteq \mathbb{C}$ and 
\[ X = V \cup \{ r e^{p2\pi i} \mid 0\leq r < 1 \text{ and $pn$ is an integer with } 1 \leq pn \leq n\} \subseteq \mathbb{C}. \]
Equipping $X$ with the usual topology, we then have that $X\setminus V$ is homeomorphic to $n$ copies of $(0,1)$. 
\item 
Let $n\geq 0$ be an integer.
The linear graph $L_n$ has $X = (0,1)$ and 
\[
	V = \left\{ \frac{1}{n+1} , \frac{2}{n+1}, \dots, \frac{n}{n+1} \right\}.
\]
In particular, $L_0$ is the exceptional edge.
\item 
Let $n\geq 0$ be an integer.
There is a graph with $X = S^1$ and 
\[
V = \{ e^{p2\pi i} \mid pn\in \mathbb{Z} \text{ and } 1 \leq pn \leq n\} \subseteq S^1;
\]
the case $n=0$ recovers the nodeless loop.

\end{itemize}
\end{example}

Henceforth, we will use combinatorial definitions of graphs, of which there are several competing definitions \cite{batanin-berger,JOYAL2011105,MR1113284,yj15}.
In this paper we primarily use Definition~\ref{definition jk graphs}, due to Joyal and Kock, as it is both extremely simple and also allows us to express the notion of \'etale map (Definition~\ref{definition etale}).
In light of Remark~\ref{remark nondetermination boundary}, we see that this definition does not capture those graphs where $X\setminus V$ has some $S^1$ components (we will return to this issue in \S\ref{subsection: egc}).
With that provisio, all of these combinatorial graph definitions are equivalent (Proposition 15.2, Proposition 15.6, and Proposition 15.8 of \cite{batanin-berger}), so we may use them interchangeably.

\begin{definition}[Feynman graphs \cite{JOYAL2011105}]\label{definition jk graphs}
A \emph{graph} $G$ is a diagram of finite sets\index{$A=A(G)$}\index{$D=D(G)$}\index{$V=V(G)$}\index{$s: D \to A$}\index{$i : A \to A$}
\[
\begin{tikzcd}
	A \arrow[loop left, "i"] & \lar[swap]{s} D \rar{t} & V
\end{tikzcd}
\]
where $i$ is a fixedpoint-free involution and $s$ is a monomorphism.\footnote{To ensure that we have a \emph{set} of graphs, insist that all of the sets $A, D, V$ are taken to be subsets of some fixed infinite set.} 
We will nearly always consider $D$ as a subset of $A$, and suppress the natural inclusion function $s: D \subseteq A$ from the notation.
\begin{enumerate}
	\item The \emph{boundary} of such a graph is the set $\eth(G) = A \setminus sD$\index{$\eth(G)$}.
	\item An \emph{edge} is just an $i$-orbit $[a,ia]$, and we write $E = A / i$ for the set of edges.
	\item An \emph{internal edge} is an edge of the form $[sd,sd']$ where $d,d'\in D$.
\end{enumerate}

\end{definition}

Let us translate the geometric descriptions from Example~\ref{examples geometric} into this setting.

\begin{example}\label{examples combinatorial}
If $Z$ is a set, write $2Z$\index{$2Z$} for the set
\[
	\{ z, z^\dagger \mid z\in Z\} \cong Z \amalg Z
\]
together with the evident involution.
We consider $Z$ as a subset of $2Z$, and write $Z^\dagger$\index{$Z^\dagger$} for its complement.
\begin{itemize}
\item 
The exceptional edge, $\exceptionaledge$\index{$\exceptionaledge$}, is the graph with $A = 2\{*\}$ and $V=D=\varnothing$.
As this graph is so important in what follows, we will give special names to its arcs and write $A = \eearcs$.\index{$\eearcs$}
\item 
The nodeless loop is not expressable in the Feynman graph formalism, as we would want to take $A = 2\{*\}$, $V=D=\varnothing$, but have an empty boundary. 
We will return to the nodeless loop in Definition~\ref{def C zero}. 
\item
The $n$-star $\medstar_n$\index{$\medstar_n$} has $V$ a one-point set, $D = \{ 1, \dots, n\}$, and $A = 2D$.
The function $s : D \to A = 2D$ is just the subset inclusion.
See Figure~\ref{figure medstar five}.
\item
The linear graph $L_n$\index{$L_n$} has $A = 2\{0,\dots, n\}$, $D= A \setminus \{0, n^\dagger\}$, $V = \{1,\dots, n\}$, $t(k) = k$ for $1\leq k \leq n$, and $t(k^\dagger) = k+1$ for $0\leq k \leq n-1$.
Each vertex neighborhood is of the form $\nbhd(k+1) = \{ k^\dagger, k+1 \}$ and the boundary is $\eth(L_n) = \{ 0, n^\dagger\}$.
\item
For the loop with $n$ vertices, we must suppose that $n>0$.
Let $V= \{1, \dots, n\}$ and $A = D = 2V$.
The target function is given by $t(k) = k$ for $1\leq k \leq n$, $t(k^\dagger) = k+1$ for $1\leq k \leq n-1$, and $t(n^\dagger) = 1$.
For vertex neighborhoods, we have $\nbhd(1) = \{ n^\dagger, 1\}$ and otherwise $\nbhd(k+1) = \{ k^\dagger, k+1 \}$.
\end{itemize}
\end{example}

\begin{figure}
\begin{tikzpicture}
\node [plain] (v) {$v$};
\node [above=.5cm of v] () {$\medstar_5$};
\draw [-] (v) to +(-.5cm,-.6cm); 
\draw [-] (v) to +(.5cm,-.6cm);
\draw [-] (v) to +(0cm,-.7cm);
\draw [-] (v) to +(-.5cm,.6cm); 
\draw [-] (v) to +(.5cm,.6cm);

\node [plain, right=3 cm of v ] (w) {$v$};
\node [above=.5cm of w] () {$\nbhd(v)$};
\draw [<-] (w) to node[swap, outer sep=-2pt]{\scriptsize{$5$}} +(-.5cm,-.6cm); 
\draw [<-] (w) to node[ outer sep=-2pt]{\scriptsize{$3$}} +(.5cm,-.6cm);
\draw [<-] (w) to node[ outer sep=-2pt]{\scriptsize{$4$}} +(0cm,-.7cm);
\draw [<-] (w) to node[ outer sep=-2pt]{\scriptsize{$1$}} +(-.5cm,.6cm); 
\draw [<-] (w) to node[swap, outer sep=-2pt]{\scriptsize{$2$}} +(.5cm,.6cm);

\end{tikzpicture}
\caption{The graph $\medstar_5$ with $\nbhd(v)=\{1,2,3,4,5\}$ and $\eth(\medstar_5)=\{1^{\dagger},2^{\dagger},3^{\dagger},4^{\dagger},5^{\dagger}\}.$ }\label{figure medstar five}

\end{figure}

It is especially convenient to have, for each graph $G$, a collection of stars.
Each such star is isomorphic to a $\medstar_n$ from the previous definition.
See Figure~\ref{figure stars examples} for a concrete example. 

\begin{definition}[Stars associated to graphs]\label{definition vertex star}
Suppose that $G$ is a graph.
We tweak the definition of $\medstar_n$ from Example~\ref{examples combinatorial} as follows:
\begin{itemize}
\item 
Let $\medstar_G$\index{$\medstar_G$} be the one-vertex graph with $A = 2\eth(G)$ and $D= \eth(G)^\dagger$.
Notice that $\eth(\medstar_G) = A \setminus D = \eth(G)$ and that the neighborhood of the unique vertex is $D = \eth(G)^\dagger$.
\item Suppose that $v$ is a vertex of $G$ and let $\nbhd(v)$ be its neighborhood in $G$.
We let $\medstar_v$\index{$\medstar_v$} denote the graph with $V = \{ v \}$, $D = \nbhd(v)$, and $A = 2\nbhd(v)$.
The boundary of $\medstar_v$ is $\nbhd(v)^\dagger \subseteq 2\nbhd(v)$.
\end{itemize}
\end{definition}

\begin{figure}[htb]
\labellist
\small\hair 2pt
 \pinlabel {$1$} [ ] at 132 92
 \pinlabel {$2$} [ ] at 77 106
 \pinlabel {$3$} [ ] at 56 81
 \pinlabel {$4$} [ ] at 85 62
 \pinlabel {$5$} [ ] at 102 61
 \pinlabel {$6$} [ ] at 101 18
 \pinlabel {$7$} [ ] at 67 34
 \pinlabel {$8$} [ ] at 15 43
 \pinlabel {$G$} [B] at 85 0
 \pinlabel {$1$} [Bl] at 221 91
 \pinlabel {$1^\dagger$} [Bl] at 221 72
 \pinlabel {$8^\dagger$} [Bl] at 221 40
 \pinlabel {$8$} [Bl] at 221 21
 \pinlabel {$\medstar_G$} [B] at 224 0
 \pinlabel {$2^\dagger$} [Bl] at 284 91
 \pinlabel {$2$} [Bl] at 284 72
 \pinlabel {$3$} [Bl] at 284 40
 \pinlabel {$3^\dagger$} [Bl] at 284 21
 \pinlabel {$\medstar_w$} [B] at 287 0
 \pinlabel {$4$} [ ] at 360 77
 \pinlabel {$4^\dagger$} [ ] at 372 91
 \pinlabel {$5$} [ ] at 369 56
 \pinlabel {$5^\dagger$} [ ] at 386 43
 \pinlabel {$6$} [ ] at 348 45
 \pinlabel {$6^\dagger$} [ ] at 336 31
 \pinlabel {$7$} [ ] at 338 66
 \pinlabel {$7^\dagger$} [ ] at 338 91
 \pinlabel {$\medstar_v$} [B] at 358 0
 \pinlabel {$w$} [ ] at 53 102
 \pinlabel {$v$} [ ] at 89 37
\endlabellist
\centering
\includegraphics[scale=0.7]{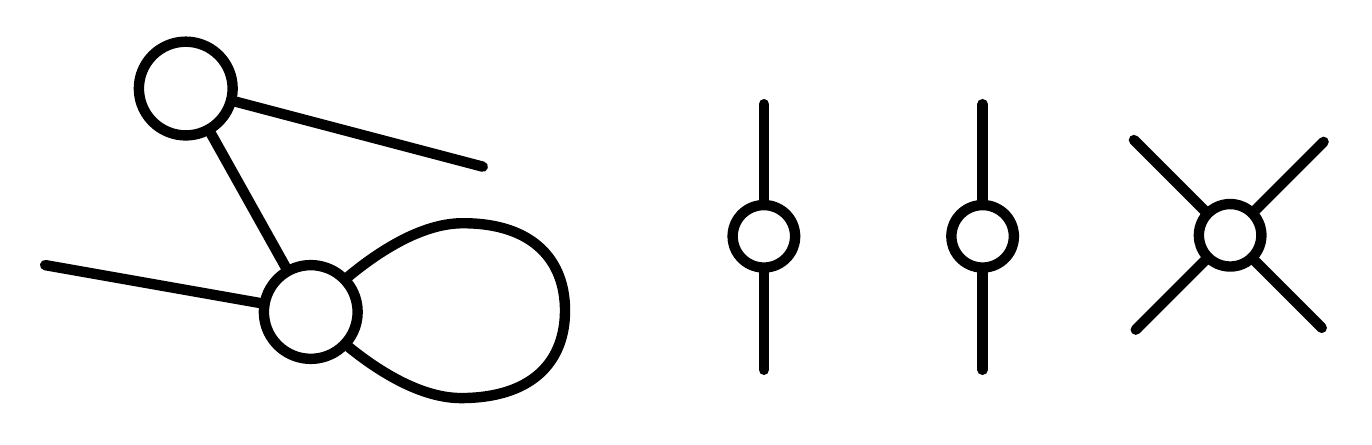}
\caption{An example of the stars associated to a graph $G$, where the involution on $A(G)$ interchanges $2n$ and $2n-1$ for $n=1,2,3,4$.}
\label{figure stars examples}
\end{figure}

\subsection{\'Etale maps as natural transformations}

\begin{definition}
\label{definition I and natural trans}
Let $\mathscr{I}$\index{$\mathscr{I}$} denote the category with three objects and three generating arrows, of shape $\begin{tikzcd}[column sep=small] \bullet \arrow[loop left] & \lar \bullet \rar & \bullet. \end{tikzcd}$
Each graph $G$ from Definition~\ref{definition jk graphs} is a functor from $\mathscr{I}$ into finite sets so that the leftward arrow is sent to a monomorphism and the generating endomorphism is sent to a free involution.
As we are considering graphs as functors, there is an obvious notion of graph map: a natural transformation of functors.
\end{definition}

Feynman graphs thus span a full subcategory of $\finset^{\mathscr{I}}$, and we will consider two subcategories (Definition~\ref{definition etale} and Definition~\ref{def: embedding}) in this subsection.
Our ultimate graph morphisms, given in Definition~\ref{def: graphical map}, will not be morphisms in the functor category $\finset^{\mathscr{I}}$.
Indeed, in that definition vertices are not sent to vertices, but rather to `subgraphs' of the codomain.
Nevertheless, by viewing graphs as functors, we can make the following definition of connectedness (which coincides with usual topological connecteness of an associated geometric version of the graphs).
\begin{definition}
\label{definition connected}
A graph $G$ is \emph{connected} if it cannot be written as a nontrivial coproduct in the functor category $\finset^{\mathscr{I}}$.
\end{definition}

\begin{construction}[The graph $G\setminus X$]\label{construction: subtracting edges}
	Suppose that $G$ is a graph and $X \subseteq E$ is a set of edges.
	Recall that in the current formalism, an edge $e$ is an orbit of the involution $i$ on $A$.
	We form a new graph $G'$ as follows, which in the future we will denote by $G\setminus X$.
	The set of vertices of $G'$ coincides with the set of vertices of $G$.
	The set of arcs of $G'$, denoted $A'$, is the $i$-closed subset of $A$ consisting of those arcs that do not appear in any edge $e$ in the set $X$.
	In symbols,
	\[
	A' = \left\{ a \in A \mid \forall e \in X, a \notin e \right\} \subseteq A.
	\]
	The involution on $A'$ is just the restriction of that on $A$.
	Finally $D' = D\cap A'$.
\end{construction}

\begin{definition}\label{definition: edge subtraction}
	Suppose that $G$ is a graph.
	\begin{itemize}
		\item If $Y \subseteq A$ is a set of arcs, temporarily let $\bar Y \subseteq E$ be the collection of edges containing elements of $Y$. 
		That is, 
		\[
		\bar Y = \{ e \in E \mid \exists a \in Y \text{ with } a\in e \}.
		\]
		We define $G\setminus Y$ to be the graph $G \setminus \bar Y$ from Construction~\ref{construction: subtracting edges}.
		\item If $G$ is any graph with boundary $\eth(G) = A\setminus D$, write $\core(G)$\index{$\core(G)$} for the graph $G\setminus \eth(G)$.
	\end{itemize}
\end{definition}

\begin{example}\label{examples of monomorphisms of diagrams}
	Suppose that $G$ is any graph and $X\subseteq E$. 
	Then the inclusion $G\setminus X \to G$ is a natural transformation. 
	Similarly, we have natural transformations $G\setminus Y \to G$ (when $Y\subseteq A$) and $\core(G) \to G$ as well.
Often (specifically, when $i\eth(G) \cap \eth(G) = \varnothing$), the latter map fits into a square expressing $G$ as a pushout of graphs (in the functor category $\finset^{\mathscr{I}}$); see Construction~\ref{construction core pushout}. 
\end{example}

The following definition is also due to Joyal and Kock \cite{JOYAL2011105}.

\begin{definition}[\'Etale maps]\label{definition etale}
A natural transformation $G\to G'$ is said to be \emph{\'etale} if the right-hand square of
\begin{equation*}
\begin{tikzcd}
	A \arrow[loop left, "i"]\dar & \lar[swap]{s} D\dar \rar{t} & V\dar \\
	A' \arrow[loop left, "i'"] & \lar[swap]{s'} D' \rar{t'} & V'
\end{tikzcd}
\end{equation*}
is a pullback.
\end{definition}

\begin{example} 
Let us give some simple examples of maps which are not \'etale.
\begin{itemize}
\item Consider the map $\core(\medstar_n) \to \medstar_n$ described in Example~\ref{examples of monomorphisms of diagrams}.  
This becomes the map $\medstar_0 \to \medstar_n$ \[
	\begin{tikzcd}
	\varnothing \dar & \lar \varnothing \dar \rar{t} & \{v\}\dar \\
\{1,\dots,n, 1^{\dagger},\dots,n^{\dagger}\}  & \lar \{1,\dots, n\} \rar{t'} & \{v\}
	\end{tikzcd}\] 
which is not \'etale unless $n=0$.
\item More generally, suppose that $G$ is a connected graph. Then $\core(G) \to G$ is \'etale if and only if $G$ is the exceptional edge or $\eth(G)$ is empty. If $G$ is not connected, then $\core(G) \to G$ splits as a sum over the set of connected components of $G$, and is thus \'etale if and only if each summand is \'etale.
\item Let $L_n$ denote the linear graph with $n$ vertices. 
If $n >0$, then there are no \'etale maps from $L_n$ to a graph with no vertices. 
\end{itemize}
\end{example}

\begin{definition}[Embeddings]\label{def: embedding}
Suppose that $G$ and $G'$ are connected graphs.
An \emph{embedding} $G\to G'$ is an \'etale map (Definition~\ref{definition etale}) where $V\to V'$ is a monomorphism.
If $G'$ is a connected graph, write $\bigembeddings(G')$\index{$\bigembeddings(G)$} for the collection of embeddings $G \to G'$ (in particular, $G$ is also connected).
\end{definition}

Since an embedding is, in particular, \'etale, the function $D\to D'$ is also a monomorphism.  
It may be the case, however, that $A \to A'$ is not.
We will specify exactly when this can happen in Lemma~\ref{lemma embedding not monomorphism}.

\begin{example}
\label{example vertices as embeddings}
We can consider $V$ as a subset of $\bigembeddings(G)$. 
Indeed, for each vertex $v\in V$ we have the associated star $\medstar_v$ from Definition~\ref{definition vertex star} which has a single vertex and $2\nbhd(v) = \nbhd(v) \amalg \nbhd(v)^\dagger$ as its set of arcs.
We write
\[
	\iota_v : \medstar_v \to G
\]
\index{$\iota_v : \medstar_v \to G$}
for the \'etale map 
\[ \begin{tikzcd}
\nbhd(v) \amalg \nbhd(v)^\dagger  \dar & \nbhd(v) \lar[hook'] \rar  \dar[hook] & \{v\} \dar[hook] \\
A & D \lar \rar["t"] & V.
\end{tikzcd} \]
The left-hand map in this diagram is just the inclusion $\nbhd(v) \to D \to A$ on the first component, while the second component (which is forced by compatibility with the involutions) sends $a^\dagger$ to $ia$. 
As the right-hand map is a monomorphism, $\iota_v$ is an embedding.
\end{example}

\begin{example}[Contracted star] \label{corolla_subgraph_contracted_corolla}
Suppose that we take $\medstar_5$ from Figure~\ref{figure medstar five} and identify $4^\dagger$ with $5$ (and likewise $5^\dagger$ with $4$).
That is, we consider the graph $G$ with one vertex, set of arcs  $A=\{1, 2,  3, 4,  5, 1^{\dagger}, 2^{\dagger},  3^{\dagger}\}$ and $D = \{ 1, 2, 3, 4, 5\}$. 
The involution is the same as that for $\medstar_5$, except that $i(4) = 5$ (and $i(5) = 4$).
Thus there is one internal edge $e=[4,5]$.
The natural embedding $\iota_v : \medstar_5 \to G$ is not injective on arcs.
\[
\begin{tikzpicture}
\node [plain] (v) {$v$};
\node [above=.5cm of v] () {$\medstar_5$};
\draw [-] (v) to +(-.5cm,-.6cm); 
\draw [-] (v) to +(.5cm,-.6cm);
\draw [-] (v) to +(0cm,-.7cm);
\draw [-] (v) to +(-.5cm,.6cm); 
\draw [-] (v) to +(.5cm,.6cm);
\node[empty, right=.2cm of v] (p){};

\node[plain, right=3.5cm of v] (w) at (0,0) {$v$};
\node [above=.5cm of w] () {$G$};
\draw [-, out=70, in=-70, loop] (w) to node{\scriptsize{e}} ();
\draw [-] (w) to +(-.5cm,.6cm); 
\draw [-] (w) to +(-.5cm,-.6cm); 
\draw [-] (w) to +(0cm,-.7cm);
\node[empty, left=.2cm of w] (q){};
\draw [black, ->, line width=1pt, shorten >=0.03cm] (p) to (q);

\end{tikzpicture}
\]
\end{example}

In the previous example we showed how to connect up two boundary edges and still have an embedding.
This can be done in the reverse direction, namely by starting with a graph and cutting at some internal edge.
In Lemma~\ref{lemma embedding not monomorphism} we give a general statement about embeddings that are not monomorphisms, and we see that they always come from such internal edge cuttings.
The reader should contrast this example with Construction~\ref{construction: subtracting edges}, where edges are deleted entirely.

\begin{example}\label{snipping example}
We describe a family of embeddings obtained by ``snipping'' a single edge.
Let $G$ be a connected graph and let $e=[a,ia]$ be a chosen internal edge in $G$.
We let $G_e$ denote the graph obtained from $G$ by snipping $e$.
Explicitly, we define $G_e$ to have the same set of vertices as $G$ and set of arcs $A \amalg \{ b, c \}$.
The involution $i_e$ on $G_e$ is given by $i_e (a) = b$ and $i_e(ia) = c$, while the rest of the structure remains the same.
This results in one of two cases: either $G_e$ is connected or not.  

\begin{itemize}
\item[Case 1] If the graph $G_e$ is no longer connected, this means that $e$ was an edge between two distinct vertices $u$ and $v$, with $a\in\nbhd(u)$ and $ia\in\nbhd(v)$.
Moreover, there are no other edges connecting $u$ and $v$.
Snipping $e$ thus results in two embeddings $f_u:G_u\rightarrow G$ and $f_v:G_v\rightarrow G$ which include the connected graph $G_u$, the half containing $u$, (respectively, $G_v$) into $G$.
Note that the embedding $f_u:G_{u}\rightarrow G$ is injective on vertices and arcs; all additional graph structure on $G_u$ is that of $G$.
The same is true for $G_v$.  

\[
\begin{tikzpicture}

\node [plain] (v) {$u$};
\node [above=.5cm of v] () {$G_u$};
\draw [-] (v) to +(-.5cm,-.8cm); 
\draw [-] (v) to node[near start, outer sep=-2pt]{\scriptsize{$ia$}} node[near end, outer sep=-2pt]{\scriptsize{$$}} +(0 cm,.8cm);
\draw [-] (v) to +(0cm,-.9cm); 
\draw [-, out=40, in=-40, loop] (v) to node[swap, outer sep=-2pt]{\scriptsize{$$}}  ();

\node[right=.6 cm of v] (q) {};

\node [plain, right=3cm of v] (s2) {$v$};
\node [plain, below=.8cm of s2] (w) {$u$};
\node [above=.5cm of s2] () {$G$};
\draw [-] (w) to +(-.5cm,-.8cm); 
\draw [-] (w) to +(0 cm,.8cm);
\draw [-] (w) to +(0cm,-.9cm); 
\draw [-, out=40, in=-40, loop] (w) to node[swap, outer sep=-2pt]{\scriptsize{$$}}  ();

\draw [-] (s2) to node[near start, outer sep=-2pt]{\scriptsize{$a$}} node[near end, outer sep=-2pt]{\scriptsize{$ia$}} (w);

\draw [-] (s2) to  +(-.5cm,-.8cm); 
\draw [-] (s2) to +(.5cm,-.8cm);
\draw [-] (s2) to +(-.5cm,.8cm); 
\draw [-] (s2) to +(.5cm,.8cm);
\node[left=.5 cm of w] (x) {};

\draw [black, ->, line width=1pt, shorten >=0.03cm, bend right] (q) to (x);

\end{tikzpicture}
\]

\item[Case 2] In the second case, we still have a connected graph after snipping $e$.
This means that the edge $e$ was part of a cycle.
This creates one embedding $G_e\rightarrow G$ where $G_e$ has two more arcs than $G$.
Specifically, if $a\in\nbhd(u)$ and $ia\in\nbhd(v)$ (allowing for the possibility that $u=v$), then we add a pair of arcs which disconnect the edge.
In the picture below, $G$ contains the new arcs $b$ and $c$ and the embedding $f$ takes $b$ to $ia$ and $c$ to $a$. 
\[
\begin{tikzpicture}

\node [plain] (s1) {$v$};
\node [plain, below=.8cm of s1] (w1) {$w$};
\node[empty, right=.2 cm of v](q){};
\node[empty, below=.2 cm of q](r){};
\node[empty, right= 1.5 cm of r](s){};
\draw [black, ->, line width=1pt, shorten >=0.03cm] (r) to (s);

\draw [-, bend right =50] (s1) to  node[black, very near end, swap]{\scriptsize{$a'$}} node[black, very near start,swap]{\scriptsize{$ia'$}}  (w1);

\draw [-] (s1) to node[black, very near end]{\scriptsize{$b$}} node[black, very near start]{\scriptsize{$a$}} +(.5cm,.8cm);
\draw [-] (w1) to node[black, very near end, swap]{\scriptsize{$c$}} node[black, very near start,swap]{\scriptsize{$ia$}}+(.5cm,-.8cm);
\node [plain, right=3cm of s1] (s2) {$v$};
\node [plain, below=.8cm of s2] (w) {$w$};

\draw [-, bend right =50] (s2) to  node[black, very near end, swap]{\scriptsize{$ia'$}} node[black, very near start,swap]{\scriptsize{$a'$}}  (w);
\draw [-, bend left=50] (s2) to node[black, very near end]{\scriptsize{$ia$}} node[black, very near start]{\scriptsize{$a$}}  (w);
\end{tikzpicture}
\]

\end{itemize}
\end{example}

\subsection{Graph substitution}
A key construction when dealing with graphs with loose ends and operadic structures is that of graph substitution.
Suppose that we are given a graph $G$, a collection of graphs $H_v$ indexed by the vertices of $G$, and specified
bijections $i\nbhd(v) \cong \eth(H_v)$.
Then we can form a new graph $G\{H_v\}$\index{$G\{H_v\}$} by a process of \emph{graph substitution}, where we replace each vertex $v$ by the graph $H_v$, identifying the edges at the boundary of $H_v$ with the edges incident to the vertex $v$ in $G$.
A detailed treatment may be found in the combinatorial settings in \cite[Ch.~5]{yj15} and \cite[\S 13]{batanin-berger}.

Intuitively, one sees that there should be a canonical identification of the vertices of $G\{H_v\}$ with $\coprod_{v\in G} V(H_v)$ (using the shorthand $v\in G$ for $v\in V(G)$), that all internal edges in $H_v$ become internal edges in $G\{H_v\}$, and that $\eth(G\{H_v\}) \cong \eth(G)$. 
Theorem 5.32 and Lemma 5.31 in \cite{yj15} tell us that graph substitution is \emph{associative} and \emph{unital}.
Unitality means that 
\[
\medstar_G\{G\}\cong G \cong G\{\medstar_v\};
\]
here $\medstar_G$ and $\medstar_v$ are as in Definition~\ref{definition vertex star} and the bijections needed to define the graph substitutions are the identity on $\eth(G)$ and in the the bijections $i\nbhd(v) \to \nbhd(v)^\dagger$, respectively.
Associativity asserts that 
\[
\left[G\{H_v\}\right] \{I^v_u\} \cong G\left\{ H_v\{I^v_u\}\right\}
\]
where the $I^v_u$ are graphs indexed on $u \in V(H_v)$ and with bijections left implicit.

\begin{remark}
\label{remark jk not closed under graph sub}
The collection of graphs from Definition~\ref{definition jk graphs} is not closed under graph substitution operations. 
Indeed, if $G$ is the loop with one vertex from Example~\ref{examples combinatorial} and $\exceptionaledge$ is the exceptional edge, then $G\{ \exceptionaledge \}$ should be the nodeless loop.
As we will only be dealing with \emph{connected} graphs (see Definition~\ref{definition connected}) for the remainder of this paper, this example is the main one we need to worry about (since graph substitution can be done one vertex at a time).
In working with disconnected graphs in generality, the result of a graph substitution may have many nodeless loops even when the original graphs involved have none.
\end{remark}

For Proposition~\ref{proposition: image} and Lemma~\ref{lemma: image factorization}, it is helpful to have an explicit description of graph substitution for Feynman graphs.
The remainder of this section is a little more difficult than what surrounds it, so it is recommended that most readers skip ahead to Section~\ref{section embeddings and boundaries} for now, carrying with them the preceding intuitive discussion and referring back as needed.
The following description is inspired by \cite[\S1.5]{Kock:GHP}.

\begin{construction}[Graph substitution]
\label{construction graph sub}
Suppose that $G$ is a graph where each component of $G$ contains at least one vertex, and let $E_i$ be its set of internal edges (outside of this situation, we must modify the two coequalizers below, adding in to the middle terms those components of $G$ that lack vertices). 
For each edge $e\in E_i$, choose an ordering $e = [x_{e}^1, x_{e}^2]$ for the set of arcs comprising $e$. 
We can exhibit $G$ as a coequalizer (in the diagram category $\finset^{\mathscr{I}}$)
\[ \begin{tikzcd}
\coprod\limits_{e \in E_i} \exceptionaledge \rar[shift left, "\outeredge"] \rar[shift right, "\inneredge" swap] & \coprod\limits_{v\in V} \medstar_v \rar & G,
\end{tikzcd} \]
\index{$\outeredge,\inneredge$}
where the map on the right is $\coprod_v \iota_v$.
\begin{itemize}
	\item $\outeredge$ is the coproduct of maps $\exceptionaledge \to \medstar_{tx_e^1}$ with $\outeredge_e(\edgemajor) = (x_e^1)^\dagger \in \eth(\medstar_{tx_e^1})$ and $\outeredge_e(\edgeminor) = x_e^1 \in D(\medstar_{tx_e^1})$;
	\item $\inneredge$ is the coproduct of maps $\exceptionaledge \to \medstar_{tx_e^2}$ with $\inneredge_e(\edgeminor) = (x_e^2)^\dagger \in \eth(\medstar_{tx_e^2})$ and $\inneredge_e(\edgemajor) = x_e^2 \in D(\medstar_{tx_e^2})$.
\end{itemize}
Now suppose we are given graphs $H_v$ and isomorphisms $m_v$ from $i(\nbhd(v)) \subseteq A(G)$ to $\eth(H_v)$.
We then have induced maps $\tilde \outeredge$ and $\tilde \inneredge$, where
\begin{itemize}
	\item $\tilde \outeredge$ is the coproduct of maps $\tilde \outeredge_e : \exceptionaledge \to H_{tx_e^1}$ with $\tilde \outeredge_e(\edgemajor) = m_{tx_e^1}(ix_e^1) \in \eth(H_{tx_e^1})$, 
	\item $\tilde \inneredge$ is the coproduct of maps $\tilde \inneredge_e : \exceptionaledge \to H_{tx_e^2}$ with 
	 $\tilde \inneredge_e(\edgeminor) = m_{tx_e^2}(ix_e^2) \in \eth(H_{tx_e^2})$.
\end{itemize}
We can then form the coequalizer 
\[ \begin{tikzcd}
\coprod\limits_{e \in E_i} \exceptionaledge \rar[shift left, "\tilde \outeredge"] \rar[shift right, "\tilde \inneredge" swap] & \coprod\limits_{v\in V} H_v \rar["\pi"] & K.
\end{tikzcd} \]
One can check that this object $K\in \finset^{\mathscr{I}}$ is always graph in the sense of Definition~\ref{definition jk graphs}.
Since colimits in $\finset^{\mathscr{I}}$ are computed levelwise, it is immediate that $V(K) = \coprod_{v\in V(G)} V(H_v)$ and $D(K) = \coprod_{v\in V(G)} D(H_v)$.
Further, the graph substitution $G\{H_v\}$ is represented by $K$ as long as $G\{H_v\}$ \emph{can be represented by Feynman graphs}.
When all of the graphs $G$ and $H_v$ are connected, this is the case except when $G$ is a loop with $n$ vertices (Example~\ref{examples combinatorial}), and all of the $H_v$ are edges.
\end{construction}

To distinguish between the various involutions, we will write $i_v$ for the involution on the graph $H_v$.
Let us analyze some of the structure of $A(K)$ by studying the preimages of certain elements.
We have three situations whose behavior follows readily from the coequalizer description.

\begin{enumerate}[label={({\Alph*})},ref={\Alph*}]
\item  \label{graph sub enumerate A}
Suppose that $e$ is an internal edge of $G$ between vertices $v$ and $w$ (which may be equal). If $H_v$ and $H_w$ are not edges, then
\[
\pi^{-1} \pi (\tilde \outeredge_e(\edgemajor)) = \{ \tilde \outeredge_e(\edgemajor), \tilde \inneredge_e(\edgemajor) \} 
\,\text{ and }\, 
\pi^{-1} \pi (\tilde \outeredge_e(\edgeminor)) = \{ \tilde \outeredge_e(\edgeminor), \tilde \inneredge_e(\edgeminor) \}. 
\]
\item  \label{graph sub enumerate B}
If $[d,i_vd]$ is an internal edge of $H_v$, then $\pi^{-1} \pi (d) = \{ d \}$.
\item  \label{graph sub enumerate C}
If $x\in \eth(G) \cap i \nbhd(v)$ and $H_v$ is not an edge, then 
\[
\pi^{-1} \pi (m_vx) = \{ m_vx \} 
\,\text{ and }\, 
\pi^{-1} \pi (i_v(m_vx)) = \{ i_v(m_vx) \}. 
\]
\end{enumerate}

We now show how to recover a standard, intuitive fact about graph substitution in an elementary way from Construction~\ref{construction graph sub}.
While reading the proof, note that $\eth(G) \to \eth(K)$ is always defined and injective, even when $K$ does not represent $G\{H_v\}$.
Indeed, it is only in showing surjectivity that we must impose particular constraints on $G$ and $H_v$ (so that $G\{H_v\}$ is not a nodeless loop).

\begin{lemma}
\label{lemma boundaries}
The function which sends $x\in \eth(G)$ to $\pi (m_{tix} x) \in A(K)$ constitutes a bijection between $\eth(G)$ and $\eth(K)$.
\end{lemma}
\begin{proof}
This is easy to see in the case when there is a vertex $v_0$ such that $H_v = \medstar_v$ (and $m_v$ is the evident map) for $v \neq v_0$.
We prove this only in that case.
For the general case, one can either make an inductive argument from this case, introduce a variation on this proof involving paths, or appeal to something like Construction~\ref{construction xA} that we will need later. 
The reader who is familiar with \cite{yj15} will note that this lemma is essentially contained in the proof of Lemma 5.10 there. 

We first show that if $x\in \eth(G) \cap i \nbhd(u)$, then $\pi (m_u x)$ is in $\eth(K)$. 
If $\pi^{-1} \pi (m_u x)$ has a single element, then $\pi (m_u x)$ is in $\eth(K)$ since $D(K) \cong \coprod_v D(H_v)$.
Since $x$ is not part of an internal edge of $G$, we know that $m_ux$ is not of the form $\tilde \outeredge_e(\edgemajor)$ or $\tilde \inneredge_e(\edgeminor)$.
Suppose that the set $\pi^{-1} \pi (m_u x)$ has more than one element. 
Then there exists an internal edge $e$ of $G$ with $m_u(x)$ equal to $\tilde \outeredge_e (\edgeminor)$ or $\tilde \inneredge_e(\edgemajor)$; without loss of generality, we assume that we are in the former case. 
We then have that both $m_u x = \tilde \outeredge_e (\edgeminor)$ and $i_u(m_u x) = \tilde \outeredge_e (\edgemajor)$ are in $\eth(H_u)$, hence $H_u$ is an edge and $u = v_0$.
We now know that 
\begin{equation}\label{arc inclusions that should be equalities}
\begin{aligned}
	\{ m_ux, \tilde \inneredge_e(\edgeminor) \} &\subseteq \pi^{-1} \pi (m_u x) \\
	\{ i_u(m_ux), \tilde \inneredge_e(\edgemajor) \} &\subseteq \pi^{-1} \pi (i_u(m_u x)),
\end{aligned}
\end{equation}
and furthermore the sets $\pi^{-1} \pi (m_u x)$ and $\pi^{-1} \pi (i_u(m_u x))$ are disjoint since the involution on $A(K)$ is free.
Every other element of $\coprod_{v \in G} A(H_v)$ is accounted for in \eqref{graph sub enumerate A},\eqref{graph sub enumerate B},\eqref{graph sub enumerate C}, (that is, each other $y$ in this set satisfies $|\pi^{-1}\pi (y)| \leq 2$) so the inclusions in \eqref{arc inclusions that should be equalities} are actually equalities.
Since neither of the two elements of $\pi^{-1} \pi (m_u x)$ is in $\coprod_{v\in G} D(H_v)$, it follows that $\pi(m_u x)$ is not in $D(K)$. 
Hence $\pi(m_ux) \in \eth(K)$.

On the other hand, if $y\in \eth(K)$ then 
\begin{equation}\label{pi inv y inclusion}
\pi^{-1}(y) \subseteq \coprod_v \eth(H_v) = \coprod_v m_v(i \nbhd(v)).\end{equation}
Suppose that $m_u(x) \in \pi^{-1}(y)$ with $x\notin \eth(G)$.
Then $e=[x,ix]$ is an internal edge of $G$.
Without loss of generality about the ordering of the arcs of this internal edge, we have $\tilde \outeredge_e(\edgemajor) = m_w(ix)$ and $\tilde \inneredge_e(\edgeminor) = m_u(x)$, where $w = tx$.
Since $\tilde \outeredge_e(\edgeminor)$ is in $\pi^{-1}(y)$, we know by \eqref{pi inv y inclusion} that $\tilde \outeredge_e(\edgeminor) = i_w m_w(ix)$ is an element of $\eth(H_w)$.
Thus $H_w$ must be an edge, so $w=v_0$.
We cannot pull off this same trick twice, so $m_w^{-1}(i_wm_w(ix))$ is in $\eth(G)$ unless $w=u$.
If $w=u$, then we are in the situation where $G$ is the loop with one node and $H_w$ is the exceptional edge, which is explicitly disallowed.
Hence there is an $m_v(x) \in \pi^{-1}(y)$ with $x\in \eth(G)$.
\end{proof}

Note in particular that $A(K)$ is isomorphic, as a set, to $\eth(G) \amalg \coprod_v D(H_v)$.
The involution on this set can be defined directly (in the case when $K$ represents $G\{H_v\}$) using the involutions on $G$ and $H_v$, the bijections $m_v$, and the function $\mathbf{x}$ from Construction~\ref{construction xA}.
We will never explicitly need this fact in this paper.

\subsection{Embeddings and boundaries}
\label{section embeddings and boundaries}
We now go deeper in our study of embeddings.
Our key result is Proposition~\ref{proposition: embedding uniqueness} which tells us to which extent embeddings are determined by the images of their boundaries.

\begin{lemma}\label{lemma: boundary injective}
	Suppose that $f: G \to G'$ is in $\bigembeddings(G')$ and $\eth(G) = A \setminus D$. 
	Then the composite $\eth(G) \hookrightarrow A \xrightarrow{f} A'$ is a monomorphism. 
\end{lemma}
\begin{proof}
If $G$ is the exceptional edge with $A = \eth(G) = \eearcs$ (see Example~\ref{examples combinatorial}), then $f(\edgemajor) \neq i'(f(\edgemajor)) = f(\edgeminor)$ so $\eth(G)=A \to A'$ is injective.
For the remaining cases, simply notice that the indicated map is the following composite.
\[ \begin{tikzcd}[column sep=small, row sep=small]
\eth(G) \dar \rar[hook, "i"] & D \dar \rar[hook, "f"] & D' \dar \rar[hook, "i'"] & A'\\
A \rar["i", "\cong" swap] & A & A' \arrow[ur,"i'", bend right, "\cong" swap]
\end{tikzcd} \]
\end{proof}
In the preceding proof, we have relied on connectivity of $G$ to ensure that $i(\eth G) \subseteq D$; the only time this does not happen for connected graphs is when $G$ is the exceptional edge.

It is useful to isolate a subset of $A(G')$ that is isomorphic to $\eth(G)$ via $f$.
We will extend the following definition to the more general setting of `graphical maps' in Definition~\ref{def boundary of graphical map}.
\begin{definition}[Boundary of an embedding] 
\label{def boundary of embedding}
If $f: G\to G'$ is an embedding, we write $\eth(f) \subset A'$\index{$\eth(f)$} for the image of the (injective) function $f|_{\eth(G)} : \eth(G) \to A'$.
\end{definition}

\begin{lemma}
\label{lemma embedding not monomorphism}
Suppose that $f: G \to G'$ is in $\bigembeddings(G')$. 
If $a_1 \neq a_2$ are elements of $A$ such that $f(a_1) = f(a_2)$, then one of the following two situations holds:
\begin{enumerate}
	\item $a_1, ia_2 \in \eth(G)$ and $ia_1, a_2 \in D$, or \label{first_condition_embed_not_mono}
	\item $a_1, ia_2 \in D$ and $ia_1, a_2 \in \eth(G)$. \label{second_condition_embed_not_mono}
\end{enumerate}
In particular, since $f|_{\eth(G)}$ is injective, $|f^{-1}(a)| \leq 2$ for every $a\in A'$.
\end{lemma}
\begin{proof}
If $G = \exceptionaledge$, then $A \to A'$ is a monomorphism and the statement is vacuously true.
We thus suppose that $G\neq \exceptionaledge$.
Since $f: D \to D'$ is a monomorphism, at least one of $a_1, a_2$ is not in $D$.
Similarly, since $f: \eth(G) \to A'$ is a monomorphism by Lemma~\ref{lemma: boundary injective}, at least one of $a_1, a_2$ is not in $\eth(G)$.
Since $A = D \amalg \eth(G)$, there are $j\neq k$ in $\{1,2\}$ so that $a_j \in D$ and $a_k \in \eth(G)$.

Since $a_k \in \eth(G)$ and $G\neq \exceptionaledge$, we know that $ia_k \in D$.
Further, we have $f(ia_1) = f(ia_2)$ with $ia_1 \neq ia_2$, so by the first paragraph we have $ia_j \in \eth(G)$.
Case \eqref{first_condition_embed_not_mono} occurs when $(j,k) = (2,1)$, while case \eqref{second_condition_embed_not_mono} occurs when $(j,k) = (1,2)$.
\end{proof}

We next wish to consider a diagram of embeddings of the form 
\[ \begin{tikzcd}
G \rar[shift left, "h"] \rar[shift right, "k" swap] & G' \rar{f} & G'' 
\end{tikzcd} \]
with $fh = fk$.
Since $f$ need not be a monomorphism in $\finset^{\mathscr{I}}$, one wouldn't expect to deduce that $h=k$.
Indeed, we have the following counterexample.

\begin{example}
\label{example edge mono failure}
Consider the three graphs $G, G', G''$ in Figure~\ref{fig: edge mono failure}.
\begin{figure}
\labellist
\small\hair 2pt
 \pinlabel {$0$} [l] at 18 77
 \pinlabel {$i0$} [l] at 18 42
 \pinlabel {$1$} [l] at 91 92
 \pinlabel {$i1$} [l] at 91 68
 \pinlabel {$2$} [l] at 91 40
 \pinlabel {$i2$} [l] at 91 16
 \pinlabel {$i3$} [r] at 155 70
 \pinlabel {$3$} [r] at 155 38
\endlabellist
\centering
\includegraphics[scale=0.7]{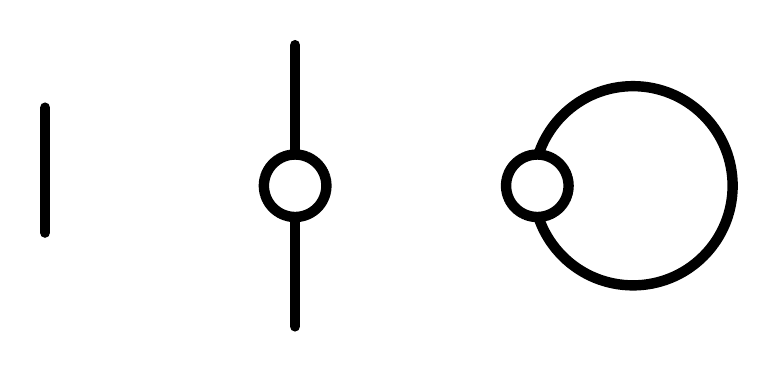}
\caption{Graphs $G, G', G''$ for Example~\ref{example edge mono failure}}
\label{fig: edge mono failure}
\end{figure}
Let $h, k : G \to G'$ and $f: G' \to G''$ be the embeddings uniquely specified by $h(0) = 1$, $k(0) = 2$, and $f(1) = f(2) = 3$.
Then $fh = fk$, but $h\neq k$.
\end{example}

The main issue in the previous example was that $G$ was the exceptional edge.
Indeed, we have the following.

\begin{lemma}
\label{lemma mono-like property}
Suppose that
\[ \begin{tikzcd}
G \rar[shift left, "h"] \rar[shift right, "k" swap] & G' \rar{f} & G'' 
\end{tikzcd} \]
is a diagram of embeddings, with $G, G', G''$ connected graphs and $G \neq \exceptionaledge$.
If $fh = fk$, then $h = k$.
\end{lemma}
\begin{proof}
We have a commutative diagram
\[ \begin{tikzcd}
D \rar[shift left, "h"] \rar[shift right, "k" swap] & D' \rar{f} & D'' 
\end{tikzcd} \]
with $f : D' \to D''$ a monomorphism, so $h$ and $k$ are identical on $D$.
We must show that they agree for elements in $A \setminus D = \eth(G)$.
Let $a \in \eth(G)$.
Since $G$ is connected and not the exceptional edge, we know $ia \in D$.
Thus we have the middle equality in $h(a) = h(iia) = ih(ia) = ik(ia) = k(iia) = k(a)$, so $h=k$ on $\eth(G)$.
\end{proof}

\begin{proposition}\label{proposition: embedding uniqueness}
	Suppose that $f_1 : G_1 \to G$ and $f_2: G_2 \to G$ are in $\bigembeddings(G)$ with neither $G_1$ nor $G_2$ the exceptional edge. 
	If $\eth(f_1) = \eth(f_2)$, then there is a unique isomorphism $z: G_1 \to G_2$ so that $f_1 = f_2z$.
	The same statement is true if both $G_1$ and $G_2$ are the exceptional edge $\exceptionaledge$.
\end{proposition}

It may be the case that $\eth(f_1) = \eth(f_2)$ but $G_1 \not \cong G_2$. For instance, in Example~\ref{example edge mono failure} we have $\eth(f) = \{ i3, 3\} = \eth(fk)$ but $G\not\cong G'$.
The following lemma addresses the empty boundary case of Proposition~\ref{proposition: embedding uniqueness}, which will be key in proving the general case.

\begin{lemma}\label{lemma: empty boundary}
	Suppose that $f: G' \to G$ is in $\bigembeddings(G)$ and the boundary of $G'$ is empty. 
	Then $f$ is an isomorphism. 
\end{lemma}
\begin{proof}
By assumption, the inclusion $D' \subseteq A'$ is an equality.
	Since $G'$ is not empty, $V' \neq \varnothing$.
	Suppose that $V' \to V$ is not surjective; then there exists a pair $v',v$ with $f(v')$ connected to $v\notin f(V')$. 
	Write $a \in \nbhd(f(v'))$ and $i(a) \in \nbhd(v)$ for the two orientations of the connecting edge.
	By the \'etale condition for $f$, there is a unique $d' \in D'$ with $f(d') = a$ and $t'(d') = v'$.
	Notice that $ia = if(d') = f(i'd')$. 
	Since $A' = D'$, $\tilde v = t'(i'd') \in V'$ is defined, and, further, $f(\tilde v) = f(t'i'd') = tf(i'd') = t(ia) = v$. 
	This is a contradiction, hence $V' \to V$ is surjective.
	Now $V' \to V$ is an bijection, so the \'etale condition implies that $D' \to D$ is a bijection as well. 
	Connectedness of $G$ ensures $D = A$.
\end{proof}

The following construction will help reduce the proof of Proposition~\ref{proposition: embedding uniqueness} to the special case from Lemma~\ref{lemma: empty boundary}.

\begin{construction}[Determination by core and boundary]
\label{construction core pushout}
As in Definition~\ref{definition I and natural trans}, we write $\mathscr{I}$ for the category $\begin{tikzcd}[column sep=small] \bullet \arrow[loop left] & \lar \bullet \rar & \bullet \end{tikzcd}$.
Given a connected graph $H \neq \exceptionaledge$, we have a pushout diagram in $\finset^{\mathscr{I}}$
\begin{equation}\label{diagram: core pushout}
\begin{tikzcd}
	\coprod\limits_{\eth(H)} \medstar_0 \rar \dar & \coprod\limits_{\eth(H)} \medstar_1 \dar \\
	\core(H) \rar[hook] & H,
\end{tikzcd}\end{equation}
where $\core(H) = H \setminus \eth(H)$ is from Definition~\ref{definition: edge subtraction}; note that $\core(H)$ and $H$ have an identical set of vertices.
Here, the top map is induced from the unique morphism $\medstar_0 \to \medstar_1$ on each component.
The left vertical map, at the component $a\in \eth(H)$, sends the unique vertex of $\medstar_0$ to the vertex $tia$.
For the diagram to commute, we know precisely what the right vertical map must do on vertices.
At the component $a\in \eth(H)$, the right vertical map sends the unique boundary arc of $\medstar_1$ to $a$.
Notice that the vertical maps need not be monomorphisms in the diagram category $\finset^{\mathscr{I}}$, and that the maps in the diagram are \'etale only when $\eth(H) = \varnothing$.
\end{construction}

\begin{proof}[Proof of Proposition \ref{proposition: embedding uniqueness}]
It is very simple that the isomorphism $z$, if it exists, is unique: since $f_1 : V_1 \to V$ and $f_2: V_2 \to V$ are monomorphisms, there is at most one map $z: V_1 \to V_2$ with $f_2z=f_1$.
A similar argument holds for uniqueness of $D_1 \to D_2$ and $\eth(G_1) \to \eth(G_2)$, which then implies uniqueness for $A_1 = \eth(G_1) \cup D_1 \to A_2 = \eth(G_2) \cup D_2$.

Notice immediately that we have isomorphisms 
\begin{equation}\label{equation: z on B}
\eth(G_1) \xrightarrow{f_1} \eth(f_1) = \eth(f_2) \xleftarrow{f_2} \eth(G_2),
\end{equation} 
which determines $z$ on $\eth(G_1)$.
At this point we can assume that $\eth(G_1)$ is nonempty, as when $\eth(G_1)$ is empty, Lemma \ref{lemma: empty boundary} implies that $f_1$ and $f_2$ are both isomorphisms.
Further, if both $G_1$ and $G_2$ are isomorphic to the exceptional edge $\exceptionaledge$, then $\eth(G_1) = A(G_1)$ and $\eth(G_2) = A(G_2)$, so \eqref{equation: z on B} gives the isomorphism $z$.

We assume for the remainder of the proof that $G_1 \neq \exceptionaledge \neq G_2$ and $\eth(G_1) \neq \varnothing \neq \eth(G_2)$.
By Lemma~\ref{lemma embedding not monomorphism} and Construction~\ref{construction core pushout} we have the outer commutative diagram
\[\begin{tikzcd}
	\coprod\limits_{\eth(G_1)} \medstar_0 \arrow[dd] \arrow[rr, "\cong", "z|_{\eth(G_1)}" swap] & & \coprod\limits_{\eth(G_2)} \medstar_0 \arrow[dd] \\ 
	& {\color{purple} G'} \dar[hook, color=purple]
	\\
	\core(G_1) \rar \arrow[ur, dashed, color=purple, "f_1'"] & G \setminus \eth(f_1) & \core(G_2) \lar \arrow[ul, dashed, color=purple, "f_2'" swap]
\end{tikzcd}\]
in $\finset^{\mathscr{I}}$.
Letting $G'$ be the connected component of $G \setminus \eth(f_1) = G \setminus \eth(f_2)$ (as in Definition~\ref{definition: edge subtraction}) which contains the vertex $f_1(tia)$ for some $a\in \eth(G_1)$, we have induced diagram maps ${\color{purple}f_j'} : \core(G_j) \to G'$ for $j=1,2$. 
These maps are \'etale, in fact, the bottom maps they factor are \'etale. 
To see this, note that we have an induced bijection
\[
\begin{tikzcd} 
t^{-1}(v) \setminus i \eth(G_j) \dar[hook] \rar[color = purple, dashed,"\cong"] & t^{-1}(f_j(v)) \setminus i\eth(f_j) \dar[hook]
\\
	t^{-1}(v) \rar["\cong"] & t^{-1}(f_j(v))
\end{tikzcd}
\]
for every vertex $v \in V_j$. 
Since $f_1$ and $f_2$ were embeddings, so too are the \'etale maps $f_1'$ and $f_2'$.

It follows from Lemma~\ref{lemma: empty boundary} that $f_1'$ and $f_2'$ are isomorphisms.
Since \eqref{diagram: core pushout} in Construction~\ref{construction core pushout} is a pushout, we obtain an isomorphism $z : G_1 \to G_2$ making the appropriate diagram commute.
\end{proof}

The collection $\bigembeddings(G)$ is rather flabby, with many uniquely isomorphic elements.
Let us rectify this.
\begin{definition}[Small set of embeddings]
Write $\embeddings(G)$\index{$\embeddings(G)$} for the quotient of $\bigembeddings(G)$ by the relation $f \sim h$ if there is an isomorphism $z$ so that $f=hz$.  	
\end{definition}
By Proposition~\ref{proposition: embedding uniqueness}, the isomorphism $z$ witnessing $f\sim h$ is unique.

\begin{example}
Let $G$ be the loop with two vertices (Example~\ref{examples combinatorial}).
Then 
\[ A(G) = \nbhd(1) \amalg \nbhd(2) = \{2^\dagger,1\} \amalg \{1^\dagger, 2\}\] 
has four elements.
There exist embeddings $f : H \to G$ if and only if $H$ is isomorphic to $\exceptionaledge\cong L_0$, $L_1$, $L_2$, or $G$. 
We use the notation for arcs from Example~\ref{examples combinatorial}.
\begin{itemize}
	\item If $H = L_n$ is a linear graph $L_0$, $L_1$, or $L_2$, then there are exactly four embeddings $f : L_n \to G$.
	Each such embedding is determined by where it sends $0\in \eth(L_n)$ (or any chosen arc in $A(L_n)$).
	The arc $f(0^\dagger)$ is determined since $f$ commutes with the involution.
	If $n > 0$, then the vertex $f(1)$ is given since $f(1) = f(t(0^\dagger)) = tf(0^\dagger)$; this determines the arc $f(1)$ since the neighborhood of the vertex $f(1)$ is the set of arcs $\{ f(0^\dagger), f(1) \}$, and so on.
	\item If $H=G$, then there are again four embeddings $G \to G$.
	Each such embedding is determined by where it sends some chosen arc, and each such embedding is an isomorphism (as in Lemma~\ref{lemma: empty boundary}).
\end{itemize}
We have exhibited sixteen elements in the infinite set $\bigembeddings(G)$, though every other $f: H \to G$ arises from one of these sixteen by fixing an isomorphism between $H$ and an element of the set $\{ L_0, L_1, L_2, G \}$.
All of these embeddings are injective on arcs except for those with domain $L_2$.

The set $\embeddings(G)$ has just seven elements.
Each of the embeddings $L_n \to G$ is isomorphic to precisely one of the others.
The class of $L_0 \to G$ is determined by which edge is hit.
The class of $L_1 \to G$ is determined by which vertex is hit.
The class of $L_2 \to G$ is determined by which edge of $G$ is hit twice.
Finally, each of the four automorphisms of $G$ are isomorphic (over $G$) to the identity automorphism.
\end{example}

\subsection{Definition of graphical maps}
In order to phrase certain `non-overlap' conditions for embeddings into a fixed graph, it is convenient to work in the free commutative monoid on a vertex set $V$.
For a finite set $S$, the free commutative monoid $\mathbb{N} S$ is isomorphic, as a monoid, to $\mathbb{N}^S$ but we write elements as $\sum_{s\in S} n_s s$ where each $n_s \in \mathbb{N}$.
We consider the power set $\wp(S)$\index{$\mathbb{N}S, \wp(S)$} as a subset of $\mathbb{N}S$, consisting of those elements with $n_s \leq 1$ for every $s\in S$.

\begin{definition}[Vertex sum $\varsigma$]
Given any \'etale map $f: G'\to G$, there is a corresponding element $ \sum_{v\in V'} f(v) \in \mathbb{N}V$ in the free commutative monoid on $V$.
The assignment of an \'etale map to its corresponding sum is invariant under isomorphisms in the domain. 
Denote by $\varsigma : \embeddings(G) \to \mathbb{N}V$\index{$\varsigma$} the map that sends an embedding $f: G' \to G$ to $\sum_{v\in V'} f(v)$.
Since we are only working with embeddings, we have $\varsigma(f) \leq \sum_{v\in V} v$, that is, $\varsigma$ lands in the power set $\wp(V) \subseteq \mathbb{N}V$. 
\end{definition}

\begin{definition}\label{def: graphical map}
	A \emph{graphical map} $\varphi: G \to G'$ consists of the following data:
	\begin{itemize}
		\item A map of involutive sets $\varphi_0 : A \to A'$ 
		\item A function $\varphi_1 : V \to \embeddings(G')$
	\end{itemize}
These data should satisfy three conditions. 
\begin{enumerate}[label={({\roman*})},ref={\thetheorem.\roman*}]
	\item The inequality $\sum_{v\in V} \varsigma(\varphi_1(v)) \leq \sum_{w\in V'} w$ holds in $\mathbb{N}V'$. \label{graphical map defn: no double vertex covering}
	\item For each $v$, we have a (necessarily unique) bijection making the diagram
	\[
		\begin{tikzcd}
			\nbhd(v) \rar{i} \dar[dashed, "\cong"]  & A \dar{\varphi_0} \\
			\eth(\varphi_1(v)) \rar[hook] & A'
		\end{tikzcd}
	\]
	commute, where the top map $i$ is the restriction of the involution on $A$.\label{graphical map defn: boundary compatibility}
	\item If the boundary of $G$ is empty, then there exists a $v$ so that $\varphi_1(v)$ is not an edge. \label{graphical map defn: collapse condition}
\end{enumerate}
\end{definition}

Figure~\ref{fig:example morphism} is a visual representation of a graphical map, where $\varphi_0$ is left implicit and where the red circles represent the images under $\varphi_1$ of the vertices of $G$. 
\begin{figure}[t]
\labellist
\small\hair 2pt
 \pinlabel {$G$} [B] at 95 428
 \pinlabel {$G'$} [B] at 443 428
 \pinlabel {{\color{blue}$\varphi$}} [ ] at 257 345
\endlabellist
\centering
\includegraphics[scale=0.4]{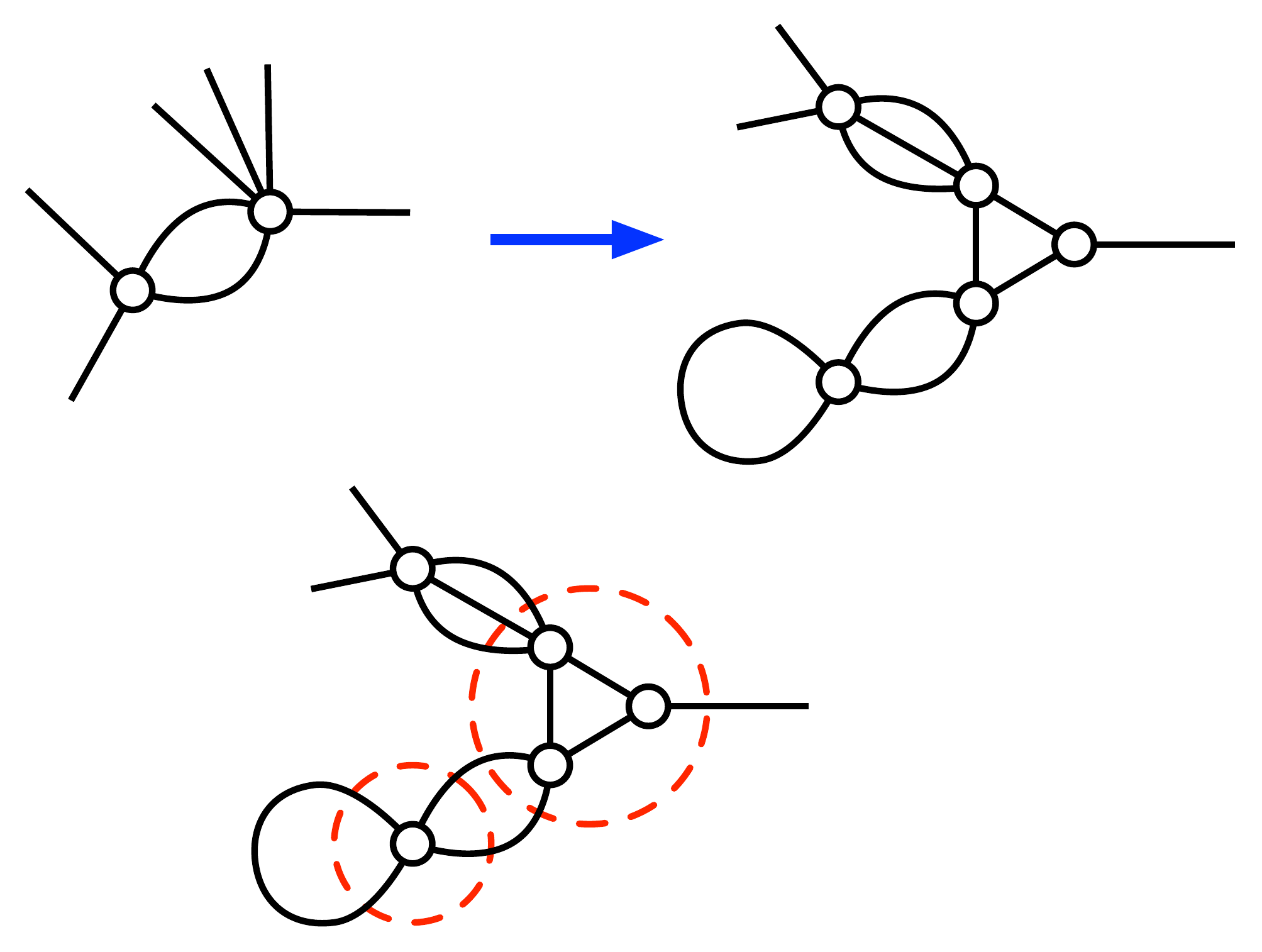}
\caption{A graphical map $\varphi : G \to G'$}
\label{fig:example morphism}
\end{figure}

\begin{remark}
We will often have need to refer to a particular element of $\bigembeddings(G')$ representing $\varphi_1(v)$\index{$\varphi_1(v), \varphi_v$}. 
We will always write $\varphi_v \in \bigembeddings(G')$ for a fixed such choice with $[\varphi_v] = \varphi_1(v) \in \embeddings(G')$.
Typically, the domain of $\varphi_v$ will be denoted by $H_v$.
\end{remark}

	The final condition of Definition~\ref{def: graphical map} is about avoiding \emph{collapse}. 
	It is only relevant if $G$ is of a particular form, that is, if $G$ is a single loop containing some (bivalent) vertices. 
	For example, if $G$ is the loop with one vertex $\begin{tikzpicture}
	\draw[line width=1pt] (0,0) circle (.9ex);
	\fill (-.9ex,0) circle (.4ex);
\end{tikzpicture} $ and $G'$ is the exceptional edge, then there is a pair $(\varphi_0, \varphi_1)$ from $G$ to $G'$ where $\varphi_0$ is a bijection and which satisfies \eqref{graphical map defn: no double vertex covering} and \eqref{graphical map defn: boundary compatibility} but not \eqref{graphical map defn: collapse condition}.

\begin{remark}[The graph category of Joyal \& Kock]
\label{remark jk category}
There is a related notion of morphism of connected graphs in \cite{JOYAL2011105}, but based on \'etale maps between connected graphs, rather than embeddings.
Joyal and Kock do not include the conditions \eqref{graphical map defn: no double vertex covering} and \eqref{graphical map defn: collapse condition} in their definition.
Further, condition \eqref{graphical map defn: boundary compatibility} is modified to reflect that \'etale maps need not be injective on boundaries.
This yields a category of connected graphs $\jkgraphcat$\index{$\jkgraphcat$}, and each graphical map in the sense of Definition~\ref{def: graphical map} is a morphism in $\jkgraphcat$.
\end{remark}

We have an ample supply of graphical maps: the embeddings.
Let us take a look at how this works.
As a precursor to Definition~\ref{graphical map composition}, we also indicate how to compose an arbitrary graphical map with an embedding.

\begin{definition}[Embeddings and restriction] 
\label{def embeddings and restrictions}
Every embedding $f: G \to G'$ in $\bigembeddings(G')$ determines a graphical map via $f: A \to A'$ and the composite
\[ \begin{tikzcd}
V \rar{f} & V' \rar[hook] & \embeddings(G'),
\end{tikzcd} \]
that is, $v \mapsto \iota_{fv}$ as in Example~\ref{example vertices as embeddings}.
We still call this graphical map `$f$.'
\begin{enumerate}
	\item If $\varphi: G \to G'$ is a graphical map and $f \in \bigembeddings(G)$, then $\varphi|_f$\index{$\varphi|_f, f\circ \varphi$} is the graphical map from the domain of $f$ to $G'$ defined by $(\varphi|_f)_0 = \varphi_0 f$ and $(\varphi|_f)_1 = \varphi_1 f$.
	\item Likewise, suppose $\varphi: G \to G'$ is a graphical map and $f : G' \to G''$ is an embedding.
	Define a new graphical map $f\circ \varphi$ with $(f \circ \varphi)_0 = f \varphi_0$ and $(f \circ \varphi)_1$ is the composite 
	\[
		V \xrightarrow{\varphi_1} \embeddings(G') \xrightarrow{f\circ (-)} \embeddings(G'').
	\]
\end{enumerate}

\end{definition}

It is relatively easy to see that $\varphi|_f$ is a graphical map.
For \eqref{graphical map defn: collapse condition}, note that if the domain of $\varphi|_f$ has empty boundary, then $f$ is an isomorphism by Lemma~\ref{lemma: empty boundary}.
In the next proposition we check that $f\circ \varphi$ is a graphical map.
Notice if $\varphi$ comes from an embedding $h$, then $\varphi|_f$ comes from the embedding $h\circ f$ and the map $f\circ \varphi$ comes from the embedding $f\circ h$.

\begin{proposition}
The pair of functions $(f\circ \varphi)_0, (f\circ \varphi)_1$ from Definition~\ref{def embeddings and restrictions} constitute a graphical map.
\end{proposition}
\begin{proof}
For each vertex $v$, pick a representative $(\varphi_v : H_v \to G') \in \bigembeddings(G')$ for $\varphi_1(v)$.
Notice that 
\[
	\sum_{v\in V} \varsigma(f\circ \varphi_v) = \sum_{v\in V} \sum_{u\in H_v} f \varphi_v(u) = f \sum_{v\in V} \sum_{u\in H_v} \varphi_v(u) = f  \sum_{v\in V} \varsigma(\varphi_v) \leq f \sum_{w\in V'} w.
\]
Since $f$ is a injective on vertices, this last term is less than or equal to $\sum_{x \in V''} x$ in $\mathbb{N} V''$, so \eqref{graphical map defn: no double vertex covering} holds.
Using Lemma~\ref{lemma: boundary injective}, commutativity of the diagram 
\[ \begin{tikzcd}
& \nbhd(v) \rar{i} \arrow[dl, bend right] \dar{\cong} & A \dar{\varphi_0} \\
\eth(H_v) \rar["\varphi_v", "\cong" swap] \arrow[dr, bend right, "\cong"] & \eth(\varphi_v) \rar[hook] \dar{f} & A' \dar{f} \\
& \eth(f\circ \varphi_v) \rar[hook] & A''
\end{tikzcd} \]
shows that \eqref{graphical map defn: boundary compatibility} holds for $f\circ \varphi$.
Condition \eqref{graphical map defn: collapse condition} for $f\circ \varphi$ follows immediately from this condition for $\varphi$.
\end{proof}

\begin{lemma}\label{lemma degen implies not injective}
	If $\varphi: G \to G'$ is a graphical map and $\varphi_1(v)$ is an edge for some $v$, then $\varphi_0$ is not injective.
\end{lemma}
\begin{proof}
By \eqref{graphical map defn: boundary compatibility} we know that if $\varphi_1(v)$ is an edge, then $\nbhd(v)$ has order two. 
Let us first address the cases when $G$ has a single vertex.
The case where $G$ is the loop with one node (Example~\ref{examples combinatorial}) is disallowed by \eqref{graphical map defn: collapse condition}, so we must be in the case when $G$ is isomorphic to the linear graph $L_1 \cong \medstar_2$.
Then $|A| = 4$, while only two elements of $A'$ are in the image of $\varphi_0$, so the result follows.

Now suppose that $v$ and $w$ are adjacent, distinct vertices of $G$ and $\varphi_1(v)$ and $\varphi_1(w)$ are both edges. 
	Write $\nbhd(v) = \{ a , b \}$ and $\nbhd(w) = \{ ia, c \}$. 
	If $b = ic$ then $G$ is a loop with two vertices and the map violates \eqref{graphical map defn: collapse condition}. 
	Thus $b\neq ic$. 
	By assumption that $\varphi_1(v)$ and $\varphi_1(w)$ are edges, we have $\varphi_0(a) = i\varphi_0(b)$ and $\varphi_0(ia) = i\varphi_0(c)$.
	Since $\varphi_0$ commutes with $i$, this implies that $\varphi_0(b) = \varphi_0(ic)$, so $\varphi_0$ is not injective.

Finally, suppose that $w$ and $v$ are adjacent vertices of $G$, $\varphi_1(v)$ is an edge, and $\varphi_1(w)$ is not an edge.
Write $\nbhd(v) = \{a,b\}$ with $ia\in \nbhd(w)$ (that is, so that $[a,ia]$ is an edge between $v$ and $w$).
	We know that $ia\neq b$ since $v\neq w$, yet $\varphi_0(b) = i\varphi_0(a) = \varphi_0(ia)$. 
	Thus $\varphi_0$ is not injective.
\end{proof}

\begin{theorem}\label{theorem: injective determination at zero}
	Suppose that $\varphi, \psi : G \to G'$ are graphical maps with $\varphi_0 =\psi_0$. 
	If $\varphi_0$ is injective,
	then $\varphi = \psi$.
\end{theorem}
\begin{proof}
	By the second condition for graphical map, for each $v$ we have $\eth(\varphi_1(v)) = \eth(\psi_1(v))$.
	By the contrapositive of the previous lemma, we know that $\varphi_1(v)$ and $\psi_1(v)$ are not edges, so by Proposition \ref{proposition: embedding uniqueness} we have $\varphi_1(v) = \psi_1(v)$.
\end{proof}

We now show how graph substitution is related to graphical maps.

\begin{proposition}\label{proposition: image}
	Suppose that $\varphi: G \to G'$ is a graphical map, and write $\varphi_v : H_v \hookrightarrow G'$ for an embedding representing $\varphi_1(v)$.
	Then there is an embedding $k : G\{H_v\} \hookrightarrow G'$ which factors all of the embeddings $\varphi_v$. 
\end{proposition}
\begin{proof}
If $G$ is the exceptional edge, then $\varphi$ is already an embedding from $G = G\{ \, \}$ to $G'$. Suppose $V$ is nonempty.
The isomorphisms $m_v: i(\nbhd(v)) \to \eth(H_v)$ are defined, using \eqref{graphical map defn: boundary compatibility}, so that $\varphi_v(m_v(x)) = \varphi_0(x)$.
Since $G$ and all $H_v$ are connected, so is $G\{H_v\}$ \cite[Proposition 6.12]{yj15}.
Consider the diagram whose top line is from Construction~\ref{construction graph sub}.
\[ \begin{tikzcd}
\coprod\limits_{e \in E_i} \exceptionaledge \rar[shift left, "\tilde \outeredge"] \rar[shift right, "\tilde \inneredge" swap] & 
\coprod\limits_{v\in V} H_v \rar["\pi"] \arrow[d, "\coprod \varphi_v" swap] & K \arrow[dl,dashed, "k"] \\
&  G'
\end{tikzcd} \]
We have
\begin{align*}
\varphi_{tx_e^1} \left( \tilde \outeredge_e (a) \right) &= \varphi_{tx_e^1} \left( m_{tx_e^1}(ix_e^1) \right) = \varphi_0 (ix_e^1) \\
\varphi_{tx_e^2} \left( \tilde \inneredge_e(a^\dagger) \right) &= \varphi_{tx_e^2} \left( m_{tx_e^2}(ix_e^2) \right) = \varphi_0 (ix_e^2) = \varphi_0(x_e^1);
\end{align*}
since all maps are equivariant we have $(\coprod \varphi_v) \tilde \outeredge = (\coprod \varphi_v) \tilde \inneredge$, thus $k : K \to G'$ exists.
The map $k$ is automatically \'etale, and further we have that $k$ is injective as a map from $V(K) = \coprod_v V(H_v)$ to $V(G')$ by \eqref{graphical map defn: no double vertex covering}.
Thus $k$ is an embedding.
\end{proof}
This proof shows that the following is well-defined (that is, does not depend on the choice of $\varphi_v \in \bigembeddings(G')$ representing $\varphi_1(v) \in \embeddings(G')$).
\begin{definition}[Image of a graphical map] \label{definition image}
If $\varphi : G \to G'$ is a graphical map, then the embedding $k : G\{H_v\} \to G'$ from Proposition~\ref{proposition: image} represents an element $\image(\varphi) \in \embeddings(G')$\index{$\image(\varphi)$} called the \emph{image of $\varphi$}.
\end{definition}

\begin{remark}[Image of an embedding]\label{remark image of embedding}
Notice that if $\varphi : G\to G'$ is coming from an embedding $f$, then we can actually take $f$ itself as a representative for $\image(\varphi)$.
Indeed, $\varphi_1(v)$ can be represented by $\iota_{f(v)} : \medstar_{f(v)} \to G'$ where $\medstar_{f(v)} \cong \medstar_v$, and $G\{ \medstar_{f(v)} \} \cong G\{\medstar_v \} \cong G$.
\end{remark}

\begin{lemma}\label{lemma: image factorization}
	Suppose $\varphi: G \to G'$ is a graphical map, and let $k  : G\{ H_v \} \hookrightarrow G'$ represent its image.
	Then there exists a graphical map $\varphi' : G \to G\{ H_v \}$ so that $\varphi_0'$ is a bijection on boundaries and
	\[ \begin{tikzcd}
	G \rar{\varphi} \dar{\varphi'} & G' \\
	G \{ H_v \} \arrow[ur, hook, "k" swap]
	\end{tikzcd} \]
	commutes.
\end{lemma}
\begin{proof}
Once again we exclude the simple case when $G= \exceptionaledge$ (in which case $k= \varphi$ and $\varphi' = \id$), and reuse the notation from Construction~\ref{construction graph sub} and the proof of Proposition~\ref{proposition: image}.
Let $\varphi_v'$ be the composite $H_v \hookrightarrow \coprod_{v\in V} H_v \xrightarrow{\pi} K$; we have $k\varphi_v' = \varphi_v$.
Write $\varphi_1'(v) = [\varphi'_v] \in \embeddings(K)$.
Define $\varphi_0' : A = D \cup iD \to A(K)$ by
\[
\varphi_0' (a) = 
\begin{cases}
\varphi_v'(m_v(a))& \text{if }a\in i(\nbhd(v)) \\
i\varphi_v'(m_v(ia)) & \text{if $ia\in \eth(G)$ and $a\in \nbhd(v)$.}
\end{cases}
\]
Then if $a\in i(\nbhd(v))$ we have $k\varphi_0' (a) = k\varphi_v'(m_v(a)) = \varphi_v(m_v(a)) = \varphi_0(a)$, while if $ia\in \eth(G)$ and $a\in \nbhd(v)$ we have $k\varphi_0' (a) = ki\varphi_v'(m_v(ia)) = i\varphi_v(m_v(ia)) = i\varphi_0(ia) = \varphi_0(a)$.
Thus we have established that $\varphi = k \varphi'$, assuming that $\varphi'$ is a graphical map.

Let us now show that $\varphi'$ is a graphical map. 
Condition \eqref{graphical map defn: no double vertex covering} follows since $\coprod \varphi_v' : \coprod V(H_v) \to V(K)$ is a bijection, hence a monomorphism.
The composite of bijections
\[ \begin{tikzcd}
\nbhd(v) \rar{i} & i(\nbhd(v)) \rar{m_v} & \eth(H_v) \rar{\varphi_v'} & \eth(\varphi_v')
\end{tikzcd} \]
satisfies $\varphi_v' m_v i = \varphi_0' i$, so \eqref{graphical map defn: boundary compatibility} holds.
Finally, \eqref{graphical map defn: collapse condition} holds for $\varphi'$ by the corresponding condition for $\varphi$.

The graphical map $\varphi'$ is a bijection on boundaries by Lemma~\ref{lemma boundaries}.
\end{proof}

Suppose that $\varphi : G \to G'$ is any graphical map. 
By the previous lemma and Lemma \ref{lemma: boundary injective} we know that $\varphi_0|_{\eth(G)}$ is injective.
We extend the definition of $\eth$, given in Definition~\ref{def boundary of embedding}, from embeddings to arbitrary graphical maps.
\begin{definition}[Boundary of a graphical map] \label{def boundary of graphical map}
	If $\varphi : G \to G'$ is any graphical map, let $\eth(\varphi) \cong \eth(G)$\index{$\eth(\varphi)$} be the subset $\varphi_0(\eth(G)) \subseteq A'$.
\end{definition}

In Definition~\ref{graphical map composition} we will explain how to compose two graphical maps.
The following is key to ensuring that composition is well-defined. 

\begin{lemma}
\label{lemma image restriction defined}
Let $\varphi : G\to G'$ be a graphical map, and let $f : G_1 \to G$ and $h : G_2 \to G$ be embeddings.
If $z$ is an isomorphism with $f = hz$, 
then $\image(\varphi|_f) = \image(\varphi|_h)$ in $\embeddings(G')$.
\end{lemma}
In other words, the function $\bigembeddings(G) \to \embeddings(G')$ which sends $f$ to $\image(\varphi|_f)$ factors through $\embeddings(G)$.
Of course, $\varphi|_f$ need not be equal to $\varphi|_h$ (they need not even have the same domain), so $\varphi|_\ell$ is not defined for $\ell \in \embeddings(G)$.
Despite that fact, we will still use the notation $\image(\varphi|_{\ell})$ when $\ell \in \embeddings(G)$.
\begin{proof}
In the proof of Proposition~\ref{proposition: image} we represented these images as coming the universal property of coequalizers.
Consider the following diagram.
The map $k_1$ represents $\image(\varphi|_f)$ while the map $k_2$ represents $\image(\varphi|_h)$.
The isomorphisms on the left come from $z$ applied to the indexing sets for the coproducts. 
\[ \begin{tikzcd}
\coprod\limits_{e \in E_i(G_1)} \exceptionaledge \rar[shift left] \rar[shift right]  \arrow[dd, "\cong"] & 
\coprod\limits_{v\in V(G_1)} H_{fv} \arrow[rr,"\pi_1"] \arrow[dr, "\coprod \varphi_{fv}" description] \arrow[dd, "\cong"]  & & K_1 \arrow[dl,dashed, "k_1"] \arrow[dd,dotted] \\
& &  G' \\
\coprod\limits_{e \in E_i(G_2)} \exceptionaledge \rar[shift left] \rar[shift right]  & \coprod\limits_{w\in V(G_2)} H_{hw}  \arrow[rr,"\pi_2"] \arrow[ur, "\coprod \varphi_{hw}" description] & & K_2 \arrow[ul,dashed, "k_2"]  \\
\end{tikzcd} \]
Then $K_1 \to K_2$ is an isomorphism as well, and we see that $k_1$ and $k_2$ represent the same element of $\embeddings(G')$.
\end{proof}

\begin{definition}[Composition in $\graphicalcat$]
\label{graphical map composition}
If $\varphi : G \to H$ and $\psi: H \to K$ are two graphical maps, define $(\psi \circ \varphi)_0$ on $A$ and $(\psi \circ \varphi)_1$ on $V$ by
\begin{align*}
	(\psi \circ \varphi)_0(a) &= \psi_0 (\varphi_0(a)) \in A(K) \\
	(\psi \circ \varphi)_1(v) &= \image(\psi|_{\varphi_1(v)}) \in \embeddings(K).
\end{align*}
\end{definition}
In light of Lemma~\ref{lemma image restriction defined}, $(\psi \circ \varphi)_1$ is a well-defined function.
We must still verify that $\psi \circ \varphi$ is a graphical map when both $\varphi$ and $\psi$ are; this will occur in the proof of Theorem~\ref{theorem graphicalcat is a category}. 
Before doing that, we should address the following potential inconsistency: we've already defined composition when one of $\varphi$ or $\psi$ is an embedding.

\begin{remark}[Composition with embeddings] 
Let $\varphi:  G\to G'$ be a graphical map with chosen embeddings $\varphi_v : H_v \to G'$ representing $\varphi_1(v)$.
Let us compare the composition from Definition~\ref{graphical map composition} with previously mentioned compositions with embeddings from Definition~\ref{def embeddings and restrictions}.
\begin{itemize}
	\item If $f \in \bigembeddings(G)$ is an embedding, we defined $(\varphi|_f)_1(w)$ to be $\varphi_1(f(w))$, which is represented by $\varphi_{f(w)} : H_{f(w)} \to G'$.
	On the other hand, regarding $f$ as a graphical map, we know that $f$ sends $w$ to $[\iota_{f(w)}] \in \embeddings(G)$.
	But $\image(\varphi|_{\iota_v} : \medstar_v \to G')$ is represented by $\varphi_v : \medstar_v \{ H_v \} = H_v \to G'$, so $(\varphi \circ f)_1(w)$ is also represented by $\varphi_{f(w)}$.
	Thus $\varphi|_f = \varphi \circ f$.
	\item Suppose that $f$ is an embedding with domain $G'$.
	In Definition~\ref{def embeddings and restrictions} we declared $(f\circ \varphi)_1(v)$ to be represented by $f \circ \varphi_v$.
	On the other hand, in Definition~\ref{graphical map composition} we said that $(f\circ \varphi)_1(v)$ should be represented by $\image(f|_{\varphi_v})$.
	The graphical map $f|_{\varphi_v}$ comes from the embedding $f\circ \varphi_v$, so by Remark~\ref{remark image of embedding} we have that $\image(f|_{\varphi_v})$ is also represented by $f\circ \varphi_v$.
	Thus there is no ambiguity about what we mean by $f\circ \varphi$.
\end{itemize}
\end{remark}

The following two lemmas will be used in verifying that $\psi\circ\varphi$ is a graphical map in the proof of Theorem~\ref{theorem graphicalcat is a category}.

\begin{lemma} \label{lemma image vartheta}
	If $\varphi : G \to G'$ is a graphical map and $\image \varphi$ is its image, then $\varsigma(\image \varphi) = \sum_{v\in V} \varsigma(\varphi_1(v))$.
\end{lemma}
\begin{proof}
Write $\varphi_v : H_v \to G'$ for an embedding representing $\varphi_1(v)$, and let $k : G\{H_v\} \to G'$ be the associated embedding representing $\image(\varphi)$.
Each $\varphi_v$ factors as $H_v \hookrightarrow G\{H_v\} \xrightarrow{k} G'$.
Identifying the vertex set of $G\{H_v\}$ with the disjoint union of the vertex sets of the $H_v$, we have 
\[
\sum_{v\in G} \varsigma(\varphi_v) = \sum_{v\in G} \sum_{w\in H_v} \varphi_v(w)  = \sum_{w\in G\{H_v\}} k(w) = \varsigma(k).
\]\end{proof}

\begin{lemma}\label{lemma: image and boundary}
	If $f \in \bigembeddings(G)$ is an embedding and $\varphi : G \to G'$ is a graphical map, then we have a commutative diagram
	\[ \begin{tikzcd}
	\eth(f) \rar[hook] \dar["\cong"] & A \dar{\varphi_0} \\
	\eth(\image(\varphi|_f)) \rar[hook] & A'
	\end{tikzcd} \]
	whose left map is a bijection.
\end{lemma}
\begin{proof}
Let $G''$ be the domain of $f$.
By Lemma~\ref{lemma: image factorization}, we know that 
\[
	\eth(\image(\varphi|_f)) = \eth(\varphi|_f) = (\varphi|_f)_0 (\eth(G'')).
\]
By definition of $\varphi|_f$, this is equal to $\varphi_0(f(\eth(G'')) = \varphi_0 (\eth(f))$.
The composition and the first map in 
\[
	\eth(G'') \xrightarrow{f} \eth(f) \xrightarrow{\varphi_0} \eth(\varphi|_f)
\]
are isomorphisms, hence $\varphi_0 : \eth(f) \to \eth(\varphi|_f)$ is an isomorphism.
\end{proof}

\begin{theorem}
\label{theorem graphicalcat is a category}
The graphical maps from Definition \ref{def: graphical map} assemble into a category $\graphicalcat$.
The objects of $\graphicalcat$ are the connected graphs (excluding nodeless loops).
\end{theorem}
\begin{proof}
The graphical maps that we have defined are all maps in the category $\jkgraphcat$ from \cite[\S 6]{JOYAL2011105}.
The composition is identical to that in $\jkgraphcat$, so the result follows as long as we can show that $\psi \circ \varphi$ from Definition~\ref{graphical map composition} is a graphical map.
Throughout, we let $\varphi_v : J_v \to H$ be an embedding which represents $\varphi_1(v)$.

We have, using Lemma~\ref{lemma image vartheta},
\begin{multline*}
	\sum_{v\in G} \varsigma((\psi \circ \varphi)_1(v)) 
	=
	\sum_{v\in G} \varsigma(\image (\psi|_{\varphi_1(v)})) 
	= 
	\sum_{v\in G} \sum_{w\in J_v} \varsigma((\psi|_{\varphi_v})_1(w)) \\
	=
	\sum_{v\in G} \sum_{w\in J_v} \varsigma((\psi|_{\varphi_v})_1(w))
	\leq  
	\sum_{u \in H} \varsigma(\psi_1(u)),
\end{multline*}
where the inequality is because $G\{J_v\} \hookrightarrow H$ is an embedding.
Since \eqref{graphical map defn: no double vertex covering} holds for $\psi_1$, this element is less than or equal to $\sum_{x\in K} x$.
Thus \eqref{graphical map defn: no double vertex covering} holds for $(\psi \circ \varphi)_1$.

To see that \eqref{graphical map defn: boundary compatibility} holds, note that we have 
\[
	\begin{tikzcd}
		\nbhd(v) \rar{i} \dar["\cong"]  & A(G) \dar{\varphi_0} \\
		\eth(\varphi_1(v)) \rar[hook] \dar["\cong"] & A(H) \dar{\psi_0} \\
		\eth(\image(\psi|_{\varphi_1(v)})) \rar[hook] & A(K)
	\end{tikzcd}
\]
where on the left-hand side, the bottom map is a bijection by Lemma \ref{lemma: image and boundary} and the top is a bijection by \eqref{graphical map defn: boundary compatibility}. 

For \eqref{graphical map defn: collapse condition}, suppose that $G$ has empty boundary.
By Proposition \ref{proposition: image}, there is an embedding $G \{ J_v \} \hookrightarrow H$ representing $\image \varphi$; since the boundary of $G\{ J_v \}$ is empty, Lemma \ref{lemma: empty boundary} implies this embedding is an isomorphism.
By \eqref{graphical map defn: collapse condition} applied to $\psi$, there is a vertex $w \in H$ with $\psi_1(w)$ not an edge; using the isomorphism of $H$ and $G\{J_v\}$, there exists a $v\in G$ so that $w \in J_v$.
Then $(\psi \circ \varphi)_1(v) = \image (\psi|_{\varphi_v})$ cannot be an edge, since $\psi_1(w)$ factors through it.
\end{proof}

\section{Factorization of graphical maps} \label{reedy}

Now that we've defined the category $\graphicalcat$, we exhibit two (orthogonal) factorization systems on it.
Recall that a factorization system on a category $\mathbf{C}$ consists of two classes of maps $L$ and $R$ (the left class and right class, respectively), each containing all isomorphisms and closed under composition.
The defining property is that every morphism $f$ of $\mathbf{C}$ factors as $f = r\ell$ with $\ell \in L$ and $r\in R$, and this factorization is unique up to unique isomorphism.
In Theorem~\ref{theorem orthogonal factorization system} we exhibit such a factorization system on $\graphicalcat$ whose right class is consists of all of the embeddings.
The left class of this factorization system consists of `active' maps, which play a major role in this paper and its companion \cite{modular_paper_two}.

The second factorization system we will deal with actually has more structure: we show that $\graphicalcat$ is a dualizable generalized Reedy category in the sense of Berger and Moerdijk \cite{bm_reedy}. 
This fact, established in Section~\ref{subsection reedy}, gives us Quillen model structures on categories of presheaves.
We will exploit this in Section~\ref{section simplicial presheaves on U} when developing model categories for Segal modular operads.

\subsection{Active maps and their properties}

Recall that a wide subcategory of a category $\mathbf{C}$ is a subcategory which contains all objects of $\mathbf{C}$.  

\begin{definition}[Active maps]\label{def: active}
A graphical map $\varphi: G\rightarrow G'$ is called \emph{active} if $\varphi_0$ is a bijection on boundaries, that is, if we have a commutative square
\[ \begin{tikzcd}
\eth(G) \rar[dashed, "\cong"] \dar[hook] & \eth(G') \dar[hook] \\
A \rar{\varphi_0} & A'
\end{tikzcd} \]
whose top map is an isomorphism.
We have two wide subcategories of $\graphicalcat$:
\begin{itemize}
\item $\graphicalcat_{\actrm}$\index{$\graphicalcat_{\actrm}$} is the subcategory of $\graphicalcat$ consisting of \emph{active maps}
\item $\graphicalcat_{\embrm}$\index{$\graphicalcat_{\embrm}$} is the subcategory consisting of \emph{embeddings} (Definition~\ref{def embeddings and restrictions}).
\end{itemize}
\end{definition}

In this subsection, we study the interactions of these two classes of maps, with the aim of showing that they constitute an orthogonal factorization system on $\graphicalcat$ (Theorem~\ref{theorem orthogonal factorization system}).
In the next subsection, we will use this fact to help establish a generalized Reedy structure on $\graphicalcat$.

\begin{lemma}\label{lemma: active maps from edge are iso}
	If $\varphi : \exceptionaledge \to G$ is an active map, then $\varphi$ is an isomorphism.
\end{lemma}
\begin{proof}
Consider the exceptional edge $\exceptionaledge$ from Example~\ref{examples combinatorial} and write $\eearcs$ for the set of arcs.
Suppose that $\varphi$ is not an isomorphism.
Then since $\varphi$ is active and $G$ is connected, the graph $G$ has a vertex $v$ together with an arc $b\in \nbhd(v)$ so that $ib \in \eth(G) = \varphi_0 \eearcs$.
Without loss of generality, we may assume that $ib = \varphi_0(\edgemajor).$ Then $b = \varphi_0(\edgeminor)$ is in $\eth(G)$.
This contradicts our assumption that $b \in \nbhd(v)$ and thus $\varphi$ must be an isomorphism.  
\end{proof}

\begin{proposition}
	A map $\varphi: G \to G'$ is active if and only if $\image \varphi$ is an isomorphism.
\end{proposition}
\begin{proof}
	The reverse implication follows because $G$ and $G\{H_v\}$ share a boundary.
	For the forward implication, if $G = \exceptionaledge$, then Lemma~\ref{lemma: active maps from edge are iso} implies $\varphi$ is an isomorphism hence $\image \varphi$ is as well. 
	If  $G \neq \exceptionaledge$ this is a special case of Proposition \ref{proposition: embedding uniqueness} using $f = \image \varphi$ and $h = \id_{G'}$.
\end{proof}

The following proposition follows immediately from the definition of active maps.

\begin{proposition}\label{proposition: kappa and active maps}
	If $\varphi$ is an active map and $\psi$ is any other graphical map so that $\psi \circ \varphi$ is defined, then $\eth(\psi) = \eth(\psi \circ \varphi)$. \qed 
\end{proposition}

\begin{definition}[{Wide subcategories of $\graphicalcat_{\actrm}$}]
\label{def wide subcats act}
Consider the following two subcategories:
\begin{enumerate} 
\item A map $\varphi$ is in $\graphicalcat^{-}$\index{$\graphicalcat^{-}$} if and only if $\varphi_0$ is surjective and $\varphi$ is active.
\item $\graphicalcat^{+}_{\actrm}$\index{$\graphicalcat^{+}_{\actrm}$} is the subcategory of $\graphicalcat_{\actrm}$ consisting of those active maps with $\varphi_0$ injective.
\end{enumerate} 
\end{definition}

We will later need another wide subcategory $\graphicalcat^+ \subseteq \graphicalcat$ (see Definition~\ref{almost injective def}), but we postpone its introduction until we have done some preliminary work. 
The following remark is not essential in what follows, but it describes a collection of generators for the category $\graphicalcat$ as well as certain important subcategories. 

\begin{remark}[Cofaces and codegeneracies]\label{remark cofaces codegens}
One can consider, for each of the subcategories $\graphicalcat_{\embrm}$, $\graphicalcat^{+}_{\actrm}$, and $\graphicalcat^{-}$, those non-isomorphisms $\varphi$ so that if $\varphi = \varphi^1 \circ \varphi^2$ (with all three morphisms in the subcategory), then either $\varphi^1$ or $\varphi^2$ is an isomorphism.
Such maps are called outer coface maps, inner coface maps, and codegeneracy maps, respectively.
To allow us to be a little more concrete about what these maps are, recall that, for a generic map $\varphi : G \to G'$ we write $\varphi_v : H_v \to G'$ for an embedding representing $\varphi_1(v)$. 
There are two kinds of embeddings which are outer cofaces: those embeddings where $G'$ has exactly one more internal edge than $G$, and maps from the exceptional edge into a star.
Inner cofaces and codegeneracies, which are all active, come with a distinguished vertex $w\in V(G)$ and have the property that $H_v \cong \medstar_v$ whenever $v \neq w$.
Such a map is an inner coface just when $H_w$ has exactly one internal edge, and is a codegeneracy just when $H_w$ is the exceptional edge.
The outer cofaces generate $\graphicalcat_{\embrm}$, the inner cofaces generate $\graphicalcat^{+}_{\actrm}$, and the codegeneracies generate $\graphicalcat^{-}$.
A coface map is just a map which is either an inner or outer coface, and the coface maps generate the category $\graphicalcat^{+}$ from Definition~\ref{almost injective def}.
\end{remark}

\begin{theorem}\label{theorem_factorization}
Given a graphical map $\varphi: G\rightarrow G'$, there is a factorization 
\[
\begin{tikzcd} 
G \dar[swap]{\varphi^-} \rar{\varphi} \arrow[dr, "\alpha"] & G' \\
G_1\rar[swap]{\delta} & G_2 \uar[swap]{\image \varphi}
\end{tikzcd}\]
 in which $\varphi^-$ is in $\graphicalcat^{-}$, $\delta$ is in $\graphicalcat^{+}_{\actrm}$, and the factorization $\varphi = (\image \varphi)\alpha$ was constructed in Lemma \ref{lemma: image factorization}. 
\end{theorem} 

\begin{proof} 
We must factor $\alpha$. 
Let $I = \{ v \in V \mid \varphi_1(v) \text{ is an edge}\}$ and let $I'$ be its complement in $V$.
Define $G_1 = G\{ \exceptionaledge \}_{v\in I}$ and let $\varphi^-$ be the evident map; note that $\varphi^-_0$ is surjective.
The set $I'$ may be identified with the set of vertices of $G_1$ via $\varphi^-$.
For each $v\in I'$, let $\delta_v$ be the embedding guaranteed by Proposition~\ref{proposition: image}, that is, so that 
\[
	H_v \xrightarrow{\delta_v} G_2 = G\{ H_v\} \xrightarrow{\image \varphi} G'
\]
is equal to $\varphi_v$.
Note that $G_1 \{ H_v \}_{v \in I'}  = G\{ \exceptionaledge \}_{v\in I}\{ H_v \}_{v \in I'} = G\{H_v \}_{v\in V } =  G_2 $ and this determines the map $\delta : G_1 \to G_2$.
The graphical map $\delta$ is active and $\delta_1(v) = [\delta_v]$ is not an edge for any $v \in G_1$.
It follows that $\delta_0$ is injective. 
\end{proof} 

We now employ a technical construction which supports the proof of Lemma~\ref{lemma mono-like property active}.
We recommend skipping Construction~\ref{construction xA} and Lemma~\ref{lemma xA mono} in an initial reading, by focusing on the case where $\varphi_0$ is injective.  

\begin{construction}\label{construction xA}
Let $\varphi : G \to G'$ be a graphical map.
Define $m = m^\varphi : A \to A$\index{$m^\varphi, \mathbf{x}$} by
\[
	ma = \begin{cases}
	a & \text{if $a\in \eth(G)$,} \\
	a & \text{if $a \in D$ and $\varphi_1(ta)$ is not an edge} \\
	a' & \text{if $a\in D$, $\varphi_1(ta)$ is an edge, and $\nbhd(ta) = \{a, ia'\}$.}
	\end{cases}
\]
For each $a$, there exists a $k > 0$ so that $m^k(a) = m^{k+1}(a)$.
This occurs since $A$ is finite and \eqref{graphical map defn: collapse condition} holds.
Write $\mathbf{x} = \mathbf{x}^\varphi : A \to A$ for the stabilization of $m$, that is, $\mathbf{x}a = m^ka$ for some $k$ and $m\mathbf{x}a = \mathbf{x}a$.
One can actually choose $k$ uniformly by taking, for instance, $k= |A|$.

Suppose that $\psi : G \to G'$ is another graphical map so that $\varphi_1(v)$ is an edge if and only if $\psi_1(v)$ is an edge.
Then $m^\varphi = m^\psi$, hence $\mathbf{x}^\varphi = \mathbf{x}^\psi$.
\end{construction}

If $\varphi$ is any graphical map, then $\varphi_0(a) = \varphi_0(ma)$, and it follows that $\varphi_0(a) = \varphi_0(\mathbf{x}a)$. 
Suppose that $\varphi^-$ is the graphical map appearing in Theorem~\ref{theorem_factorization}.
Notice that if $(\varphi^-)_0(a) = (\varphi^-)_0(a')$, then $\mathbf{x}a = \mathbf{x}a'$.
One interpretation of Construction~\ref{construction xA} is that it provides a preferred section to the surjective function $(\varphi^-)_0$.

\begin{lemma}
\label{lemma xA mono}
If $\varphi : G\to G'$ is active, then the restriction 
\[
	\varphi_0|_{\mathbf{x}A} : \mathbf{x}A \hookrightarrow A \xrightarrow{\varphi_0} A'
\]
is a monomorphism.
\end{lemma}
\begin{proof}
Suppose that $a_1, a_2$ are elements of $\mathbf{x}A$ such that $\varphi_0(a_1) = \varphi_0(a_2)$.
If $a_1$ and $a_2$ are both elements of $\eth(G)$, then $a_1 = a_2$ since $\varphi$ is active.

If $a_k \in D \cap \mathbf{x}A$ write $f_k : H_k \hookrightarrow G'$ for an embedding representing $\varphi_1(ta_k)$; we know that $H_k$ is not the exceptional edge.
Letting $i\ell \in \eth(H_k)$ be the unique element with $f_k(i\ell) = \varphi_0(ia_k)$ (by \eqref{graphical map defn: boundary compatibility}), we know $t\ell \in V(H_k)$ is defined since $H_k \neq \exceptionaledge$.
We thus have $\varphi_0(a_k) \in D'$ with $v_k = t\varphi_0(a_k) =t f_k(\ell) = f_k(t\ell)$ in the image of $f_k$.

Now, if $a_1$ is an element of $D \cap \mathbf{x}A$, then $a_2$ is as well.
If $a_2$ was in $\eth(G)$, then $\varphi_0(a_2)$ would be in $\eth(G')$, but in the previous paragraph we showed that $\varphi_0(a_1)$ is an element of $D'$.
We are thus left to address the situation where $a_1, a_2 \in D \cap \mathbf{x}A$.
We have $v_1 = v_2 \in f_1(V(H_1)) \cap f_2(V(H_2))$, so by \eqref{graphical map defn: no double vertex covering} we have $ta_1 = ta_2$.
Now $a_1,a_2 \in \nbhd(ta_1) = \nbhd(ta_2)$, so by \eqref{graphical map defn: boundary compatibility}, $\varphi_0(a_1) = \varphi_0(a_2)$ implies $a_1=a_2$.
\end{proof}

On a first reading of the following lemma, we recommend focusing on the case when $\varphi_0$ is already injective, so that $m(a) = a = \mathbf{x}(a)$ for all $a$ and Lemma~\ref{lemma xA mono} is trivial.  
\begin{lemma}
\label{lemma mono-like property active}
Suppose $\varphi, \psi : G \to G'$ are in $\graphicalcat_{\actrm}$ and $f: G' \to G''$ is in $\graphicalcat_{\embrm}$ with $f \varphi = f \psi$.
Then $\varphi = \psi$.
\end{lemma}
\begin{proof}
If $G$ is the exceptional edge, then the active maps $\varphi$ and $\psi$ are isomorphisms by Lemma~\ref{lemma: active maps from edge are iso}.
Any embedding $\exceptionaledge \to G''$ is a monomorphism (in $\finset^{\mathscr{I}}$), so the fact that $f \varphi = f \psi$ implies that the embeddings $\varphi$ and $\psi$ are equal.
For the remainder of the proof we only consider the case when $G \neq \exceptionaledge$.

Since $f$ is an embedding, for each vertex $v$, we have $\varphi_1(v)$ is an edge if and only if $(f\varphi)_1(v)$ is an edge, and similarly for $\psi$.
Since $f\varphi = f\psi$, we thus have $\varphi_1(v)$ is an edge if and only if $\psi_1(v)$ is an edge.
This implies that the functions $\mathbf{x}^\varphi$, $\mathbf{x}^\psi$ from Construction~\ref{construction xA} are equal; we simply write $\mathbf{x}$ for this function. 

We wish to show that $\varphi_0 = \psi_0$.
Suppose that $a$ is an arc of $G$ with $\varphi_0(a) \neq \psi_0(a)$; we will show that this leads to a contradiction. 
Since $\varphi_0(a) = \varphi_0(\mathbf{x}a)$ and $\psi_0(a) = \psi_0(\mathbf{x}a)$, we may replace $a$ by $\mathbf{x}a$ and assume $a\in \mathbf{x}A$.
Since $f\varphi_0(a) = f\psi_0(a)$, we apply Lemma~\ref{lemma embedding not monomorphism} to the arcs $\varphi_0(a)$ and $\psi_0(a)$ of $G'$.
One of these arcs is in $\eth(G')$ while the other is in $D'$.
Without loss of generality, assume that $\varphi_0(a) \in \eth(G')$ and $\psi_0(a) \in D'$, the other situation is symmetric.
Since $\varphi_0|_{\mathbf{x}A} : \mathbf{x}A \to A'$ is a monomorphism (Lemma~\ref{lemma xA mono}), $\eth(G)$ is a subset of $\mathbf{x}A$, and $\varphi_0|_{\eth(G)} : \eth(G) \to \eth(G')$ is a bijection, we know that $a\in \eth(G)$.
But then $\psi_0(a) \in \eth(G')$ since $\psi$ is active.
Since $D' \cap \eth(G') = \varnothing$, this is impossible.
Thus $\varphi_0(a) = \psi_0(a)$ for every $a\in A$.

We now turn to vertices.
Since $\varphi_0 = \psi_0$, we know by \eqref{graphical map defn: boundary compatibility} that $\eth(\varphi_1(v)) = \eth(\psi_1(v))$ for every vertex $v$. 
But $\varphi_1(v)$ is an edge if and only if $\psi_1(v)$ is an edge, so we can apply Proposition~\ref{proposition: embedding uniqueness} to deduce that $\varphi_1(v) = \psi_1(v)$ for every vertex $v$. Thus $\varphi_1 = \psi_1$.
\end{proof}

We next show that it may be the case that two graphical maps with common domain and codomain are distinct despite being identical on arcs.
This is similar to behavior that occurs for the wheeled properadic graphical category $\Gamma_\circlearrowright$ from \cite{hrybook,hry_factorizations}, where we only have determination by edge maps for the identity of isomorphisms \cite[Lemma 3.9]{hry_factorizations}. 
This behavior is either not present or can be avoided for other types of graphs, for example in the situation of the dendroidal category $\Omega$ of \cite{mw}, the unrooted tree category $\Xi$ of \cite{hry_cyclic}, and the properadic graphical category $\Gamma$ (see \cite[Corollary 6.62]{hrybook}).

\begin{example} \label{example: active map} 
Active maps are \emph{not} necessarily determined by what they do on arcs.
As an example, consider graphical maps from the loop with two vertices to the loop with one vertex, $G$.
Any graphical map between graphs without boundary is automatically active.
\[
\begin{tikzpicture}
\draw (-1, 0.3) arc (170:10:1);
\node[plain] at (-1,0) (v){$v_1$};
\node[plain] at (1,0) (w){$v_2$};
\draw (-1, -0.3) arc (-170:-10:1);

\node[empty] at (-1,.8)(a){\scriptsize{$a$}};
\node[empty] at (1,.8)(ia){\scriptsize{$ia$}};

\node[empty] at (-1,-.8)(b){\scriptsize{$b$}};
\node[empty] at (1,-.8)(ib){\scriptsize{$ib$}};

\node[empty] at (1.5,0.5) (q){};
\node[empty] at (1.5,-0.5) (q2){};
\draw (5, 0.3) arc (165:-160:1);
\node[plain] at (5,0) (v2){$w$};

\node[empty] at (5,.8)(c){\scriptsize{$c$}};
\node[empty] at (5,-.8)(ic){\scriptsize{$ic$}};

\node[empty] at (4.5,0.5) (p){};
\node[empty] at (4.5,-0.5) (p2){};

\draw [black, ->, line width=1pt, shorten >=0.03cm] (q) to (p);
\draw [black, ->, line width=1pt, shorten >=0.03cm] (q2) to (p2);

\end{tikzpicture}
\]
The set $\embeddings(G)$ has only three elements: each of the graphs $L_0 \cong \exceptionaledge$, $L_1 \cong \medstar_2$, and $G$ admit exactly two embeddings into $G$.
All other embeddings into $G$ are isomorphic to one of these six.
For each of the source graphs $H \in \{L_0, L_1, G\}$, the two embeddings $H \to G$ are isomorphic using the unique nontrivial automorphism of $H$, hence give rise to a single element $[H\to G]$ in $\embeddings(G)$.

Let us exhibit two maps $\varphi, \psi$ where $\varphi_0 = \psi_0$ is the function sending $a,ib$ to $c$ and $ia, b$ to $ic$.
On vertices, we can declare that 
\[
	\begin{gathered}
		\varphi_1(v_1) = [\iota_w : \medstar_w \hookrightarrow G] = \psi_1(v_2) \\
		\varphi_1(v_2) = [\exceptionaledge \to G] = \psi_1(v_1).
	\end{gathered}
\]
Thus $\varphi_0 = \psi_0$ but $\varphi \neq \psi$.
\end{example}

\begin{proposition}\label{proposition valence not two}
	Suppose that $\varphi, \psi : G \to G'$ are active maps and $\varphi_0 = \psi_0$.
	Further, assume that one of the following two conditions holds:
\begin{itemize}
\item the boundary of $G$ is nonempty, or
\item there exists a vertex $v$ so that neither $\varphi_1(v)$ nor $\psi_1(v)$ is an edge.
\end{itemize}
Then $\varphi = \psi$.
\end{proposition}
Since only bivalent vertices may be sent to edges, the second condition is automatic whenever $G$ has a vertex which is not bivalent.
At least one of the two conditions is guaranteed to be satisfied unless $G$ is the loop with $n$ vertices (Definition~\ref{examples combinatorial}).
At least one of the two conditions holds whenever $\varphi, \psi \in \graphicalcat^+_{\actrm}$.
\begin{proof}
	Notice that for each $v \in G$, we have $\eth(\varphi_1(v)) = \eth(\psi_1(v)) \subseteq A'$ by \eqref{graphical map defn: boundary compatibility}. 
	We would like to apply Proposition \ref{proposition: embedding uniqueness} to infer that $\varphi_1(v) = \psi_1(v)$, but it might be the case that one of these is an edge while the other isn't.
	We will show that the given conditions imply this never happens, whence the statement follows.
	Let $S_1 = \{ v \, | \, \varphi_1(v) \text{ is an edge}\}$, $S_2 = \{ v \, | \, \psi_1(v) \text{ is an edge}\}$ and $S = (S_1 \setminus S_2) \cup (S_2 \setminus S_1)$. 
	Our goal is to show that $S$ is empty.

	To this end, let $T = V \setminus (S_1 \cup S_2)$ be the set of vertices $v$ so that neither $\varphi_1(v)$ nor $\psi_1(v)$ is an edge; in particular, every vertex which is not in $T$ must be bivalent.
	Let $\gamma : V \to \mathbb{N} \cup \{-1\}$ be the function that sends a vertex to its distance from the boundary or $T$ (so $\gamma(v) = -1$ if and only if $v$ is in $T$).
	By hypothesis, at least one of $T$ and $\eth(G)$ is nonempty, so $\gamma$ is well-defined.
	Let $\gamma|_S: S \to \mathbb{N}$ be the restriction.
	
	Let $v \in S$ be minimal with respect to $\gamma|_S$, say $\gamma(v) = n$.
	Without loss of generality, assume $\varphi_1(v)$ is an edge and $\psi_1(v)$ is not an edge.
Consider a path $v=w_n, w_{n-1}, \dots, w_0$ with $\gamma(w_k) = k$ and $\nbhd(w_k) = \{ ia_{k+1}, a_k \}$.
	If $ia_0 \notin \eth(G)$, write $w_{-1} \in T$ for the vertex with $ia_0 \in \nbhd(w_{-1}).$
Write $f_k : H_k \hookrightarrow G'$ for an embedding representing $\psi_1(w_k)$.

We first note that $\varphi_1(w_k)$ is an edge for $ 0 \leq k < n$.
Otherwise, we would have that $w_k \notin S_1$, so by minimality of $n$ we would have $w_k \notin S_2$. Thus $w_k \notin S_1\cup S_2$, that is, $w_k \in T$. 
Hence $\gamma(k) = -1 < k$, a contradiction.

Since $\varphi_1(w_k)$ is an edge for $0\leq k 
\leq n$ (that is, $w_0, \dots, w_n \in S_1$), we have $\varphi_0(ia_{k+1}) = i\varphi_0(a_k) = \varphi_0(ia_k)$ for $0 \leq k \leq n$.
So we get $\varphi_0(ia_{n+1}) = \varphi_0(ia_n) = \cdots = \varphi_0(ia_0)$.

If $w_{-1} \in T$ exists, that is, if $ia_0 \notin \eth(G)$, then we are in a situation where both $\psi_1(v) = \psi_1(w_n)$ and $\psi_1(w_{-1})$ are not edges, that is, that $H_n \neq \exceptionaledge \neq H_{-1}$.
It follows that $\psi_0(ia_{n+1}) \in \nbhd(f_n u_n)$ for some vertex $u_n \in H_n$ and $\psi_0(ia_0) \in \nbhd(f_{-1} u_{-1})$ for some vertex $u_{-1} \in H_{-1}$.
But of course $\psi_0(ia_{n+1}) = \varphi_0(ia_{n+1}) = \varphi_0(ia_0) = \psi_0(ia_0)$, so $f_n u_n = f_{-1} u_{-1}$, contradicting \eqref{graphical map defn: no double vertex covering}.

	We are now in the situation where $\psi_1(w_k)$ is an edge for $0 \leq k < n$ and  $ia_0 \in \eth(G)$. 
	But then $\psi_0(ia_0) = \psi_0(ia_{n+1})  \in \nbhd(f_nu_n)$, so $\psi_0(ia_{0}) \notin \eth(G')$. 
	This contradicts the assumption that $\psi$ is active.
\end{proof}

\begin{lemma}\label{lemma_factorization_good}
Suppose $\varphi : G\rightarrow G'$ is a graphical map in $\graphicalcat_{\actrm}$.
If $\varphi$ has two decompositions $\varphi= \gamma \alpha = \delta \beta$ where $\gamma, \delta \in \graphicalcat^{+}_{\actrm}$ and $\alpha,\beta \in \graphicalcat^{-}$ 
then there is an isomorphism $z : H\rightarrow K$ making the following diagram commute: 
\[\begin{tikzcd}[sep=small] & H \arrow[dd, "z" description] \arrow[dr,"\gamma"] & \\
G\arrow[ur, "\alpha"]\arrow[dr, "\beta" swap] & & G'\\
& K  \arrow[ur, "\delta" swap] &
\end{tikzcd}\] 
\end{lemma} 

\begin{proof}
If $v$ is a vertex of $G$, then the following are equivalent:
\begin{itemize}
	\item $\varphi_1(v)$ is an edge,
	\item $\alpha_1(v)$ is an edge,
	\item $\beta_1(v)$ is an edge.
\end{itemize}
Thus $H$ is an edge if and only if $K$ is an edge. In this case we know that $\gamma$ and $\delta$ are isomorphisms by Lemma~\ref{lemma: active maps from edge are iso}, and we can take $z = \delta^{-1} \gamma$. 
From now on, suppose that $H$ and $K$ are not edges.

Since $\alpha_0$ is surjective, $\alpha_1(v)$ is either an edge or is $[\iota_w : \medstar_w \to H]$ for some $w\in V(H)$ (using the notation from Example~\ref{example vertices as embeddings}) for every vertex $v$ (and likewise for $\beta$).
Let us first define $z$ on vertices.
If $v$ is a vertex of $H$ let $\tilde v$ be the unique vertex of $G$ so that $\alpha_1(\tilde v)$ contains $v$.
Let $z(v)$ be the unique vertex in $\beta_1(\tilde v)$.
On $\nbhd(v)$, define $z$ via the following bijection
	\[
		\begin{tikzcd}
			\nbhd(v) & \eth(\alpha_1(\tilde v)) \lar["\cong" swap, "i"] & \nbhd(\tilde v) \lar["\cong" swap, "\alpha_0i"] \rar["\cong", "\beta_0i" swap] &  \eth(\beta_1(\tilde v)) \rar["\cong", "i" swap] & \nbhd(z(v)).
		\end{tikzcd}
	\]
Since $\alpha$ and $\beta$ are bijections on the boundary, we can extend $z$ to the boundary so that $z\alpha = \beta$.

At this point we know $z \alpha_0 = \beta_0$ and $\gamma_0 \alpha_0 = \delta_0 \beta_0$.
Since $\alpha_0$ is surjective, the fact that $\gamma_0 \alpha_0 = \delta_0 z \alpha_0$ implies that $\gamma_0 = \delta_0 z$.
By assumption, we know that $H$ contains at least one vertex.
For each vertex $v$ of $H$ we have $\gamma_1(v)$ and $(\delta z)_1(v)$ are not edges, so $\gamma = \delta z$ by Proposition~\ref{proposition valence not two}.
\end{proof}

\begin{proposition}\label{prop_factors_image_followed_by_outer_coface}
Every map in $\graphicalcat$ factors uniquely (up to unique isomorphism) as a map in $\graphicalcat_{\actrm}$ followed by a map in $\graphicalcat_{\embrm}$.
\end{proposition} 

\begin{proof}
Let $\varphi : G\to G'$ be a graphical map.
The existence of such a factorization is guaranteed by Lemma~\ref{lemma: image factorization}.
Consider two such factorizations $\varphi = h\alpha = f\beta$ with $\alpha,\beta$ active maps and $f,h$ embeddings.
\[
\begin{tikzcd}
	G \rar{\beta} \dar{\alpha} \arrow[dr, "\varphi" description] & G_1 \dar{f} \\
	G_2 \rar{h} & G'
\end{tikzcd}
\]
By Proposition \ref{proposition: kappa and active maps}, $\eth(f) = \eth(f\beta) =  \eth(\varphi) = \eth(h\alpha) = \eth(h)$.
Since $f$ is an embedding, then, for a vertex $v$ of $G$, we have that $\beta_1(v)$ is an edge if and only if $\varphi_1(v)$ is an edge and similarly for $\alpha$. 
Thus $\beta_1(v)$ is an edge if and only if $\alpha_1(v)$ is an edge.
Since $\beta$ is active, the graph $G_1$ is the exceptional edge if and only if $\beta_1(v)$ is an edge for every $v$ and similarly for $G_2$.
Thus $G_1 = \exceptionaledge$ if and only if $G_2 = \exceptionaledge$.

By Proposition \ref{proposition: embedding uniqueness} there is a unique isomorphism $z$ with $fz = h$.
Now we have a diagram
\[
\begin{tikzcd}
	G \rar{\beta} \dar{\alpha} & G_1 \dar{f} \\
	G_2 \rar{h}\arrow[ur, "z" description] & G'
\end{tikzcd}
\]
where the outer square commutes, as does the lower triangle.
But then $fz\alpha = h\alpha = f\beta$.
Since $z\alpha$ and $\beta$ are active maps and $f$ is an embedding, we have $z\alpha = \beta$ by Lemma~\ref{lemma mono-like property active}.
\end{proof}

\begin{theorem}
\label{theorem orthogonal factorization system}
The category $\graphicalcat$ admits a factorization system with left class $\graphicalcat_{\actrm}$ and right class $\graphicalcat_{\embrm}$.
\end{theorem}
\begin{proof}
By Proposition~\ref{prop_factors_image_followed_by_outer_coface} we know that any graphical map factors uniquely, up to unique isomorphism, as an active map followed by an embedding.
Further, the classes $\graphicalcat_{\actrm}$ and $\graphicalcat_{\embrm}$ are closed under composition and contain all isomorphisms.
Thus the conditions of \cite[Proposition 14.7]{AdamekHerrlichStrecker:ACCJC} are satisfied, and $(\graphicalcat_{\actrm}, \graphicalcat_{\embrm})$ is an orthogonal factorization system on $\graphicalcat$.
\end{proof}

\subsection{Reedy structure}
\label{subsection reedy}

A \emph{dualizable generalized Reedy structure} on a small category $\Rr$ consists of 
\begin{itemize} 
\item  wide subcategories $\Rr^+$  and $\Rr^{-}$, and 
\item  a degree function $\deg: \ob(\Rr) \rightarrow\mathbb{N}$
\end{itemize} 
satisfying five axioms from \cite[Definition 1.1]{bm_reedy}.
The subcategory $\Rr^+$ is commonly called the `direct category' and $\Rr^-$ the `inverse category.' 
Our goal for the remainder of the section is to prove Theorem~\ref{theorem reedy}, which asserts that the structure from Definition~\ref{defn graph cat reedy} constitutes a dualizable generalized Reedy structure on $\graphicalcat$.

\begin{definition}
\label{almost injective def}
We say a graphical map $\varphi : G \to G'$ is \emph{almost injective on edges} if in the decomposition $\varphi = f \varphi'$ (from Proposition~\ref{prop_factors_image_followed_by_outer_coface}), with $f$ an embedding and $\varphi'$ an active map, we have that the function $\varphi'_0$ is injective.
Write $\graphicalcat^{+}$\index{$\graphicalcat^{+}$} for the class of maps which are \emph{almost injective on edges}.
\end{definition}

It is clear that $\graphicalcat^{+}$ contains all isomorphisms, hence all identities.
This class is also closed under composition, so $\graphicalcat^+$ is a wide subcategory of $\graphicalcat$:
\begin{remark}[$\graphicalcat^+$ is subcategory]
Suppose that $f\alpha = \beta h$ with $f,h \in \graphicalcat_{\embrm}$ and $\alpha,\beta\in \graphicalcat_{\actrm}$.
Further, assume that $\beta \in \graphicalcat_{\actrm}^+$, that is, that $\beta_0$ is injective.
This implies that $\beta_1(w)$ is never an edge by Lemma~\ref{lemma degen implies not injective}, hence $(\beta h)_1(v)$ is never an edge.
Since $(\beta h)_1(v) = (f\alpha)_1(v)$ is not an edge for any vertex $v$ and $f$ is an embedding, we know that $\alpha_1(v)$ is never an edge.
This implies that the function $\mathbf{x}^{\alpha}$ from Construction~\ref{construction xA} is the identity, so by Lemma~\ref{lemma xA mono} we know $\alpha_0$ is a monomorphism.
It follows that $\graphicalcat^{+}$ is closed under composition, since if $\varphi= h \gamma$ and $\psi = k\beta$ are two composable morphisms (with specified factorizations), then $\psi \varphi = k\beta h \gamma = k f \alpha \gamma$.
\end{remark}

\begin{definition}[Generalized Reedy structure]
\label{defn graph cat reedy}
The categories $\graphicalcat^+$ and $\graphicalcat^-$ are as given in Definition~\ref{almost injective def} and Definition~\ref{def wide subcats act}, respectively.
For the latter, recall that a map $\varphi$ is in $\graphicalcat^{-}$ if and only if $\varphi_0$ is surjective and $\varphi$ is active.

Recall that an internal edge $e = [a,ia]$ is one in which $a,ia \in D$, and that $E_i \subseteq E$ denotes the set of internal edges.
For a graph $G$, define the \emph{degree of $G$} to be the sum of the number of vertices and the number of internal edges, that is \[ \deg(G)=|V|+ |E_i|.\] 
\end{definition}

\begin{proposition}\label{prop reedy one} 
Non-invertible morphisms in $\graphicalcat^+$ (respectively, $\graphicalcat^{-}$) raise (respectively, lower) the degree.
Isomorphisms in $\graphicalcat$ preserve the degree. 
\end{proposition} 
\begin{proof}
We first prove the statement about $\graphicalcat^+$.
In light of Proposition~\ref{prop_factors_image_followed_by_outer_coface}, it is enough to show that maps in $\graphicalcat_{\embrm}$ and $\graphicalcat_{\actrm}^+$ are nondecreasing on degree, and that non-isomorphisms in these wide subcategories are strictly increasing in degree.

Embeddings are injective on vertices and on the set of internal edges, so in particular are nondecreasing in degree. 
If $f : G \to G'$ is an embedding which is not a bijection on vertices, then because $f$ is nondecreasing on internal edges, $f$ is strictly increasing in degree.
Suppose that $f$ is an embedding which is not an isomorphism but is a bijection on vertices.
This cannot be the case if $G' = \exceptionaledge$.
Then $D\to D'$ is also a bijection (since $f$ is \'etale), hence $A\to A'$ is not a bijection.
The map $f: A \to A'$ is automatically a surjection, since if  $a' \in A' \setminus f(A) \subseteq \eth(G')$, then $ia' = f(d)$ for some $d \in D \cong D'$.
But then 
$a' = iia' = if(d) = f(i(d)) \in f(A),$ which is a contradiction.
Hence there exist $a_1 \neq a_2 \in A$ with $f(a_1) = f(a_2)$.
Since $D \to D'$ is injective, either $a_1 \in \eth(G)$ or $a_2 \in \eth(G).$ Without loss of generality, assume $a_1 \in \eth(G)$; by Lemma~\ref{lemma: boundary injective}, $a_2 \notin \eth(G)$.
Since $D \to D'$ is injective and $f(ia_1) = f(ia_2)$, we know $ia_2 \in \eth(G)$.
Thus in $G'$ we have created a new internal edge $[f(a_1), if(a_1)] = [f(a_2), if(a_2)]$ which does not come from an internal edge of $G$.
Thus $\deg(G) < \deg(G').$

Below we write $\varphi : G \to G'$ for a map in $\graphicalcat$ and $\varphi_v : H_v \hookrightarrow G'$ for a representative of $\varphi_1(v)$.

Suppose $\varphi$ is in $\graphicalcat_{\actrm}^+$. 
Since $\varphi_0$ is injective, $\varphi_1(v)$ is not an edge for any $v$ by Lemma~\ref{lemma degen implies not injective}.
Thus $\varphi$ is nondecreasing on number of vertices.
Also, $\varphi$ takes internal edges to internal edges.
If $\varphi$ is not an isomorphism, then there exists a vertex $v$ so that $H_v$ contains an internal edge.
Thus $\varphi$ is strictly increasing in internal edges, hence in degree. 

Now suppose $\varphi$ is in $\graphicalcat^-$.
Since $\varphi_0$ is surjective, for each $v$ we know that $H_v$ has no internal edges.
Thus $H_v$ has at most one vertex and $\varphi$ is nonincreasing in degree.
Suppose $\varphi$ is not invertible.
Then there exists a $v$ with $H_v \neq \medstar$; since $\varphi_0$ is surjective, $H_v$ has no internal edges.
It follows that $\varphi_1(v)$ is an edge, and $\varphi$ is strictly decreasing on number of vertices, hence on degree.
\end{proof}

\begin{proposition}\label{prop reedy three}
Every morphism $\varphi$ of $\graphicalcat$ factors as $\varphi = \varphi^+\varphi^-$ with $\varphi^+$ in $\graphicalcat^{+}$ and $\varphi^-$ in $\graphicalcat^-$, and this factorization is unique up to isomorphism. 
\end{proposition}

\begin{proof} 

It follows from Theorem~\ref{theorem_factorization} that we can factor every graphical map $\varphi=f \delta \varphi^-$ where $\varphi^- \in \graphicalcat^-$, $f$ is an embedding (hence in $\graphicalcat^+$), and $\delta\in\graphicalcat^+_{\actrm}$.
Setting $\varphi^+ = f \delta$, it remains to check that the decomposition is unique up to isomorphism.

Suppose that $\varphi = \alpha \beta$ with $\alpha \in \graphicalcat^+$ and $\beta \in \graphicalcat^-$. 
Factor $\alpha = h \gamma$ with $\gamma$ active and $h$ an embedding.
Since $\varphi = f (\delta \varphi^-) = h (\gamma \beta)$ are two factorizations into active followed by embedding, by Proposition~\ref{prop_factors_image_followed_by_outer_coface}, there is an isomorphism $z$ so that $fz = h$ and $\delta \varphi^- = z \gamma \beta$.
By Lemma~\ref{lemma_factorization_good}, there is an isomorphism $z'$ with $\delta z' = z \gamma$ and $z' \beta = \varphi^-$.
Further, $\varphi^+ z' = f\delta z' = fz\gamma = h\gamma = \alpha$.
\end{proof}

\begin{proposition}\label{prop reedy four}
If $\varphi: G \to G'$ is in $\graphicalcat^{-}$, $\theta \in\Iso(\graphicalcat)$, and $\theta  \varphi = \varphi$, then $\theta = \id_{G'}$.
Likewise, if $\varphi: G \to G'$ is in $\graphicalcat^{+}$, $\theta \in\Iso(\graphicalcat)$, and $\varphi\theta  = \varphi$, then $\theta = \id_{G}$. 
\end{proposition} 

\begin{proof}
For the first statement, note that the source and target of $\theta$ are both $G'$.
The assumption implies that $\theta_0 \varphi_0 = (\id_{G'})_0 \varphi_0$, with $\varphi_0$ surjective.
Thus $\theta_0 = (\id_{G'})_0$.
By Theorem \ref{theorem: injective determination at zero}, $\theta = \id_{G'}$.

For the second statement,
note that the source and target of $\theta$ are both $G$.
Write $\varphi = f \varphi'$ with $f$ an embedding and $\varphi' \in \graphicalcat^+_{\actrm}$.
Then $f \varphi' \theta = \varphi \theta = \varphi = f\varphi'$.
By Lemma~\ref{lemma mono-like property}, we have $\varphi' \theta = \varphi'$.
But $\varphi'_0$ is injective, hence $\theta_0 = (\id_{G})_0$. 
By Theorem \ref{theorem: injective determination at zero}, $\theta = \id_{G}$.
\end{proof}

\begin{theorem}\label{theorem reedy}
With the structure from Definition~\ref{defn graph cat reedy}, the graphical category $\graphicalcat$ is a dualizable generalized Reedy category.
\end{theorem} 
\begin{proof}
In Proposition~\ref{prop reedy one}, Proposition~\ref{prop reedy three}, and Proposition~\ref{prop reedy four} we have shown all of the conditions of \cite[Definition 1.1]{bm_reedy} except for $\graphicalcat^{+} \cap \graphicalcat^{-} = \textnormal{Iso}(\graphicalcat)$.
The uniqueness of the decomposition in Proposition~\ref{prop reedy three} implies the inclusion from left to right.
It is also clear that any isomorphism is in both $\graphicalcat^-$ and $\graphicalcat^+$, concluding the proof.
\end{proof}

\section{Simplicial presheaves on \texorpdfstring{$\graphicalcat$}{U}} 
\label{section simplicial presheaves on U}

The purpose of this section is to describe categories of presheaves over the graphical category $\graphicalcat$ and give circumstances under which a presheaf over $\graphicalcat$ models an up-to-homotopy modular operad.
To do so we use the language of Quillen model categories and take \cite{hirschhorn} as our standard reference. We are mostly concerned with categories of presheaves over the graphical category $\graphicalcat$ into a cofibrantly-generated model category $\modelcat$.
Such categories always admit a projective model structure where weak equivalences and fibrations are defined entry-wise in $\modelcat$.
Since our graphical category is a dualizable generalized Reedy category, the category of presheaves over $\graphicalcat$ also admits a Reedy model structure in the sense of \cite{bm_reedy}.
We will study certain (left) Bousfield localizations of these model categories in Section~\ref{section:Reedy model categories}. 

\begin{notation}
\label{notation presheaves}
Let $\mathbf{C}$ be a category.
The category of \emph{$\graphicalcat$-presheaves} in $\mathbf{C}$ is the category of contravariant functors from $\graphicalcat$ to $\mathbf{C}$. 
We denote this category by $\mathbf{C}^{\graphicalcat^{\oprm}}$.
\begin{enumerate}
\item If $X$ is a $\graphicalcat$-presheaf in $\mathbf{C}$ write the evaluation of a presheaf $X$ at a graph $G\in\graphicalcat$ as $X_{G}$\index{$X_G$}. 

\item We write $\rgc[G] \in \Set^{\graphicalcat^{\oprm}}$\index{$\rgc[G]$} for the representable presheaf at a graph $G$, that is, $\rgc[G]_{H} =\graphicalcat(H,G)$ when $H$ is in $\graphicalcat$. 
\end{enumerate} 
\end{notation} 

The Yoneda Lemma says that a map  $x:\rgc[G]\rightarrow X$ in $\Set^{\graphicalcat^{\oprm}}$ is equivalent to an element $x\in X_{G}$. Every $X\in\Set^{\graphicalcat^{\oprm}}$ is, up to isomorphism, a colimit of representables  \[ X\cong \colim  \rgc[G],\] where the colimit is indexed by the maps $\rgc[G]\rightarrow X$. 

In Section~\ref{subsection reedy} we showed that $\graphicalcat$ admits the structure of a dualizable generalized Reedy category.
We now recall the basic definitions of the Berger--Moerdijk Reedy model structure on a diagram category $\modelcat^{\Rr}$ when $\Rr$ is a generalized Reedy category,
all of which can be found just before \cite[Theorem 1.6]{bm_reedy}.
Recall that a model category $\modelcat$ is $\Rr$-projective if for every $r\in \Rr$, the category $\modelcat^{\aut(r)}$ admits a model structure whose weak equivalences and fibrations are created in $\modelcat$ (that is, by forgetting the $\aut(r)$ actions).
This is the case, for instance, if $\modelcat$ is cofibrantly generated \cite[11.6.1]{hirschhorn}.  

\begin{definition}
\label{definition reedy model structure}
Suppose that $\Rr$ is a generalized Reedy category and $\modelcat$ is an $\Rr$-projective category (for instance, if $\modelcat$ is cofibrantly generated).
\begin{enumerate}
	\item If $r$ is an object of $\Rr$, then $\Rr^-(r)$\index{$\Rr^-(r), \Rr^+(r)$} is the category whose objects are maps of $\Rr^- \setminus \Iso(\Rr)$ with domain $r$.
	\item If $r$ is an object of $\Rr$, then $\Rr^+(r)$ is the category whose objects are maps of $\Rr^+ \setminus \Iso(\Rr)$ with codomain $r$.
\end{enumerate}
Let $X : \Rr \to \modelcat$ be a functor, and let $r$ be an object of $\Rr$.
\begin{enumerate}
	\item The $r$-th matching map is defined to be the map
	\[
		X_r \to \lim_{\Rr^-(r)} X = M_rX
	\]\index{$M_rX, L_rX$}
	whose codomain is the $r$-th matching object.
	\item The $r$-th latching map is defined to be the map
	\[
		L_r X = \colim_{\Rr^+(r)} X \to  X_r 
	\]
	whose domain is the $r$-th latching object.
\end{enumerate}
Finally, let $f: X \to Y$ be a map in $\modelcat^{\Rr}$.
\begin{enumerate}
	\item If the relative matching map $X_r\to M_rX \times_{M_rY} Y_r$ is a fibration in $\modelcat$ for every $r$, then $f$ is called a \emph{Reedy fibration}.
	\item If the relative latching map $X_r \cup_{L_rX} L_r Y \to Y_r$ is a cofibration in $\modelcat^{\aut(r)}$ for every $r$, then $f$ is called a \emph{Reedy cofibration}.
\end{enumerate}
\end{definition}

The theorem of Berger and Moerdijk asserts that $\modelcat^{\Rr}$ admits a model structure with these (co)fibrations and with levelwise weak equivalences.
If $\Rr$ happens to be `dualizable', then $\Rr^{\oprm}$ is also a generalized Reedy category, so $\modelcat^{\Rr^{\oprm}}$ will inherit such a Berger--Moerdijk Reedy model structure.
For concision, we refer to this model structure as the \emph{Reedy model structure} in what follows.
The most important special case for us is $\sSet^{\graphicalcat^{\oprm}}$ where $\sSet$ has the Kan--Quillen model structure.
We will say that a presheaf $X\in \sSet^{\graphicalcat^{\oprm}}$ is `Reedy fibrant' (respectively, `Reedy cofibrant') if it is fibrant (respectively, cofibrant) in the Reedy model structure. 

\subsection{The Segal core of a graph} 
\label{section segal core}
Suppose that $G$ is a graph containing at least one vertex, and recall the beginning of Construction~\ref{construction graph sub}, where we exhibited $G$ as a coequalizer (in the diagram category $\finset^{\mathscr{I}}$).
We can form a corresponding coequalizer in $\Set^{\graphicalcat^{\oprm}}$, 
\[ \begin{tikzcd}
\coprod\limits_{e \in E_i} \rgc[\exceptionaledge] \rar[shift left, "\outeredge"] \rar[shift right, "\inneredge" swap] & \coprod\limits_{v\in V} \rgc[\medstar_v] \rar & \segalcore[G],
\end{tikzcd} \]
and we call the target the \emph{Segal core}\index{$\segalcore[G]$} of $G$ (which should not be confused with the graph $\core(G)$ from Definition~\ref{definition: edge subtraction}).
It comes with a map $\segalcore[G] \to \rgc[G]$ induced by $\coprod \iota_v : \coprod_{v\in V} \rgc[\medstar_v] \to \rgc[G]$.
In the case when $G = \exceptionaledge$, we declare the map $\segalcore[G] \to \rgc[G]$ to be the identity map on $\rgc[G]$.

Notice that the object $\segalcore[G]$ does not depend upon the choices we made for the orderings of the internal edges of $G$.
Indeed, any two such choices yield isomorphic results via a unique isomorphism of coequalizer diagrams utilizing only the involution on $\exceptionaledge$.

\begin{remark}[Alternative description]
\label{remark alternate description of segal core}
Suppose that $G$ is a graph with at least one vertex, where we've made choices about orderings of each internal edge as in Construction~\ref{construction graph sub} (whose notation we freely use). 
Define a new category $\mathbf{C}^G$ with object set $E_i \amalg V$.
The non-identity morphisms in this category are precisely the set of arcs comprising the internal edges (that is, the set of arcs of $\core(G)$), so that an internal arc $x$ goes from the internal edge $[x]\in E_i$ associated to $x$ to the vertex $tx \in V$. 
There is a functor $\mathbf{C}^G \to \Set^{\graphicalcat^{\oprm}}$ such that
\begin{align*}
	x_e^1 : e \to tx_e^1 &\qquad\text{ maps to }\qquad \outeredge_e : \rgc[\exceptionaledge] \to \rgc[\medstar_{tx_e^1}] & \text{and}\\
x_e^2 : e \to tx_e^2 &\qquad\text{ maps to }\qquad \inneredge_e : \rgc[\exceptionaledge] \to \rgc[\medstar_{tx_e^2}].
\end{align*}
The colimit of this functor is $\segalcore[G]$.
\end{remark}

There is an inclusion $\Set^{\graphicalcat^{\oprm}} \hookrightarrow \sSet^{\graphicalcat^{\oprm}}$ coming from the inclusion $\Set \to \sSet$, and we use this to consider $\segalcore[G]$ and $\rgc[G]$ as objects in $\sSet^{\graphicalcat^{\oprm}}$. 

\begin{lemma}
\label{cofibration edges to star lemma}
As in Example~\ref{examples combinatorial}, let $\exceptionaledge$ have arc set $\eearcs$ and $\medstar_n$ have $D(\medstar_n) = \{ 1, \dots, n\}$, $\eth(\medstar_n) = \{ 1^\dagger, \dots, n^\dagger\}$. 
Let $h_k : {\exceptionaledge} \to \medstar_n$ be the embedding sending $\edgemajor$ to $k$.
If $S$ is any subset of $\{ 1, \dots, n\}$, then the map
\[
	\coprod_{k\in S} h_k \colon \coprod_{k\in S} \rgc[\exceptionaledge] \to \rgc[\medstar_n]
\]
is a cofibration in the generalized Reedy model structure on $\sSet^{\graphicalcat^{\oprm}}$
\end{lemma}
Note that the same statement then holds when $h_k$ is replaced by the embedding $h_k' = h_kz : {\exceptionaledge} \to \medstar_n$ sending $\edgemajor$ to $k^\dagger$ (where $z$ is the unique nontrivial automorphism of ${\exceptionaledge}$).
\begin{proof}
For the purposes of this proof, we can take $S= \{ 1, \dots, n \}$ since $\rgc[\exceptionaledge]$ is Reedy cofibrant.
Let $Z$ be any object of $\sSet^{\graphicalcat^{\oprm}}$.
Then the map
\begin{equation}\label{equation matching object Z}
	\prod_{k=1}^n Z_{h_k} \colon Z_{\medstar_n} \to \prod_{k=1}^n Z_{\exceptionaledge}
\end{equation}
is a concrete realization of the map $Z_{\medstar_n} \to M_{\medstar_n}(Z)$ to the matching object.
This occurs because $\medstar_n$ has degree 1, and the only degree 0 object in $\graphicalcat$ is $\exceptionaledge$.
We then have a splitting of categories (see Definition~\ref{definition reedy model structure})
\[
	(\graphicalcat^{\oprm})^{-}(\medstar_n) = \left[ \graphicalcat^{+}(\medstar_n) \right]^{\oprm} = \coprod_{k=1}^n \{ h_k \overset\cong\leftrightarrow h_{k}' \} \simeq \coprod_{k=1}^n \{ h_k \}
\]
where each groupoid $\{ h_k \overset\cong\leftrightarrow h_{k}' \}$ has two objects together with a unique isomorphism between them.
As we can use the discrete category on the right to compute the limit expressing the matching object, we see \eqref{equation matching object Z} does indeed model this matching map.

Suppose that $X \to Y$ is an acyclic Reedy fibration in $\sSet^{\graphicalcat^{\oprm}}$.
Diagrams 
\begin{equation}\label{diagram needing a lift} \begin{tikzcd}
\coprod\limits_{k=1}^n  \rgc[\exceptionaledge]  \rar \dar & X \dar \\ 
 \rgc[\medstar_n] \rar & Y
\end{tikzcd} \end{equation}
correspond to vertices $y\in Y_{\medstar_n}$ and $\ell \in \prod_{k=1}^n X_{\exceptionaledge}$ mapping to the same vertex of $\prod_{k=1}^n Y_{\exceptionaledge}$ in the diagram
\[ \begin{tikzcd}
X_{\medstar_n} \rar \dar & Y_{\medstar_n} \dar \\
\prod\limits_{k=1}^n X_{\exceptionaledge} \rar & \prod\limits_{k=1}^n Y_{\exceptionaledge}
\end{tikzcd} \]
whose vertical maps are induced by the $h_k$.
A lift for \eqref{diagram needing a lift} is the same thing as a vertex $x\in X_{\medstar_n}$ which maps to both $y$ and $\ell$. 

Since the only objects of $\graphicalcat$ of degree less than or equal to $1$ are edges and stars, Proposition 5.7 of \cite{bm_reedy} implies that 
\[
	X_{\medstar_n} \to M_{\medstar_n}(X) \times_{M_{\medstar_n}(Y)} Y_{\medstar_n}
\]
is an acyclic fibration of simplicial sets, hence surjective (in particular, on vertices).
By our calculation of $M_{\medstar_n}(Z)$ as \eqref{equation matching object Z}, we then know that every diagram \eqref{diagram needing a lift} admits a lift $\rgc[\medstar_n] \to X$.
Since $\coprod_{k=1}^n h_k$ lifts against all acyclic Reedy fibrations, it is a Reedy cofibration.
\end{proof}

As one can see from the proof of the preceding lemma, one does not expect these maps to be \emph{projective} cofibrations (that is, in the projective model structure, where acyclic fibrations are the levelwise acyclic fibrations). 
Therefore, we expect the projective version of the following proposition to be false in general.

\begin{proposition}
\label{proposition segal core is cofibrant}
If $G \in \graphicalcat$, then $\segalcore[G]$ is Reedy cofibrant in $\sSet^{\graphicalcat^{\oprm}}$.
\end{proposition}
\begin{proof}
The map
\[
	\inneredge : \coprod_{e \in E_i} \rgc[\exceptionaledge] \to \coprod_{v\in V} \rgc[\medstar_v]
\]
is the coproduct (over $V$) of maps, each of which isomorphic to one from Lemma~\ref{cofibration edges to star lemma}.
This is because $\inneredge$ restricts to a monomorphism 
\[
	\coprod_{E_i} \{ \edgemajor \} \to \coprod_{v\in V} \nbhd(v),
\]
and for each $v$ we can consider $S = \inneredge(\coprod_{E_i} \{ \edgemajor \}) \cap \nbhd(v) \subseteq D(\medstar_v)$.
A similar argument (using $\edgeminor$ instead of $\edgemajor$) shows that $\outeredge$ is isomorphic to a coproduct of maps from Lemma~\ref{cofibration edges to star lemma}.
Hence both $\outeredge$ and $\inneredge$ are cofibrations.
As the pushout of a cofibration is a cofibration, all of the maps in the defining diagram for $\segalcore[G]$
\[ \begin{tikzcd}
\coprod\limits_{e \in E_i} \rgc[\exceptionaledge] \rar{\outeredge} \dar{\inneredge} & \coprod\limits_{v\in V} \rgc[\medstar_v] \dar \\
\coprod\limits_{v\in V} \rgc[\medstar_v] \rar & \segalcore[G]
\end{tikzcd} \]
are cofibrations, so $\varnothing \to \coprod_{E_i} \rgc[\exceptionaledge] \to \segalcore[G]$ is a composition of cofibrations.
\end{proof}

\subsection{A Segal model for up-to-homotopy modular operads}\label{section:Reedy model categories}
Recall the dualizable generalized Reedy structure on $\graphicalcat$ from Section~\ref{subsection reedy}.
The Reedy model structure on $\sSet^{\graphicalcat^{\oprm}}$ is simplicial with mapping objects $\map(X,Y) \in \sSet$.\index{$\map, \map^h$} 
We will also utilize homotopy function complexes, denoted by $\map^h(X,Y) \in \sSet$, and do not insist upon a particular model for these.

\begin{definition}[Segal modular operads]\label{defn mon segal obj}
Suppose that $X$ is an object of $\sSet^{\graphicalcat^{\oprm}}$.
\begin{itemize}
\item The presheaf $X$ will be called \emph{(weakly-) monochrome} if $X_{\exceptionaledge}$ is weakly contractible.
\item The presheaf $X$ is said to \emph{satisfy the (weak) Segal condition} if for each $G\in \graphicalcat$, the Segal map
\[ X_G = \map(\rgc[G], X) \rightarrow \map(\segalcore[G],X)\] 
is a weak equivalence of simplicial sets.
\end{itemize}
If $X$ is Reedy fibrant, monochrome, and satisfies the Segal condition, then we call $X$ a \emph{(monochrome) Segal modular operad}.
\end{definition} 
The purpose of this subsection is to point out that there is a model category whose fibrant objects are precisely the Segal modular operads.
In the companion paper \cite{modular_paper_two}, we give a precise definition of (colored) modular operads (called compact symmetric multicategories in \cite{JOYAL2011105}) and prove Theorem~\ref{nerve theorem paper one intro}.
Segal modular operads should be thought of as one-colored modular operads where all of the structure is only defined up to coherent homotopy.
At the end of Section~\ref{section: stable graphs}, we provide potential examples which should be adaptable to give non-strict examples of Definition~\ref{defn mon segal obj}.

\begin{remark}
\label{remark on nerve of modular operads}
If $X_{\exceptionaledge}$ is a point, instead of just being weakly equivalent to a point, then $\map(\segalcore[G], X)$ is isomorphic to the product
\[
	\prod_{v\in G} X_{\medstar_v} \cong \map(\segalcore[G], X),
\]
and the $v$th projection of the map from $X_G$ to this product is induced from $\iota_v : \medstar_v \to G$.
In this case, the Segal map being a weak equivalence tells us $X$ should be determined by its value on vertices.

Suppose that $P$ is a modular operad in $\sSet$ (in the sense of \cite{modular_paper_two}) whose color set has just one element.
If $X = NP$ is the \emph{nerve} of $P$, then $X_{\exceptionaledge}$ is a point and the Segal map $X_G \to \prod X_{\medstar_v}$ is an isomorphism for every $G$.
Both a precise construction of $N$ and a proof of this fact are provided in the companion paper \cite{modular_paper_two}.
Thus every one-colored modular operad gives rise to a monochrome presheaf that satisfies the Segal condition.
Note, however, that Reedy fibrancy requires more assumptions on the modular operad $P$.
\end{remark}

\begin{theorem}\label{theorem localized reedy model structure} 
The category $\sSet^{\graphicalcat^{\oprm}}$ admits a cofibrantly generated model structure whose fibrant objects are the Segal modular operads.
\end{theorem}  

\begin{proof}
The Reedy model structure on $\sSet^{\graphicalcat^{\oprm}}$ is left proper and cellular by \cite[Theorem 7.2 \& Proposition 7.4]{hry_cyclic}.
Thus if $S$ is any set of maps, we may apply left Bousfield localization \cite[Theorem 4.1.1]{hirschhorn} to obtain a localized model structure $\mathscr{L}_S\sSet^{\graphicalcat^{\oprm}}$ with the same underlying category.
We specialize to the case when $S$ is the set of maps consisting of the Segal core inclusions $\segalcore[G] \hookrightarrow \rgc[G]$ as well as the unique map $\varnothing \to \rgc[\exceptionaledge]$. 
Here, we are using the inclusion $\Set^{\graphicalcat^{\oprm}} \hookrightarrow \sSet^{\graphicalcat^{\oprm}}$ to regard these set-valued presheaves as simplicial set-valued presheaves.

To complete the proof, we only need to characterize the fibrant objects in this localized model structure.
As with any left Bousfield localization, these are the objects $X$ so that $X$ is fibrant in the original model structure and $\map^h(s, X)$ is a weak equivalence of simplicial sets for every $s\in S$.
In other words, we must characterize those Reedy fibrant $X$ so that (for all $G$)
\begin{align*}
	\map^{h}(\rgc[G], X) &\to \map^h(\segalcore[G], X) \\
	\map^{h}(\rgc[\exceptionaledge], X) &\to \map^h(\varnothing, X)
\end{align*}
are weak equivalences of simplicial sets.

In any simplicial model category, if $A$ is cofibrant and $Z$ is fibrant, then $\map^h(A,Z)$ is weakly equivalent to $\map(A,Z)$ by \cite[Corollary 4.7]{DwyerKan:FCHA}. 
Note that $\rgc[G]$, $\varnothing$, and $\segalcore[G]$ are all cofibrant in $\sSet^{\graphicalcat^{\oprm}}$ (the last of these by Proposition~\ref{proposition segal core is cofibrant}), which is a simplicial model category.
Further, $\map(\rgc[G], X) = X_G$ for any presheaf $X$.
Rephrasing the condition for fibrancy in $\mathscr{L}_S\sSet^{\graphicalcat^{\oprm}}$ gives that $X$ is fibrant if and only if 
\begin{itemize}
	\item $X$ is Reedy fibrant,
	\item $X_G \to \map(\segalcore[G], X)$ is a weak equivalence of simplicial sets for all $G$, and 
	\item $X_\exceptionaledge \to \map(\varnothing, X) = \Delta[0]$ is a weak equivalence.
\end{itemize}
Thus $X$ is fibrant if and only if it is a Segal modular operad.
\end{proof}

\begin{lemma}\label{lemma: qe reedy projective}
	Suppose that $\Rr$ is a generalized Reedy category and $\modelcat$ is a cofibrantly generated model category. 
	Write $\modelcat^{\Rr}_{\reedysub}$ for the diagram category with the Berger--Moerdijk Reedy model structure \cite{bm_reedy} and $\modelcat^{\Rr}_{\projsub}$ for the same category with the projective model structure \cite[Theorem 11.6.1]{hirschhorn}.
	Then the identity functor
	\begin{equation}\label{equation unlocalized qe}
	\begin{tikzcd}
		\modelcat^{\Rr}_{\projsub} \arrow[r, shift left="1.5"] & \arrow[l, shift left="1.5", "\scriptscriptstyle\perp" swap] \modelcat^{\Rr}_{\reedysub}
	\end{tikzcd}\end{equation}
	is a Quillen equivalence.

	Assume further that $\modelcat$ is left proper and cellular. 
	If $S$ is any set of maps in $\modelcat^{\Rr}$ and $\mathscr{L}_S$ denotes left Bousfield localization at $S$ \cite[Definition 3.3.1]{hirschhorn}, then $\mathscr{L}_S\modelcat^{\Rr}_{\reedysub}$  and $\mathscr{L}_S\modelcat^{\Rr}_{\projsub}$ have the same class of weak equivalences.
	In particular, \eqref{equation unlocalized qe} remains a Quillen equivalence after left Bousfield localization at $S$.
\end{lemma}
\begin{proof}
	It is immediate that \eqref{equation unlocalized qe} is a Quillen adjunction since each cofibration in the projective model structure is a cofibration in the Reedy model structure, and each fibration in the Reedy model structure is a fibration in the projective model structure.
	Since the two model structures have the same class of weak equivalences, \eqref{equation unlocalized qe} is a Quillen equivalence.

It remains to show that the localized model structures have the same class of weak equivalences. 
Suppose that $W$ is any object, $W\to \widehat{W}$ is a Reedy fibrant replacement of $W$, and $f : A \to B$ is any morphism in $\modelcat^{\Rr}$.
We then have the following commutative diagram of homotopy function complexes.
\begin{equation}\label{maph comm diagram}
\begin{tikzcd}
\map^h(B,W) \rar \dar{\simeq} & \map^h(A,W) \dar{\simeq} \\
\map^h(B,\widehat{W}) \rar & \map^h(A,\widehat{W})
\end{tikzcd}
\end{equation}
The vertical maps in this diagram are weak equivalences using \cite[17.6.3]{hirschhorn}.

If $W$ is a projective $S$-local object, then $\widehat{W}$ is a Reedy $S$-local object.
To see this, notice that if $f : A\to B$ is any element of $S$, then the top map of \eqref{maph comm diagram} is a weak equivalence by assumption, which implies that the bottom map is as well.

Now suppose that $f : A \to B$ is a Reedy local equivalence and $W$ is a projective $S$-local object.
Since we know that $\widehat{W}$ is a Reedy $S$-local object, we have that the bottom map of \eqref{maph comm diagram} is an equivalence, hence the top map is as well.
Since $W$ was an arbitrary projective $S$-local object, this implies that $f$ is a projective $S$-local equivalence.

On the other hand, suppose that $f : A \to B$ is a projective $S$-local equivalence.
Any Reedy $S$-local object $W$ is automatically a projective $S$-local object (since every Reedy fibrant object is also projectively fibrant).
Hence $\map^h(A, W) \leftarrow \map^h(B,W)$ is an equivalence.
Since the Reedy $S$-local object $W$ was arbitrary, this implies that $f$ is a Reedy $S$-local equivalence.
\end{proof}

\begin{proposition}\label{prop localized projective model structure}
There exists a model category structure on $\sSet^{\graphicalcat^{\oprm}}$ so that an object $X$ is fibrant if and only if 
\begin{itemize} 
\item $X_G$ is fibrant for all graphs $G$, 
\item $X_{\exceptionaledge} \simeq *$, and 
\item for all graphs $G$, the Segal map \[ X_G = \map(\rgc[G], X) \simeq \map^h(\rgc[G], X) \rightarrow \map^h(\segalcore[G],X)\] is a weak equivalence of simplicial sets.
\end{itemize} 
Furthermore, this model structure is Quillen equivalent (via the identity functor) to the model structure from Theorem~\ref{theorem localized reedy model structure}.
\end{proposition}
\begin{proof}
	The proof of the first part is the same as in Theorem~\ref{theorem localized reedy model structure}, except we start with the projective model structure on $\sSet^{\graphicalcat^{\oprm}}$ instead of the Reedy model structure.
	The second statement is a direct application of Lemma~\ref{lemma: qe reedy projective}.
\end{proof}

\section{Variations on the modular graphical category}\label{section: variations}
We now discuss two variations on the graphical category $\graphicalcat$.
The first of these essentially just adds in a single object, the nodeless loop.
We've postponed the introduction of the nodeless loop until now partly because it allows us to use a cleaner definition of graph in the early parts of the paper, and because we could avoid addressing many special cases throughout.
Further, from the point of view of Segal modular operads, the value of a presheaf at the nodeless loop should be indistinguishible (up to homotopy) from the value at the exceptional edge.
That said, the extended graphical category comes in handy in \cite{modular_paper_two} when proving that each Segal presheaf has an associated modular operad. 

The second variation we address is related to the original definition \cite{MR1601666} of modular operads, where the underlying collections had an additional genus grading.
Modular operads in this sense satisfied a geometric condition called stablity.
In Section~\ref{section: stable graphs} we modify $\graphicalcat$ to have objects graphs which have a genus labeling on each vertex which satisfies a stability condition.

\subsection{The extended graphical category}\label{subsection: egc}

Inspired by Remark~\ref{remark nondetermination boundary}, we drop the assumption that graphs have boundary exactly equal to $A\setminus D$.
The following extension allows us to express the nodeless loop from Example~\ref{examples combinatorial}.

\begin{definition}\label{definition unsafe graphs}
A \emph{graph} $G$ consists of 
\begin{itemize}
	\item a diagram of finite sets
\[
\begin{tikzcd}
	A \arrow[loop left, "i"] & \lar[swap]{s} D \rar{t} & V
\end{tikzcd}
\]
where $i$ is a fixedpoint-free involution and $s$ is a monomorphism, and
\item a subset $\eth(G) \subseteq A$ so that 
\begin{enumerate}
	\item $iD \setminus D \subseteq \eth(G) \subseteq A \setminus D$, and
	\item $\eth(G) \setminus iD$ is an $i$-closed subset of $A$.
\end{enumerate}
\end{itemize}
The subset $\eth(G)$ is called the \emph{boundary} of $G$.
If the boundary $\eth(G)$ is maximal, that is, if $\eth(G) = A \setminus D$, then we say that $G$ is \emph{safe}. 
\end{definition}

For the rest of this subsection, the graphs from Definition~\ref{definition jk graphs} are the safe graphs, while other graphs may be referred to as \emph{unsafe}.
But what are these unsafe graphs? 
Before answering this question fully, let us give an example that we could not quite include in Example~\ref{examples combinatorial}. 
\begin{definition}[Nodeless loop]\label{def C zero}
The \emph{loop with zero vertices} is the graph with $A=2\{0\} = \{0,0^\dagger\}$ and $D = V = \eth = \varnothing$.
Any graph isomorphic to this one will be called a \emph{nodeless loop}.
\end{definition}

\begin{remark}
Nodeless loops are not Feynman graphs in the sense of Definition~\ref{definition jk graphs}. 
As a result, the monad for compact symmetric multicategories in \cite{JOYAL2011105} is not well-defined at level $n=0$. 
We investigate this issue in more depth in the companion paper \cite{modular_paper_two}, and Sophie Raynor gives another approach in \cite{Raynor:DLCSM}.
Nodeless loops are not necessary for non-unital variations of modular operads, which explains their omission from \cite{MR1601666}.
\end{remark}

Let us return to the question at hand.
Given two graphs $G$ and $H$, we can form a new graph $G \amalg H$ by taking the coproduct of the underlying functors in $\finset^{\mathscr{I}}$ (see Definition~\ref{definition I and natural trans}) and declaring that $\eth(G \amalg H) = \eth(G) \amalg \eth(H)$.
A graph will be called \emph{connected} if it is nonempty and cannot be decomposed nontrivially via $\amalg$ (equivalently, if the underlying object in $\finset^{\mathscr{I}}$ is connected).
A graph is safe if, and only if, all of its connected components are safe.
In other words, unsafe graphs are precisely those graphs that have at least one unsafe connected component.
The conditions in Definition~\ref{definition unsafe graphs} imply that $A \setminus (\eth(G) \amalg D)$ is $i$-closed for any graph $G$.
If $G$ is unsafe, then $A \setminus (\eth(G) \amalg D)$ contains at least one element $x$, whence it also contains $ix$.
Thus $G$ contains the nodeless loop with arc set $\{ x, ix \}$ as a summand.

\begin{remark}
The only unsafe, connected graphs are nodeless loops.
A graph is unsafe precisely when it contains at least one nodeless loop as a summand.
\end{remark}

\begin{proposition}
Isomorphism classes of graphs from Definition~\ref{definition unsafe graphs} are in one-to-one correspondence with Yau--Johnson graphs \cite{yj15}. \qed
\end{proposition}
\begin{proof}
Couple \cite[Proposition 15.6]{batanin-berger} with a minor variation of \cite[Proposition 15.2]{batanin-berger}.
\end{proof}

We now adapt \'etale maps (Definition~\ref{definition etale}) and embeddings (Definition~\ref{def: embedding}) to the present context.
\begin{definition}
Suppose that $G$ and $G'$ are (possibly unsafe) graphs.
\begin{itemize}
	\item An \emph{\'etale map} $G\to G'$ is a morphism of underlying objects in $\finset^{\mathscr{I}}$ 
\begin{equation*}
\begin{tikzcd}
	A \arrow[loop left, "i"]\dar & \lar[swap]{s} D\dar \rar{t} & V\dar \\
	A' \arrow[loop left, "i'"] & \lar[swap]{s'} D' \rar{t'} & V'
\end{tikzcd}
\end{equation*}
so that 
\begin{itemize}
	\item the right-hand square is a pullback, and
	\item the set $A\setminus (\eth(G) \amalg D)$ maps into $A' \setminus (\eth(G') \amalg D')$.
\end{itemize}
\item An \emph{embedding} $G\to G'$ is an \'etale map where $V\to V'$ is a monomorphism.
\end{itemize}
\end{definition}

If $G$ is safe, then $\eth(G) \amalg D = A$, so the second condition for \'etale maps is automatically satisfied.
We have not added too many embeddings:
\begin{itemize}
	\item If $G$ is a nodeless loop and $G \to G'$ is an embedding, then $G'$ is also a nodeless loop.
	\item If $G'$ is a nodeless loop and $G\to G'$ is an embedding, then $G$ is either an exceptional edge or a nodeless loop.
\end{itemize}

We now adapt Definition~\ref{def: graphical map} to our more general class of connected graphs (that includes nodeless loops).
A more hands-on description follows in Remark~\ref{egc defn explicit}.
\begin{definition}\label{egc defn}
The extended graphical category $\egc$\index{$\egc$} has objects the connected graphs from Definition~\ref{definition unsafe graphs}.
A morphism $\varphi : G\to G'$ consists of 
	\begin{itemize}
		\item A map of involutive sets $\varphi_0 : A \to A'$ 
		\item A function $\varphi_1 : V \to \embeddings(G')$
	\end{itemize}
satisfying \eqref{graphical map defn: no double vertex covering}, \eqref{graphical map defn: boundary compatibility}, and 
\begin{enumerate}[label={({\roman*}')},ref={\thetheorem.\roman*'},start=3]
	\item If the boundary of $G$ is empty and $\varphi_1(v)$ is an edge for every $v$, then $G'$ is a nodeless loop. \label{egc defn: collapse condition}
\end{enumerate}
\end{definition}
Composition is defined essentially as in Definition~\ref{graphical map composition}.

Condition \eqref{egc defn: collapse condition} implies that if $G$ is a nodeless loop and $G\to G'$ is a map, then $G'$ is also a nodeless loop.
On the other hand, if $G'$ is a nodeless loop then the set $\embeddings(G')$ has precisely two elements. 
In this case, a map $\varphi: G \to G'$ is entirely determined by $\varphi_0$.
Associativity of composition in $\egc$ then follows from Theorem~\ref{theorem graphicalcat is a category} and associativity of composition in the category of involutive sets.

\begin{remark}\label{egc defn explicit}
By comparing \eqref{graphical map defn: collapse condition} and \eqref{egc defn: collapse condition}, we see that $\graphicalcat$ is a full subcategory of $\egc$.
Further, if $G \to G'$ is a map and $G' \in \graphicalcat$, then $G\in \graphicalcat$, i.e. $\graphicalcat$ is a sieve on $\egc$ (as in Proposition~\ref{proposition sieve}).
Let $K$ denote a nodeless loop.
We then have
\begin{align*}
	|\egc (K, G)| &= \begin{cases}
		0 & \text{if } G \in \graphicalcat \\
		2 & \text{if } G \text{ is a nodeless loop}
	\end{cases} \\
	|\egc(G, K)| &=\begin{cases}
		2 & \text{if each vertex of $G$ has valence two,} \\
		1 & \text{if $A(G)$ is empty,} \\
		0 & \text{if $G$ contains a vertex of valence different from $0$ or $2$.}
	\end{cases}
\end{align*}
In the cases where these sets are nonempty, they are identified with $\hom(A(K),A(G))$, respectively $\hom(A(G),A(K))$.
Essentially only the linear graphs $L_n$, the isolated vertex $\medstar_0$, and the loops with $n$ vertices (including nodeless loops) admit maps to a nodeless loop.
\end{remark}

\begin{theorem}
\label{egc factorization theorem}
The category $\egc$ admits a factorization system extending that on $\graphicalcat$ from Theorem~\ref{theorem orthogonal factorization system}.
\end{theorem}
\begin{proof}[Sketch of Proof]
Let $K$ be a fixed nodeless loop.
The right class $\egc_{\embrm}$ consists of embeddings.
It contains two maps $\exceptionaledge \to K$, two maps $K\to K$, as well as all maps isomorphic to these and all maps in $\graphicalcat_{\embrm}$.
The left class $\egc_{\actrm}$ is obtained from $\graphicalcat_{\actrm}$ by adding in the unique map $\medstar_0 \to K$, the maps from loops with $n$ vertices to $K$, and all maps isomorphic to these.
Since $\graphicalcat$ is a sieve, we only need to check factorizations and uniqueness of such on maps whose codomain is $K$. 
There are only a few such cases and this is routine.
\end{proof}

Likewise, a version of Theorem~\ref{theorem reedy} is true for $\egc$.

\begin{theorem}
The category $\egc$ admits the structure of a dualizable generalized Reedy category.
\end{theorem}
\begin{proof}[Sketch of Proof]
The degree function must be modified from that in Definition~\ref{defn graph cat reedy}, and is essentially given in \cite[Definition 3.2]{hry_factorizations}.
The exceptional edge $\exceptionaledge$ has degree $0$, while the isolated vertex $\medstar_0$ has degree $1$.
For all other graphs, the degree is given by the formula $|V| + |E_i| + 1$, i.e., an increase of one from the usual degree.
We emphasize that the nodeless loop has degree $2$.

Let $K$ be a nodeless loop.
We describe $\egc^-$ and $\egc^+$ up to isomorphisms.
The inverse category $\egc^-$ consists of maps in $\graphicalcat^-$ and maps from loops with $n$ vertices to $K$.
The direct category $\egc^+$ consists of maps in $\graphicalcat^+$, $\exceptionaledge \to K$, $K \to K$, and $\medstar_0 \to K$. 
As in Theorem~\ref{egc factorization theorem}, the analogue of Proposition~\ref{prop reedy three} may be proved by factoring only those maps with codomain $K$.
\end{proof}

\begin{definition}
Let $\iota : \graphicalcat \to \egc$ denote the inclusion functor.
\begin{itemize}
\item If $G$ is a safe graph, then the \emph{Segal core inclusion} is just the left Kan extension
\[
	\iota_! (\segalcore[G] \hookrightarrow \rgc[G])
\]
of the usual Segal core inclusion (Section~\ref{section segal core}).
We use the same notation for the domain, writing this as $\segalcore[G] \to \regc[G]$.
\item The Segal core inclusion for a nodeless loop $K$ is
\[
	\regc[\exceptionaledge] = \segalcore[K] \hookrightarrow \regc[K].
\]
\item A $\egc$-presheaf $X$ is said to \emph{satisfy the strict Segal condition} if $\hom(-,X)$ sends every Segal core inclusion to a bijection of sets.
\end{itemize}
\end{definition}

\begin{theorem}
\label{theorem segal to segal}
If $X\in \Set^{\graphicalcat^{\oprm}}$ is Segal, then its right Kan extension $\iota_* X \in \Set^{\egc^{\oprm}}$ is also Segal.
\end{theorem}
\begin{proof}
As $\graphicalcat$ is a full subcategory of $\egc$, we have that $(\iota_* X)_G = X_G$ for every safe graph $G$, so the Segal condition holds at safe graphs $G$.
Thus we must show that $(\iota_* X)_K \to (\iota_* X)_{\exceptionaledge} = X_{\exceptionaledge}$ is a bijection when $K$ is a nodeless loop.

Write $C_n$ for the loop with $n$ vertices ($n\geq 1$) from Example~\ref{examples combinatorial}, all of which are safe graphs.
We also write $C_0$ for the loop with zero vertices from Definition~\ref{def C zero}. 
We restrict $\graphicalcat$ and $\egc$ to skeletal full subcategories $\mathcal{A} \subseteq \graphicalcat$ and $\widetilde{\mathcal{A}} \subseteq \egc$ whose objects are $\medstar_0$, $L_n$ for $n\geq 0$, and $C_m$ for $m\geq 1$ (resp. $m\geq 0$).
Every map in $\egc$ with source or target $C_0$ is isomorphic to a map in $\widetilde{\mathcal{A}}$, so it is sufficient to restrict $X$ to $\mathcal{A}^{\oprm}$ and examine its right Kan extension along $\iota : \mathcal{A}^{\oprm} \to \widetilde{\mathcal{A}}^{\oprm}$.

The arc set of every object in $\widetilde{\mathcal{A}}$ is of the form $2\{0,\dots, n\}$ or $2\{1,\dots, m\}$, and we say that a morphism $\varphi \in \widetilde{\mathcal{A}}$ is \emph{oriented} if $\varphi_0$ is of the form $2f$.
That is, a map $\varphi$ is oriented if it satisfies the condition that if $j$ is an integer then $\varphi_0(j)$ is not of the form $k^\dagger$ for an integer $k$.
Let $\mathcal{B} \subseteq \mathcal{A}$ denote the category with objects $L_n$ ($n\geq 0$) and $C_m$ ($m\geq 1$) with maps the oriented maps.
Each object in $\mathcal{B}$ admits a unique oriented map to $C_0$ and there is a functor \[ F: \mathcal{B}^{\oprm} \to C_0 \downarrow (\iota: \mathcal{A}^{\oprm} \to \widetilde{\mathcal{A}}^{\oprm}) \]
taking a graph $G$ to the opposite of the oriented map $G \to C_0$.
One can check that the functor $F$ is initial, so the pointwise formula for right Kan extension (Theorem 1 of \cite[X.3]{maclane}) gives
\[
	(\iota_* X)_{C_0} \cong \lim_{C_0 \downarrow \iota} X \cong \lim_{\mathcal{B}^{\oprm}} X.
\]

It remains to show that $\lim_{\mathcal{B}^{\oprm}} X$ is isomorphic to $X_{L_0}$.
Let $p : L_0 \to L_1$ be defined by $p(0) = 1$ and $q: L_0 \to L_1$ be defined by $q(0) = 0$.
For $n\geq 1$ we have a diagram
\begin{equation} \label{eq Lzero to Ln} \begin{tikzcd}
X_{L_0} \rar \dar \arrow[dr] & X_{L_1} \arrow[dr, "\text{diagonal}"] \\
X_{L_n} \rar{\cong} & X_{L_1} {{}_{p^*}\kern-4pt\times\kern-3pt{}_{q^*}} X_{L_1} \cdots {{}_{p^*}\kern-4pt\times\kern-3pt{}_{q^*}} X_{L_1} \rar[hook] & X_{L_1}^{\times n}
\end{tikzcd} \end{equation}
coming from the unique oriented maps $L_1 \to L_0$ and $L_n \to L_0$ and the $n$ oriented embeddings $L_1 \to L_n$.
The bottom left map is an isomorphism by the Segal condition, that is, elements in $X_{L_n}$ are lists $(x_1, \dots, x_n)$ with $p^*x_j = q^* x_{j+1}$ for $1\leq j < n$.

There is no map in $\mathcal{B}$ from $C_m$ to $L_0$.
However, the Segal condition implies the oriented embedding $L_m \to C_m$ which is the identity on vertices induces an inclusion $X_{C_m} \to X_{L_m}$.
That is, $X_{C_m}$ may be regarded as the subset of $X_{L_m}$ consisting of those lists $(x_1, \dots, x_m)$ satisfying the additional condition $q^*x_1 = p^*x_m$.
The oriented rotation $r_m : C_m \to C_m$ acts on $X_{C_m}$ by rotating these lists.
We see that the map $X_{L_0} \to X_{L_m}$ from \eqref{eq Lzero to Ln}, which lands in the diagonal, actually factors through $X_{C_m}$; write $\kappa_m : X_{L_0} \to X_{C_m}$ for this special function not coming from $\mathcal{B}$.
As $\kappa_m$ lands in a diagonal, we have $r_m^* \kappa_m = \kappa_m$. 

One now checks that the special functions $\kappa_m : X_{L_0} \to X_{C_m}$ and the natural maps $X_{L_0} \to X_{L_n}$ determine a function $X_{L_0} \to \lim_{\mathcal{B}^{\oprm}} X$ which is both left and right inverse to the projection $\lim_{\mathcal{B}^{\oprm}} X \to X_{L_0}$.
This is tedious but straightforward.
\end{proof}

\subsection{Genus grading and stable maps}\label{section: stable graphs} 

The original definition \cite{MR1601666} of modular operad had an additional `genus' grading.
In this case, the underlying objects satisfy a stability condition.
One can certainly import these notions directly into the setting of colored modular operads studied in \cite{modular_paper_two}.
In this section, we discuss the presheaf side, and propose, in Theorem~\ref{theorem stable segal modular operads}, a stable version of the Segal modular operads of Definition~\ref{defn mon segal obj}.

\begin{definition}
Let $G$ be a graph.
\begin{itemize}
	\item A \emph{genus function} for $G$ is a function $g : V(G) \to \mathbb{N}$. 
	\item The \emph{total genus} of a pair $(G,g : V \to \mathbb{N})$ is given by
	\[
		g(G) = \beta_1(G) + \sum_{v\in V} g(v)
	\]
	where $\beta_1(G)$ is the first Betti number of $G$.
	More generally, if $f : H \to G$ is an embedding, then we can define
	\[
		g(f) = \beta_1(H) + \sum_{v\in V(H)} g(f(v))
	\]
	which descends to a function $g : \embeddings(G) \to \mathbb{N}$\index{$g : \embeddings(G) \to \mathbb{N}$}.
	\item A pair $(G,g)$ is called \emph{stable} if $G$ is connected and for every vertex $v$, 
	\[
		2g(v) + |\nbhd(v)| - 2 > 0.
	\]
\end{itemize}
\end{definition}

If $G\neq\exceptionaledge$, then the first Betti number of $G$ is given by $\beta_1(G)=|E_i|-|V|+1$.
Using this fact, or the long exact sequence for relative homology, one sees that $\beta_1(G\{H_v\}) = \beta_1(G) + \sum_v \beta_1(H_v)$ (which should be proved working one vertex at a time) whenever $G$ and all of the $H_v$ are connected.

\begin{itemize}
\item The exceptional edge admits only one genus function $g$, and $g(\exceptionaledge) = \beta_1(\exceptionaledge) = 0$.
This graph trivially satisfies the stability condition.
\item Note that if $(G,g)$ is a stable graph, then $G$ has no bivalent vertices with genus $0$. 
Moreover, if $g(v)=1$, then $|\nbhd(v)|>0$. 
\item The function $g : \embeddings(G) \to \mathbb{N}$ sends $\iota_v : \medstar_v \to G$ to $g(v)$ and $\id_G : G \to G$ to the total genus $g(G)$.
\end{itemize}

Suppose that $\varphi : H \to G$ is any graphical map and $g$ is a genus function on $G$.
Then the composition 
\[
	g_{\varphi}: V(H) \xrightarrow{\varphi_1} \embeddings(G) \xrightarrow{g} \mathbb{N}
\]
is a genus function for $H$.
If $(G,g)$ happens to be stable, it is not necessarily true that $(H,g_\varphi)$ is also stable.
However, if $\varphi$ is an embedding then $(H, g_{\varphi})$ is stable since stability is just checked at each vertex.

\begin{example}
\label{never edges stable}
Suppose that $(G,g)$ is stable and $\varphi : H \to G$ is a graphical map.
If there is a vertex $v$ with $\varphi_1(v)$ an edge, then $(H,g_{\varphi})$ is not stable.
This is because $g_\varphi(v) = g(\exceptionaledge \to G) = 0$, so $2g_\varphi(v) + |\nbhd(v)| - 2 = |\nbhd(v)| - 2 = 0$.
\end{example}

\begin{definition}[Stable graphical category]
The stable graphical category $\graphicalcat_{\strm}$\index{$\graphicalcat_{\strm}, R: \graphicalcat_{\strm} \to \graphicalcat$} has:
\begin{itemize}
\item Objects those pairs $(G,g)$ where $G\in \graphicalcat$ is a graph and $g$ is a genus function so that $(G,g)$ is stable.
\item Morphisms $(G,g) \to (G',g')$ are precisely those graphical maps $\varphi : G \to G'$ so that the diagram
\[\tcdtriangle{V(G)}{g}{\varphi_1}{\embeddings(G')}{g'}{\mathbb{N}}\]
commutes. One has such a morphism just when $g = g'_\varphi$.
\end{itemize}
One defines composition using the composition in $\graphicalcat$.
We let $R: \graphicalcat_{\strm} \to \graphicalcat$ be the functor which forgets the genus function.
\end{definition} 

\begin{proposition}
The morphisms defined above for $\graphicalcat_{\strm}$ are closed under composition.
\end{proposition}
\begin{proof}
Consider two morphisms $\varphi : G^{(1)} \to G^{(2)}$ and $\psi : G^{(2)} \to G^{(3)}$ that define maps of stable graphs
\[
	 (G^{(1)}, g^{(1)}) \xrightarrow{\varphi} (G^{(2)}, g^{(2)}) \xrightarrow{\psi} (G^{(3)}, g^{(3)}).
\]
We wish to show that $\psi \circ \varphi$ is in $\graphicalcat_{\strm}$.
In other words, assuming that $g^{(1)} = g^{(2)}_{\varphi}$ and $g^{(2)} = g^{(3)}_{\psi}$, we want to show that $g^{(1)} = g^{(3)}_{\psi\circ\varphi}$. 
Let us compute: since $g^{(1)} = g^{(2)}_{\varphi}$, we have for each vertex $v$ of $G^{(1)}$ that $g^{(1)}(v) = g^{(2)}(\varphi_v : H_v \to G^{(2)})$.
Writing $\psi_{v,w} : K_{v,w} \hookrightarrow G^{(3)}$ for a representative of $\psi_1(\varphi_v(w))$ and using $g^{(2)} = g^{(3)}_{\psi}$, we have
\begin{align*}
	g^{(1)}(v) = g^{(2)}(\varphi_v) &= \beta_1(H_v) + \sum_{w\in H_v} g^{(2)} (\varphi_v(w)) \\
	& = \beta_1(H_v) + \sum_{w\in H_v} g^{(3)} (\psi_1(\varphi_v(w))) \\
	& = \beta_1(H_v) + \sum_{w\in H_v} g^{(3)} (\psi_{v,w} : K_{v,w} \to G^{(3)})
\end{align*}
The summand for a given $w$ is $\beta_1(K_{v,w}) + \sum_{u\in K_{v,w}} g^{(3)}(\psi_{v,w}(u))$, so rearranging we have
\begin{align*}
g^{(1)}(v) &= \beta_1(H_v) + \sum_{w\in H_v} \beta_1(K_{v,w}) + \sum_{w\in H_v} \sum_{u\in K_{v,w}} g^{(3)}(\psi_{v,w}(u)) \\
&= \beta_1(H_v\{K_{v,w}\}) + \sum_{w\in H_v} \sum_{u\in K_{v,w}} g^{(3)}(\psi_{v,w}(u)),
\end{align*}
which is exactly
\[
	g^{(3)} (\image(\psi|_{\varphi_v}) : H_v\{K_{v,w}\} \to G^{(3)}).
\]
Thus $g^{(1)}(v) = g^{(3)}((\psi\circ\varphi)_1(v)),$ so $g^{(1)} = g^{(3)}_{\psi\circ\varphi}$. 
\end{proof}

\begin{remark}
Example~\ref{never edges stable} tells us that the functor $R: \graphicalcat_{\strm} \to \graphicalcat$ factors through $\graphicalcat^+$.
In particular, \eqref{graphical map defn: collapse condition} is \emph{automatic} for a map between stable graphs, and if we were starting from scratch in this section we would omit this condition entirely.
In any case, combining $R$ with the degree function from Definition~\ref{defn graph cat reedy} yields a generalized Reedy structure on $\graphicalcat_{\strm}$ where $\graphicalcat_{\strm}^+ = \graphicalcat_{\strm}$ and $\graphicalcat_{\strm}^- = \Iso(\graphicalcat_{\strm})$.
Further, there is an orthogonal factorization system $(R^{-1}\graphicalcat_{\actrm}, R^{-1}\graphicalcat_{\embrm})$ on $\graphicalcat_{\strm}$ coming from Theorem~\ref{theorem orthogonal factorization system}.
\end{remark}

If $(G,g)$ is a stable graph, then one can define the Segal core $\segalcore_{\strm}[G,g]$ as a subobject of the representable presheaf $\rgc_{\strm}[G,g]$ just as we did in the previous section.
Imitating the other proofs from Section~\ref{section simplicial presheaves on U} yields the following analogue of Theorem~\ref{theorem localized reedy model structure}.

\begin{theorem}
\label{theorem stable segal modular operads}
The category $\sSet^{\graphicalcat_{\strm}^{\oprm}}$ admits a cofibrantly generated model structure whose fibrant objects are are those presheaves $X$ satisfying the following three conditions:
\begin{itemize}
	\item $X$ is Reedy fibrant,
	\item $X_{\exceptionaledge}$ is weakly contractible, and
	\item for each $(G,g)\in \graphicalcat_{\strm}$, the Segal map
\[ X_{G,g} = \map(\rgc_{\strm}[G,g], X) \rightarrow \map(\segalcore_{\strm}[G,g],X)\] 
is a weak equivalence of simplicial sets.\qed 
\end{itemize} 
\end{theorem}

As in the unstable setting, there is a nerve functor landing in $\graphicalcat_{\strm}$-presheaves.
The analogue of Theorem~\ref{nerve theorem paper one intro} also holds in the stable setting.
This produces a large collection of fibrant objects in the model structure from Theorem~\ref{theorem stable segal modular operads}, in a similar manner to Remark~\ref{remark on nerve of modular operads}.
We now propose an example, related to surfaces, that is not of this form.
Our example is written by considering $\graphicalcat_{\strm}$-presheaves in the category of groupoids.
It can be transferred to simplicial sets by using the classifying space functor that goes from the category of groupoids to the category of simplicial sets.

\begin{example}\label{example tillmann}
We consider compact, orientable surfaces $S$ where the set of boundary components is given an ordering and each boundary component is equipped with a specified collar.
Define a collection of groupoids $\{\mathcal{S}_{g,n}\}$ where $\mathcal{S}_{g,n}$ has objects $S$ where $S$ is a surface of genus $g$ with $n$ boundary components constructed by gluing atomic surfaces as in \cite[2.2]{Tillmann:HigherGenus}.
Morphisms in $\mathcal{S}_{g,n}$ are given by isotopy classes of homeomorphisms which fix the boundary components pointwise and preserves the orderings (modulo the identifications imposed in \cite{Tillmann:HigherGenus}). 
Notice that the automorphism group of an object $S$ in $\mathcal{S}_{g,n}$ is the mapping class group $\Gamma_{g,n}$.
Tillmann \cite[2.1]{Tillmann:HigherGenus}, and later Giansiracusa--Salvatore \cite{GiansiracusaSalvatore:Formality}, show that by gluing along boundaries the collection $\{\mathcal{S}_{g,n}\}$ constitutes a modular operad. 
In forthcoming work of the second author, there is a need to not just understand how surfaces are built up from atomic pieces, but how these atomic surfaces are assembled. 
She will show that there is a groupoid-valued $\graphicalcat_{\strm}$-presheaf $X$ which satisfies a weak Segal condition and is related to the nerve of the above modular operad.
If $(G,g)$ is stable, then an object of $X_{G,g}$ consists of one surface for each vertex of $G$ as well as gluing data for the collars connected by the edges.
\end{example}

\section{Simply-connected graphs}\label{section simply connected} 

We now introduce two full subcategories of $\graphicalcat$ which are related to notions of \emph{cyclic operad} \cite{GK98}; see \cite[\TWOcyclicnerve]{modular_paper_two}. 
The objects of these subcategories are unrooted trees. 

\begin{definition}
	Denote by $\graphicalcat_{\cycrm} \subseteq \graphicalcat_0 \subseteq \graphicalcat$\index{$\graphicalcat_{\cycrm}, \graphicalcat_0$} the following full subcategories:
	\begin{itemize}
		\item $\ob(\graphicalcat_0)$ is the set of connected \emph{acyclic} graphs.
		\item Graphs in $\graphicalcat_{\cycrm}$ additionally have \emph{nonempty boundary}.
	\end{itemize}
\end{definition}

In the language of \cite{hinich-vaintrob}, $\graphicalcat_{\cycrm}$ is related to cyclic operads while $\graphicalcat_0$ is related to \emph{augmented} cyclic operads.

\begin{proposition}\label{proposition sieve}
	The full subcategories $\graphicalcat_0$ and $\graphicalcat_{\cycrm}$ are \emph{sieves} (in the sense of \cite[Definition 6.2.2.1]{htt}) on $\graphicalcat$. That is, if $\varphi : G \to T$ is a morphism in $\graphicalcat$ with $T\in \graphicalcat_0$, then $G\in \graphicalcat_0$ (and similarly for $\graphicalcat_{\cycrm}$).
\end{proposition}
\begin{proof}
As in Proposition~\ref{proposition: image}, the map $\varphi$ factors as
\[
\begin{tikzcd}[column sep=small]
G \arrow[rr,"\varphi"] \arrow[dr] & & T \\
& G\{H_v\} \arrow[ur, hook, "k" swap]
\end{tikzcd}
\]
with $k$ an embedding.
Embeddings into simply-connected graphs must have simply-connected sources, hence $G\{H_v\} \in \graphicalcat_0$.
On the other hand, any loop in $G$ (that is, a path with no repeated entries except for the ends) may be extended to a loop in $G\{H_v\}$ since each $H_v$ is connected.
Since such a loop cannot exist in $G\{H_v\}$, we see that there was no loop in $G$. 

For the second statement, note that any map $\varphi : G \to G'$ in $\graphicalcat$ with $\eth(G) = \varnothing$ is automatically active, which implies that $\eth(G') = \varnothing$.
This implies that $\graphicalcat_{\cycrm}$ is a sieve on $\graphicalcat_0$, hence on $\graphicalcat$.
\end{proof}

\begin{remark}\label{remark comparison sc}
The category $\graphicalcat_{\cycrm}$ is related to other categories in the literature.
\begin{enumerate}
\item Tashi Walde's category $\Omega_{\cycrm}$ from \cite{Walde:2SSIIO} can be considered as a nonsymmetric version of $\graphicalcat_{\cycrm}$.
To be precise, $\graphicalcat_{\cycrm}$ is equivalent to a category $\graphicalcat_{\cycrm}'$ whose objects have cyclic orderings on each set $\nbhd(v)$.
Then $\Omega_{\cycrm}$ is the wide subcategory of $\graphicalcat_{\cycrm}'$ consisting of maps which preserve the cyclic orderings. 
\item \label{remark comparison sc item two}
In \cite{hry_cyclic}, the authors developed a category $\Xi$ with objects the unrooted trees with nonempty boundary.
Let $\graphicalcat_{\cycrm}''$ denote the category (equivalent to $\graphicalcat_{\cycrm}$) where the boundary and each $\nbhd(v)$ comes equipped with a total ordering and where morphisms can disregard those orderings.
There is a functor $\graphicalcat_{\cycrm}'' \to \Xi$ which is the identity on objects. For a nonlinear tree $T$, we have
\[
	\graphicalcat_{\cycrm}''(T,S) = \Xi(T,S),
\]
while for linear trees $T$ this map is just surjective.
\end{enumerate}
\end{remark}

\begin{example}\label{example difference between tree cats}
The primary difference between $\graphicalcat_{\cycrm}''$ and $\Xi$ is that maps in the latter category are based on edges rather than on arcs.
If $L_n$ is linear and $S$ is an arbitrary unrooted tree, then the sets of morphisms $L_n\to S$ which do not factor through $\exceptionaledge$ are the same in both categories $\graphicalcat_{\cycrm}''$ and $\Xi$.
On the other hand, the set of morphisms $L_n \to S$ which factor through $\exceptionaledge$ is in bijection with $A(S)$ in $\graphicalcat_{\cycrm}''$, while in $\Xi$ it is in bijection with $E(S)$.
In particular, we have
\begin{align*}
\graphicalcat_{\cycrm}''({\exceptionaledge}, S) &\cong A(S) & \Xi({\exceptionaledge}, S) &\cong E(S) \\
|\graphicalcat_{\cycrm}''(L_n, {\exceptionaledge})| & = 2 &|\Xi(L_n, {\exceptionaledge})| & = 1. 
\end{align*}
\end{example}

\begin{proposition}
The categories $\graphicalcat_{\cycrm}$ and $\graphicalcat_0$ are both dualizable generalized Reedy categories, with structure induced from the ambient category $\graphicalcat$.
\end{proposition}
\begin{proof}
This follows from the fact that these are sieves on $\graphicalcat$ (Proposition~\ref{proposition sieve}), the fact that $\graphicalcat$ is Reedy (Theorem~\ref{theorem reedy}), and the fact that sieves in a Reedy category are again Reedy categories (Lemma~\ref{lemma sieve} below). 
\end{proof}

\begin{lemma}\label{lemma sieve}
Suppose that $\Rr$ is a (dualizable) generalized Reedy category, and $\Ss \subseteq \Rr$ is a full subcategory.
If $\Ss$ is a sieve on $\Rr$, then $\Ss$ is also a (dualizable) generalized Reedy category with $\Ss^\pm = \Rr^\pm \cap \Ss$ and $\deg_{\Ss} = \deg_{\Rr}|_{\ob(\Ss)}$.
\end{lemma}
\begin{proof}
Most of the axioms are immediate.
The only place we use the sieve property is in verifying that the factorizations in \cite[Definition 1.1(iii)]{bm_reedy} for $\Rr$ are actually factorizations in $\Ss$:
given a diagram 
\[
\begin{tikzcd}[column sep=small]
s \arrow[rr,"f"] \arrow[dr, "f^-" swap] & & s' \\
& r \arrow[ur, "f^+" swap]
\end{tikzcd}
\]
in $\Rr$ with $s,s'\in \Ss$, the existence of the arrow $f^+ : r \to s'$ guarantees that $r$ is an object of $\Ss$ also.
\end{proof}

\begin{remark}[Reedy structure on $\Xi$]
In Section 3 of \cite{hry_cyclic}, it was shown that $\Xi$ carries the following generalized Reedy structure:
\begin{itemize}
\item Morphisms in $\Xi^+$ are those which are injective on edges.
\item Morphisms in $\Xi^-$ are those which are surjective on edges and active.
\item The degree of an unrooted tree is given by the number of vertices.
\end{itemize}
Alternatively, one may describe $\Xi^+$ as the subcategory of monomorphisms and $\Xi^-$ as the subcategory of split epimorphisms \cite[Theorem 4.9]{hry_cyclic}.
\end{remark}

\subsection{Comparison of Reedy model structures}

The functor $J : \graphicalcat_{\cycrm}'' \to \Xi$ from Remark~\ref{remark comparison sc}\eqref{remark comparison sc item two} preserves the Reedy factorization system, but it does not commute with the degree function. 
However, we have the following:
\begin{theorem}
\label{reedy theorem xi}
The functor
\[
	J^* : \sSet^{\Xi^{\oprm}} \to \sSet^{(\graphicalcat_{\cycrm}'')^{\oprm}}
\]
preserves and reflects fibrations and weak equivalences.
It detects cofibrations but it does not preserve them. 
\end{theorem}
To avoid clutter, we will omit the $''$ from $\graphicalcat_{\cycrm}''$ and just write $\graphicalcat_{\cycrm}$ for the remainder of the paper.

We first show that the matching objects for $X$ and for $J^*X$ coincide.
\begin{lemma}\label{lemma for matching object sc}
The functor
\[
	F: \graphicalcat_{\cycrm}^+(T) \to \Xi^+(T)
\]
is an equivalence of categories.
Therefore, if $T \in \sSet^{\Xi^{\oprm}}$, then the $T$-th matching map $X_T \to M_TX$ is isomorphic to the $T$-th matching map $(J^*X)_T \to M_T(J^*X)$.
\end{lemma}
\begin{proof}
The functor $F$ is surjective on objects.
We would like to show that $F$ is fully-faithful, which will imply that $F$ is an equivalence of categories.
Then so is 
\[  
(\graphicalcat_{\cycrm}^{\oprm})^-(T) \cong [\graphicalcat_{\cycrm}^+(T)]^{\oprm} \xrightarrow{F^{\oprm}} [\Xi^+(T)]^{\oprm} \cong (\Xi^{\oprm})^-(T)
\]
and the coincidence of the matching maps (see Definition~\ref{definition reedy model structure}) follows.

Suppose that $\alpha : G \to T$ and $\beta : H \to T$ are two objects of $\graphicalcat_{\cycrm}^+(T)$, that is, $\alpha, \beta \in \graphicalcat_{\cycrm}^+ \setminus \Iso(\graphicalcat_{\cycrm}^+)$.
Since $T$ is simply-connected, $\beta$ must be a monomorphism on arcs (this essentially follow from Lemma~\ref{lemma embedding not monomorphism} after factoring $\beta$ as in Proposition~\ref{prop_factors_image_followed_by_outer_coface}).
If $\gamma, \gamma' \in \graphicalcat_{\cycrm}^+(G,H)$ are two morphisms from $\alpha$ to $\beta$, that is, if $\beta \gamma = \alpha = \beta \gamma'$, then we have $\gamma_0 = \gamma'_0$.
For each vertex $v$ of $G$, we know that $\gamma_1(v)$ and $\gamma'_1(v)$ are not edges.
Thus, by Proposition~\ref{proposition: embedding uniqueness}, $\gamma_1(v) = \gamma'_1(v)$ since they have the same boundaries.
We've shown that $\hom(\alpha, \beta)$ has at most one element, so 
\begin{equation}\label{map to show iso}
\hom(\alpha, \beta) \to \hom(F\alpha, F\beta)
\end{equation} 
is injective.

Now suppose we have a map $F\alpha \to F\beta$ in $\Xi^+(T)$, that is, suppose we have a commutative diagram 
\[ \begin{tikzcd}[column sep=small]
G \arrow[rr, "\gamma"]  \arrow[dr, "F\alpha" swap] & & H \arrow[dl, "F\beta"] \\
& T
\end{tikzcd} \]
in $\Xi^+$.
We have a commutative square
\begin{equation}\label{map in xi plus T}
\begin{tikzcd}
\graphicalcat_{\cycrm}^+(H,T) \times \graphicalcat_{\cycrm}^+(G,H) \rar{\circ} \dar & \graphicalcat_{\cycrm}^+(G,T) \dar \\
\Xi^+(H,T) \times \Xi^+(G,H) \rar{\circ} & \Xi^+(G,T) 
\end{tikzcd} 
\end{equation}
whose vertical maps are isomorphisms as long as $G$ has at least one vertex.
So if $G$ has at least one vertex we know there is a $\tilde \gamma$ so that $\beta \tilde \gamma = \alpha$ and $J(\tilde \gamma) = \gamma$.
Hence, in this case, \eqref{map to show iso} is surjective.
On the other hand, if $G$ is the exceptional edge, then $\gamma$ hits a single edge $[x,ix]$ in $H$.
Since \eqref{map in xi plus T} commutes, we have (using the notation for the arcs of $\exceptionaledge$ from Example~\ref{examples combinatorial})
\[
	[\alpha_0(\edgemajor), \alpha_0(\edgeminor)] = (F\alpha)_0 [\edgemajor, \edgeminor] = (F\beta)_0 \gamma_0 [\edgemajor, \edgeminor] = (F\beta)_0 [x,ix] = [\beta_0(x), \beta_0(ix)].
\]
If $\alpha_0(\edgemajor) = \beta_0(x)$, define $\tilde \gamma$ by $\tilde \gamma(\edgemajor) = x$, while if $\alpha_0(\edgemajor) = \beta_0(ix)$, let $\tilde \gamma(\edgemajor) = ix$. 
We've thus established that \eqref{map to show iso} is also surjective when $G$ is the exceptional edge.
\end{proof}

We also have the following lemma.
\begin{lemma}\label{lemma for latching object sc}
The functor
\[
	F: \graphicalcat_{\cycrm}^-(T) \to \Xi^-(T)
\]
is an equivalence of categories.
Therefore, if $T \in \sSet^{\Xi^{\oprm}}$, then the $T$-th latching map $L_TX \to X_T$ is isomorphic to the $T$-th latching map $L_T(J^*X) \to (J^*X)_T$.
\end{lemma}
\begin{proof}
For the most part, the proof follows that of Lemma~\ref{lemma for matching object sc} with strictly formal changes.
The one exception is in the second paragraph, where we applied Proposition~\ref{proposition: embedding uniqueness} to show that if the two maps $\gamma, \gamma'$ are the same on arcs, then they are the same.
In the present situation, this follows from Proposition~\ref{proposition valence not two} using that all maps in $\graphicalcat_{\cycrm}^-$ are active and that all objects in $\graphicalcat_{\cycrm}$ have nonempty boundary.
\end{proof}

In light of the preceding lemma, it may seem strange that Theorem~\ref{reedy theorem xi} asserts that $J^*$ does not preserve cofibrations.
After all, $X$ and $J^*X$ have the same latching maps!
The following proposition addresses the underlying reason, while Remark~\ref{what the left kan extension does} allows us to produce concrete examples of cofibrations which are not preserved by $J^*$.

\begin{proposition}
\label{prop cofibration behavior}
Suppose that $f: X \to Y$ is a morphism in $\sSet^{\Xi^{\oprm}}$, with $J^*f : J^*X \to J^*Y$ its image in $\sSet^{\graphicalcat_{\cycrm}^{\oprm}}$.
\begin{itemize}
	\item 
	Suppose that $f$ is a Reedy cofibration.
	Then $J^*f$ is a Reedy cofibration if and only if $f$ is an isomorphism when evaluated at $\exceptionaledge$.
	\item
	If $J^*f$ is a Reedy cofibration, then $f$ is a Reedy cofibration.
\end{itemize}
\end{proposition}
\begin{proof}
If $T$ contains at least one vertex, then $J$ induces an identity between the two automorphism groups $\aut_{\graphicalcat}(T)$ and $\aut_{\Xi}(T)$ of $T$.
By Lemma~\ref{lemma for latching object sc}, it follows that the relative latching map 
\[
	X_T \cup_{L_TX} L_TY \to Y_T
\]
is a cofibration in $\sSet^{\aut_{\Xi}(T)^{\oprm}}$ if and only if 
\[
	(J^*X)_T \cup_{L_T(J^*X)} L_T(J^*Y) \to (J^*Y)_T
\]
is a cofibration in $\sSet^{\aut_{\graphicalcat}(T)^{\oprm}}$.

On the other hand, if $T$ is the exceptional edge, then $\aut_{\graphicalcat}(\exceptionaledge)^{\oprm}$ is the cyclic group of order two, $C_2$, while $\aut_{\Xi}(\exceptionaledge)^{\oprm}$ is the trivial group (see Example~\ref{example difference between tree cats}). 
Further, $L_\exceptionaledge Z = \varnothing$.
Thus the relative latching map just becomes $X_\exceptionaledge \to Y_\exceptionaledge$.
Now $\aut_{\graphicalcat}(\exceptionaledge)^{\oprm}$ acts trivially on $(J^*X)_{\exceptionaledge}$ and $(J^*Y)_{\exceptionaledge}$.
So if $J^*(f)$ is a Reedy cofibration, then $(J^*X)_{\exceptionaledge} \to (J^*Y)_{\exceptionaledge}$ is a cofibration in $\sSet^{C_2}$, hence in $\sSet$.
By the previous paragraph, we know that $f$ is then a Reedy cofibration.
If $f$ is a Reedy cofibration, then $J^*(f)$ is a Reedy cofibration if and only if $X_{\exceptionaledge} \to Y_{\exceptionaledge}$ is a cofibration in $\sSet^{C_2}$.
Since both sides have trivial $C_2$ actions, this happens if and only if $X_{\exceptionaledge} \to Y_{\exceptionaledge}$ is an isomorphism (Lemma 2.4 of \cite[Chapter V]{GoerssJardine:SHT}).
\end{proof}

\begin{remark}
\label{what the left kan extension does}
Suppose that $Z$ is an object of $\sSet^{\graphicalcat_{\cycrm}^{\oprm}}$ and let $J_!Z$ in $\sSet^{\Xi^{\oprm}}$ be its left Kan extension along $J^{\oprm} : \graphicalcat_{\cycrm}^{\oprm} \to \Xi^{\oprm}$.
The following three sets of morphisms coincide
\[
	\hom(Z, \varnothing) = \hom(Z, J^*\varnothing) \cong \hom(J_!Z, \varnothing),
\]
so $Z$ is nonempty if and only if $J_!Z$ is nonempty.
\end{remark}

\begin{proof}[Proof of Theorem~\ref{reedy theorem xi}]
As Reedy weak equivalences are levelwise and $J$ is the identity on objects, a map $f$ in  $\sSet^{\Xi^{\oprm}}$ is a weak equivalence if and only if $J^*f$ is a weak equivalence.
By Lemma~\ref{lemma for matching object sc}, $f$ is a fibration if and only if $J^*f$ is a fibration.
The statement about detecting cofibrations follows from Proposition~\ref{prop cofibration behavior}.

Finally, let us show that $J^*$ does not preserve cofibrations.
Since $J_!$ is left Quillen, it preserves cofibrations.
Both $J^*$ and $J_!$ are left adjoints, and, hence, they preserve initial objects. 
Suppose that $Z$ is any cofibrant object in $\sSet^{\graphicalcat_{\cycrm}^{\oprm}}$ other than the initial object.
Then $J_!Z$ is nonempty by Remark~\ref{what the left kan extension does} and is also cofibrant.
We will show that $J^*J_!Z$ is not cofibrant.

Let us observe that $(J^*J_!Z)_\exceptionaledge$ is nonempty.
We know that $(J^*J_!Z)_T = (J_!Z)_T$ is nonempty for some tree $T$.
The edge set of $T$ is not empty, so there is at least one map $\exceptionaledge \to T$, which gives a function $\varnothing \neq (J^*J_!Z)_T \to (J^*J_!Z)_\exceptionaledge$.
Now $C_2$ acts trivially on $(J^*J_!Z)_{\exceptionaledge} \neq \varnothing$, and we conclude that $J^*J_!Z$ is not cofibrant.
\end{proof}

\begin{remark}
Barwick, in \cite[Theorem 3.22]{Barwick:OLRMCLRBL}, gave a characterization of those functors $F$ between strict Reedy categories so that restriction $F^*$ is left or right Quillen for every model category $\modelcat$; see \cite{hv15} for an alternate presentation of this theorem.
It follows from Theorem~\ref{reedy theorem xi} that the na\"ive generalization of this characterization does not hold for functors between generalized Reedy categories, as $J^{\oprm}:\graphicalcat_{\cycrm}^{\oprm} \to \Xi^{\oprm}$ satisfies a strong analogue of the appropriate condition (called `cofibering' in \cite{hv15}) that would imply $J^*$ is left Quillen (which is false). 
For simplicity of notation, we study whether or not $J : \graphicalcat_{\cycrm} \to \Xi$ is some kind of `fibering' functor using \cite[Proposition 8, p.~32]{hv15}.

Suppose that $\sigma : S \to T$ is in $\Xi^-$,
and write $\mathbf{C}_\sigma$ for the category that would be the evident analogue of $\mathrm{Fact}_{\Xi^-}(S,\sigma)$ in \cite[Definition 7, p.~27]{hv15} and $\partial(S / (J^\leftarrow / T))$ in \cite[Theorem 3.22]{Barwick:OLRMCLRBL}.
Concretely, the category $\mathbf{C}_\sigma$ has as objects those pairs $(\nu, \mu)$ where $\nu\in \graphicalcat_{\cycrm}^- \setminus \Iso(\graphicalcat_{\cycrm})$, $\mu \in \Xi^-$, and $\mu J(\nu) = \sigma$. 
A morphism $(\nu, \mu) \to (\nu', \mu')$ consists of a morphism $\tau \in \graphicalcat_{\cycrm}^-$ making the diagrams 
\begin{equation}\label{two triangles}
\begin{tikzcd}[column sep=small]
& S \arrow[dl,"\nu" swap] \arrow[dr,"\nu'"]& & \bullet \arrow[rr,"J\tau"] \arrow[dr, "\mu"] & & \bullet \arrow[dl, "\mu'" swap] \\
\bullet \arrow[rr, "\tau"] & & \bullet & & T
\end{tikzcd} \end{equation}
commute.
If $\sigma$ is an isomorphism, then $\mathbf{C}_\sigma$ is empty.
If $T$ contains a vertex, then $\graphicalcat_{\cycrm}^-(S,T) = \Xi^-(S,T)$, so there is a unique lift $\tilde \sigma \in \graphicalcat_{\cycrm}^-$ of $\sigma$ and the object $(\tilde \sigma, \id_T)$ is a terminal object of $\mathbf{C}_\sigma$.
If $T$ is isomorphic to the exceptional edge, then $S$ must be isomorphic to a linear graph $L_n$; we suppose that $\sigma$ is not an isomorphism.
Let $\tilde \sigma \in \graphicalcat_{\cycrm}^-$ be either of the two lifts of $\sigma$.
Then $(\tilde \sigma, \id_T)$ is again a terminal object of $\mathbf{C}_\sigma$.
Indeed, if $\nu : S \to R$ is any map in $\graphicalcat_{\cycrm}$, then there is a unique map $\tau : R \to T$ so that $\tau \nu = \tilde \sigma$.
Further, any diagram that looks like the right hand triangle of \eqref{two triangles} commutes when $T \cong \exceptionaledge$.

Thus we have shown that $\mathbf{C}_\sigma$ is either empty or \emph{contractible} for any $\sigma \in \Xi^-$. 
In the strict case, the conditions of Barwick and Hirschhorn--Voli\'c simply request that each $\mathbf{C}_\sigma$ is either empty or connected in order to infer that 
\[
	J^* : \sSet^{\Xi^{\oprm}} \to \sSet^{\graphicalcat_{\cycrm}^{\oprm}}
\]
is left Quillen.
We presume that there is a much simpler counterexample to $F^*$ being left Quillen than this one.
Further, we anticipate that the characterization for when $F^*$ is right Quillen should be similar to that in the strict case.
\end{remark}

\bibliographystyle{amsalpha}
\bibliography{modular}

\end{document}